\documentclass{compositio}

\usepackage[latin1]{inputenc}
\usepackage[all]{xy}
\usepackage{graphics}
\usepackage{amsmath}
\usepackage{amscd}
\usepackage{amssymb}
\usepackage{amsthm}
\usepackage{enumitem}
\usepackage{comment}
\renewcommand{\contentsname}{Contents\\{\footnotesize\normalfont(A table
of contents should normally not be included)}}

\newtheorem{thm}{Theorem}[section]
\newtheorem{hyp}[thm]{Hypothesis}
\newtheorem{pro}[thm]{Proposition}
\newtheorem{lem}[thm]{Lemma}
\newtheorem{cor}[thm]{Corollary}
\newtheorem{defn}[thm]{Definition}

\newtheorem*{rem*}{Remarks}
\newtheorem{rems}[thm]{Remark}
\newtheorem*{conj*}{Conjecture}

\newcommand{\lra}{\longrightarrow}

\newcommand{\mrm}{\mathrm}
\newcommand{\mbb}{\mathbb}

\newcommand{\g}{\mathrm}

\begin{document}

\title{On higher regulators of Siegel threefolds II: the connection to the special value}
\author{Francesco Lemma}
\email{francesco.lemma@imj-prg.fr}
\address{Institut math\'ematique de Jussieu-Paris Rive Gauche, UMR 7586, B\^atiment Sophie Germain, Case 7012, 75205 Paris Cedex 13}

\thanks{}

\begin{abstract}
We establish a connection between motivic cohomology classes over the Siegel threefold and non-critical special values of the degree four $L$-function of some cuspidal automorphic representations of $\mathrm{GSp}(4)$. Our computation relies on our previous work \cite{lemma} and on an integral representation of the $L$-function due to Piatetski-Shapiro.
\end{abstract}

\maketitle

\section{Introduction}

The analytic class number formula of Dedekind and Dirichlet, proved in the middle of the nineteenth century, is a mysterious relationship between an analytic invariant and an arithmetic invariant of a number field $K$. It relates the leading Taylor coefficient at $0$ of the zeta function of $K$ to the units $\mathcal{O}_K^\times$ of $K$ through the regulator map
$$
\begin{CD}
\mathcal{O}_K^\times @>>> \mbb{R}^{r_1+r_2}
\end{CD}
$$
where $r_1$ and $r_2$ denote the number of real and complex places of $K$ respectively. According to Beilinson's conjectures on special values of motivic $L$-functions \cite{beilinson1}, such relationships should exist in great generality. To generalize the analytic class number formula, Beilinson replaces units by motivic cohomology classes and the classical regulator by the higher regulator
$$
\begin{CD}
H^i_\mathcal{M}(X, \mbb{Q}(n)) @>>>H^i_\mathcal{H}(X/\mbb{R}, \mbb{R}(n))
\end{CD}
$$
from motivic cohomology to absolute Hodge cohomology. Here $X$ denotes a smooth projective scheme over $\mbb{Q}$. For an introduction to Beilinson's conjectures, to the motivic formalism underlying them, and for a survey of known results, the interested reader might consult the article \cite{nekovar}. In this paper, we establish a connection between elements in the motivic cohomology of the Siegel threefold and special values of the degree four $L$-function of some cuspidal automorphic representations of the symplectic group $\mathrm{GSp}(4)$. The author hopes that the present work may not be useless to discover more general phenomena explaining the connection between special values of automorphic $L$-functions and mixed motives.\\


Like in the previous approaches, the motivic cohomology classes that we work with are constructed using Beilinson's Eisenstein symbol and the functorial properties of motivic cohomology. For any non-negative integer $n$ and level $K$, the Eisenstein symbol (\cite{beilinson} \S 3) is a $\mbb{Q}$-linear map
\begin{equation} \label{symbole-eisenstein}
\begin{CD}
Eis^n_\mathcal{M}: \mathcal{B}_{n} @>>> H^{1}_\mathcal{M}(M_K, \mathrm{Sym}^n V_2(1))
\end{CD}
\end{equation}
whose target is the motivic cohomology of the Shimura variety $M_K$ of $\mrm{GL}(2)$, or modular curve, of level $K$, with coefficients in the $n$-th symmetric power of the motivic sheaf $V_2$ associated to the standard representation of $\mrm{GL}(2)$. The definition of $\mathcal{B}_n$ will be given later, but we would like to mention that for $\phi_f \in \mathcal{B}_n$, the image of $Eis_\mathcal{M}^n(\phi_f)$ under the regulator
$$
\begin{CD}
r_\mathcal{D}: H^{1}_\mathcal{M}(M_K, \mathrm{Sym}^n V_2(1)) @>>> H^{1}_\mathcal{D}(M_K/\mbb{R}, \mathrm{Sym}^n V_2(1))
\end{CD}
$$
in Deligne-Beilinson cohomology can be explicitly described by real analytic Eisenstein series. At present, most of the results relating special values of $L$-functions to regulators rely on the Eisenstein symbol (\cite{beilinson}, \cite{deninger1}, \cite{deninger2}, \cite{kings}, \cite{kato}...).\\

To explain the construction of the motivic cohomology classes that we shall study, let us introduce some notation. Let $I_2$ be the identity matrix of size two and let $\psi$ be the symplectic form whose matrix is
$$
\psi=
\begin{pmatrix}
& I_2\\
-I_2& \\
\end{pmatrix}.
$$
The symplectic group $\mathrm{GSp}(4)$ is defined as
$$
\mrm{GSp}(4)=\{g \in \mathrm{GL}_{4 / \mathbb{Q}} \,|\, ^t\!g \psi g = \nu(g) \psi, \, \nu(g) \in \mathbb{G}_m\}.
$$
It is a reductive linear algebraic group over $\mbb{Q}$ and contains the group $\mathrm{GL}(2) \times_{\mathbb{G}_m} \mathrm{GL}(2)$, where the fiber product is over the determinant, via the embedding
$$
\begin{CD}
\mathrm{GL}(2) \times_{\mathbb{G}_m} \mathrm{GL}(2) @>\iota>> \mathrm{GSp}(4)\\
\end{CD}
$$
defined by
$$
\iota\left( \begin{pmatrix}
a & b\\
c & d\\
\end{pmatrix}, \begin{pmatrix}
a' & b'\\
c' & d'\\
\end{pmatrix} \right) =
\begin{pmatrix}
a & & b & \\
 & a' &  & b'\\
c &  & d & \\
 & c' &  & d'\\
\end{pmatrix}.
$$
For a fixed level $L$, the Shimura variety $S_L$ associated to $\mrm{GSp}(4)$ is a  smooth quasi-projective threefold defined over $\mbb{Q}$. Moreover, for any level $K$, the morphism $\iota$ induces a closed embedding
$$
\begin{CD}
M_K \times M_K @>\iota>> S_L
\end{CD}
$$
for some $L$. The basic idea is to map the external cup-product $Eis^p_\mathcal{M} \sqcup Eis^q_\mathcal{M}$ to the motivic cohomology of $S_L$ via this embedding. To be more precise, let  $p$ and $q$ be non-negative integers, and let $(Sym^p V_2 \boxtimes Sym^q V_2)(3)$ denote the irreducible algebraic representation $(Sym^p V_2 \boxtimes Sym^q V_2) \otimes \mrm{det}^{\otimes 3}$ of $\mathrm{GL}(2) \times_{\mathbb{G}_m} \mathrm{GL}(2)$, where $\mrm{det}$ is the determinant character. Let $W$ be an irreducible representation of $\mrm{GSp}(4)$ such that, as representations of $\mathrm{GL}(2) \times_{\mathbb{G}_m} \mathrm{GL}(2)$, we have
\begin{equation} \label{branching}
(Sym^p V_2 \boxtimes Sym^q V_2)(3) \subset \iota^*W.
\end{equation}
 In what follows, we shall take the liberty to denote by the same symbol $W$ the motivic sheaf on $S_L$ corresponding to $W$ (see section \ref{sect-relatives-motives}). Taking the external cup-product of $Eis^p_\mathcal{M}$ and $Eis^q_\mathcal{M}$, we obtain the map
$$
\begin{CD}
Eis^p_\mathcal{M} \sqcup Eis^q_\mathcal{M}: \mathcal{B}_{p} \otimes_\mbb{Q} \mathcal{B}_{q} @>>> H^2_\mathcal{M}(M_K \times M_K, (Sym^p V_2 \boxtimes Sym^q V_2)(2)).
\end{CD}
$$
Composing with the map induced by the inclusion $(Sym^p V_2 \boxtimes Sym^q V_2)(2) \subset \iota^* W(-1)$ and with the Gysin morphism corresponding to $\iota$, we obtain the map
\begin{equation} \label{eisM}
\begin{CD}
Eis^{p, q, W}_\mathcal{M}: \mathcal{B}_{p} \otimes_\mbb{Q} \mathcal{B}_{q} @>>> H^4_\mathcal{M}(S_L, W).
\end{CD}
\end{equation}
To go further, let us recall the main result of \cite{lemma}, which will be crucial for us. Irreducible algebraic representations of $\mrm{GSp}(4)$ are classified by their highest weight, i.e. by two integers $k \geq k' \geq 0$. Let $k$ and $k'$ be two such integers. Assume that $k+k' \equiv p+q \pmod 2$ and let $W$ be an irreducible algebraic representation of $\mathrm{GSp}(4)$ with highest weight $\lambda(k, k', c)$, with the notation of section 2,  where $c=p+q+6$. Then, condition ($\ref{branching}$) above is equivalent to the following inequalities:
\begin{itemize}
\item we have $p \leq k$,
\item if $0 \leq p < k'$ and $p<k-k'$ then $k-k'-p \leq q \leq k-k'+p$,
\item if $0 \leq p < k'$ and $k-k' \leq p$ then $p-k+k' \leq q \leq p+k-k'$,
\item if $k' \leq p \leq k$ and $k'<k-p$ then $k-k'-p \leq q \leq k+k'-p$,
\item if $k' \leq p \leq k$ and $k-p \leq k'$ then $p-k+k' \leq q \leq k+k'-p$.
\end{itemize}
Let $H^3_{B, !}(S_L, W)$ be the image of the Betti cohomology with compact support in the cohomology without support, in the middle degree. By the theory of mixed Hodge modules, there is a pure real $\mbb{Q}$-Hodge structure of weight $3-c=-p-q-3$ on $H^3_{B, !}(S_L, W)$. "Real" means that the vector space $H^3_{B, !}(S_L, W)$ is endowed with an involution whose $\mbb{C}$-antilinear complexification stabilizes the Hodge filtration. For $H^3_{B, !}(S_L, W)$, this involution is just the map induced by complex conjugation on the complex points of $S$ and on $W$. Let us also denote by $Eis_\mathcal{H}^{p, q, W}$ the composite of $Eis_\mathcal{M}^{p, q, W}$ and of the regulator
$$
\begin{CD}
r_\mathcal{H}: H^4_\mathcal{M}(S_L, W) @>>> H^4_\mathcal{H}(S_L/\mbb{R}, W)
\end{CD}
$$
in absolute Hodge cohomology.

\begin{thm} \label{lemma1} \cite{lemma} Thm. 6.8. Assume $k>k'>0$, $k+k' \neq p+q$, $k-p-q-1 \neq 0$, $k-k'-p-q-2 \neq 0$, $k-p-q-2 \neq 0$. Assume that $(k-k'-p-q-2)/2$ and $(k-k'+p+q)/2$ are even and that the cusps in the boundary of the Baily-Borel compactifications of $M_K \times M_K$ and of $S_L$ are totally real. Then $Eis^{p, q, W}_\mathcal{H}$ factors through the inclusion
$$
\mathrm{Ext}^1_{\mathrm{MHS}_\mathbb{R}^+}(\mathbb{R}(0), H^3_{B, !}(S_L, W)_\mbb{R}) \subset H^4_\mathcal{H}(S_L/\mathbb{R}, W)
$$
where $\mathrm{MHS}_\mathbb{R}^+$ denotes the abelian category of mixed real $\mbb{R}$-Hodge structures and $H^3_{B, !}(S_L, W)_\mbb{R}$ denotes $H^3_{B, !}(S_L, W) \otimes_\mbb{Q} \mbb{R}$.
\end{thm}

Let $\pi = \bigotimes_v' \pi_v$ be a cuspidal automorphic representation of $\mrm{GSp}(4)$. As $\mrm{GSp}(4)$ is its own Langlands dual group, we can associate to $\pi$ the partial Euler product
$$
L_V(s, \pi, r)=\prod_{v \notin V} L(s, \pi_v, r)
$$
where $V$ denotes the set of places of $\mbb{Q}$ where $\pi$ is ramified together with the archimedean place and where $r: \mrm{GSp}(4) \lra \mrm{GL}(4)$ is the natural inclusion. This $L$-function is called the spinor, or degree four, $L$-function in the literature. Assume that $p=k-1$, $q=k'-1$ and that the non-archimedean component of $\pi$ occurs in the middle degree cohomology of $S$, with coefficients in $W$. This means that the central character of $\pi$ has infinity type $-k-k'-4$ and that the archimedean component of $\pi$ is a discrete series of Harish-Chandra parameter $(k+2, k'+1)$. Then, the non-archimedean part $\pi_f$ of $\pi$ is defined over its rationality field, which is a number field $E(\pi_f)$ \cite{bhr} and that will be enlarged if necessary.  As a consequence, we can look at the $L$-function $L_V(s, \pi, r)$ as an $E(\pi_f) \otimes_\mbb{Q} \mbb{C}$-valued function. In general, the automorphic representation $\pi$ will be stable at infinity (Def. \ref{stable} and Rem. \ref{abondance}) and we assume here that it is the case. Then, the $\mbb{Q}$-vector space
$$
M_B(\pi_f, W)=\mathrm{Hom}_{\mathbb{Q}[G(\mathbb{A}_f)]}(Res_{E(\pi_f)/\mathbb{Q}}\pi_f, H^3_{B, !}(S, W))
$$
underlies a pure real $\mathbb{Q}$-Hodge structure with coefficients in $E(\pi_f)$, which is of rank four and weight $-k-k'-1$. Let $\mbb{A}_f$ be the finite adeles of $\mbb{Q}$ and let
$$
\mathrm{Ext}^1_{\mathrm{MHS}_\mathbb{R}^+}(\mathbb{R}(0), H^3_{B, !}(S, W)_\mbb{R})= \varinjlim_L \mathrm{Ext}^1_{\mathrm{MHS}_\mathbb{R}^+}(\mathbb{R}(0), H^3_{B, !}(S_L, W)_\mbb{R})
$$
where the limit is taken over all levels $L \subset \mrm{GSp}(4)(\mbb{A}_f)$, which is a $\mbb{R}[\mrm{GSp}(4)(\mbb{A}_f)]$-module. The sub $\mbb{Q}[\mrm{GSp}(4)(\mbb{A}_f)]$-module generated by the images of the $Eis^{p, q, W}_\mathcal{H}$ for varying levels will be denoted by $\mathcal{K}(p, q, W)$. As $\pi$ is assumed to be stable, the $\pi_f$-isotypical component  $\mathrm{Ext}^1_{\mathrm{MHS}_\mathbb{R}^+}(\mathbb{R}(0),  M_{B}(\pi_f, W)_\mathbb{R})$ of $\mathrm{Ext}^1_{\mathrm{MHS}_\mathbb{R}^+}(\mathbb{R}(0), H^3_{B, !}(S, W))_\mbb{R})$ is a rank one $E(\pi_f) \otimes_\mbb{Q} \mbb{R}$-module and is endowed with its Deligne $E(\pi_f)$-structure $\mathcal{D}(\pi_f, W)$. Let $\mathcal{K}(\pi_f, W)$ denote the projection of $\mathcal{K}(p, q, W)$ on $\mathrm{Ext}^1_{\mathrm{MHS}_\mathbb{R}^+}(\mathbb{R}(0),  M_{B}(\pi_f, W)_\mathbb{R})$. To state our main result, we need to consider the Deligne period $c^-(\pi_f, W)$ associated to $\pi_f$ and $W$ and Harris' occult period invariant $a(\pi, \nu_1, \nu_2)$, first introduced in \cite{occult}. Roughly speaking, the invariant $a(\pi, \nu_1, \nu_2)$ measures the difference between the rational structure on $\pi_f$ coming from de Rham cohomology and the one coming from the Fourier expansion of cusp forms along the Siegel parabolic. A precise definition will be given in the body of the paper but note that the notation $a(\pi, \nu_1, \nu_2)$ is slightly abusive as the occult period invariant can not be defined merely by reference to the abstract representation $\pi$ but depends on its realization in cohomology. Here is our main result:

\begin{thm}  \label{emy<3} Let $\pi$ be as above and let $\check{\pi}$ the contragredient representation. Let $\nu_1^0$, respectively $\nu_2^0$, be a finite order Hecke character of sign $(-1)^{k-1}$, respectively of sign $(-1)^{k'-1}$, and let $\nu_1=|\,\,|^{1-k'}\nu_1^0$, respectively $\nu_2=|\,\,|^{1-k} \nu_2^0$. Let $V$ be the finite set of places where $\pi, \nu_1$ or $\nu_2$ is ramified together with the archimedean place. Assume that:
\begin{itemize}
\item we have $k>k'>0$,
\item we have $k+1 \equiv k' \equiv 0 \pmod 2$,
\item we have $k \neq 3, k' \neq 2$,
\item the automorphic representation $\pi$ is stable at infinity.
\end{itemize}
Then
$$
\mathcal{K}(\pi_f, W)= \pi^{-2}a(\pi, \nu_1, \nu_2) c^-(\pi_f, W)L_V(k+k'-1/2, \check{\pi})\mathcal{D}(\pi_f, W).
$$
\end{thm}

Three remarks are in order.
\begin{itemize}
\item The author announced a similar result some time ago (\cite{lemma1} Thm. 4) but the proof he thought to have found contained an error. The present work shows that a slight variant of \cite{lemma1} Thm. 4 is true and goes significantly further (see Cor. \ref{non-nullite1} below).
\item The hypothesis $k+1 \equiv k' \equiv 0 \pmod 2$ implies that we can apply Thm. \ref{lemma1} to $p=k-1$ and $q=k'-1$. This limitation on $p$ and $q$ can be called technical and we would be able to get a similar result for arbitrary $p$ and $q$ satisfying the assumption of Thm. \ref{lemma1} at the cost of very long (but elementary) calculations. See Rem. \ref{molev} for more details.
\item The term $\pi^{-2} a(\pi, \nu_1, \nu_2) c^-(\pi_f, W)$ is expected to be an algebraic number according to Beilinson's conjecture. This expectation is coherent with the main result the paper \cite{occult}, where the transcendental part of critical values of the spinor $L$-function are computed in terms of $a(\pi, \nu_1, \nu_2)$. By looking closely at the arguments of \cite{occult}, it is not too difficult to see that the transcendental part of the uncomputed constants appearing in the main result of \cite{occult} are powers of $\pi$. As a consequence, according to Deligne conjecture on critical values \cite{valeurs-deligne}, as stated for example in \cite{occult} 2.6, we should have $a(\pi, \nu_1, \nu_2)c^-(\pi_f, W) \in \pi^{\mathbb{Z}} \overline{\mbb{Q}}^\times$. The explicit computation of these constants and of the archimedean integrals in \cite{occult} should give $a(\pi, \nu_1, \nu_2)c^-(\pi_f, W) \in \pi^{2} \overline{\mbb{Q}}^\times$. For more details on the compatibility between the main result of \cite{occult} and Deligne conjecture see the Remarks following \cite{occult} Thm. 3.5.5.
\end{itemize}

\begin{cor} \label{non-nullite1} Let $n \geq 0$ be an integer. Let $A \lra S$ be the universal abelian surface of infinite level over the Siegel threefold and let $A^n$ be the $n$-th fold fiber product over $S$. If $n$ is  odd and $n \geq 7$, then motivic cohomology space $H^{n+4}_\mathcal{M}(A^n, \mbb{Q}(n+2))$ is non-zero.
\end{cor}

The proof of Thm. \ref{emy<3} relies on three main ingredients. The first is Thm. \ref{lemma1} above. The second is the analytic description of the composite $Eis^{p, q, W}_\mathcal{D}=r_\mathcal{D} \circ Eis^{p, q, W}_\mathcal{M}$ where
$$
\begin{CD}
r_\mathcal{D}: H^4_\mathcal{M}(S, W) @>>> H^4_\mathcal{D}(S/\mbb{R}, W)
\end{CD}
$$
is the regulator in real Deligne-Beilinson cohomology. This follows from Beilinson's explicit description of the image of the Eisenstein symbol in Deligne-Beilinson cohomology, the functoriality of the regulator and the explicit description of the cup-product and the Gysin morphism in Deligne-Beilinson cohomology. The third ingredient is an integral representation of the spinor $L$-function whose study was initiated in \cite{piatetski-shapiro} and carried on in \cite{bump-friedberg-furusawa}, \cite{moriyama}, \cite{takloo-bighash}. The contribution of the present work is to explain why these three ingredients, which might seem of quite different natures, are in fact closely related.\\

Let us give an overview of the different sections of the article. In section \ref{notations} we collect conventions and notations that will be important in the sequel. We would like to draw the reader's attention to section 2.4 where the normalizations of the measures on adelic groups are explained and to section 2.5 where the convention on the weight of variation of Hodge structures is adopted, as this last point differs from one author to the other. Section \ref{Hodge-deRham} recalls the connection between discrete series $L$-packets for $\mrm{GSp}(4)$ and the Hodge decomposition of $H^3_!$. In section \ref{sect-calc-reg}, we provide the basis for the computation of the regulator. First, the explicit description of the map $Eis^{p, q, W}_\mathcal{D}$ is given (Prop. \ref{eisenstein-deligne}). Then, we adapt to our setting an idea of Beilinson which permits to reduce the computation of the regulator to the computation of a Poincar\'e duality pairing (Lem. \ref{dualite}). Finally, in a series of lemmas, the computation of the pairing, hence of the regulator, is reduced to the computation of an adelic integral, where the integrand is the product of a cusp form by an Eisenstein series (Cor. \ref{resume}). Section \ref{calcul-de-l-integrale} is devoted to the computation of this integral. First, we need to compare very precisely the Eisenstein series appearing in our integral to the one defined by Piatetski-Shapiro in \cite{piatetski-shapiro}. This is done in Prop. \ref{comparaison}. Then, we have to study an integral as defined by Piatetski-Shapiro. The constructions of \cite{piatetski-shapiro} are based on the Fourier expansion of cusp forms along the Siegel parabolic and more precisely, on the existence of Bessel models for automorphic representations of $\mrm{GSp}(4)$. The integrals of \cite{piatetski-shapiro} expand into Euler products, whose factors need to be computed. In the case we are interested in, the unramified non-archimedean local integrals are computed in Prop. \ref{integrale-nonram}. The ramified non-archimedean integrals and the archimedean integral are analyzed in Prop. \ref{integrales-ram} and Prop. \ref{integrale-archi} respectively. In section \ref{section-periodes}, we introduce Deligne periods, Harris' occult period invariant and perform a period computation which explains the contribution of $\mathcal{D}(\pi_f, W)$ to Thm. \ref{emy<3}. In section \ref{derniere-section}, we explain how to deduce the existence of split Bessel models from results of Moriyama \cite{moriyama-ajm} and Takloo-Bighash \cite{takloo-bighash}. We finally prove Thm. \ref{emy<3} and Cor. \ref{non-nullite1} as an easy consequence of the previous results.

\section{Notations and conventions} \label{notations}

In this section, we collect some notations, conventions and basic results that will be used in the rest of the article. The reader might prefer to look at this section only according to his needs, following the references given in the paper.\\

2.1. Given a ring $A$, a $A$-algebra $A \longrightarrow B$ and a $A$-module $M$, we will denote by $M_B$ the $B$-module $B \otimes_A M$ when the base ring $A$ is clear from the context. Similarly for any $A$-scheme $X$, we will denote by $X_B$ the $B$-scheme obtained by extension of scalars to $B$.\\

2.2. Let $\mathbb{A}_f=\mathbb{Q} \otimes_{\mathbb{Z}} \widehat{\mathbb{Z}}$, respectively $\mathbb{A}=\mathbb{R} \times \mathbb{A}_f$, denote the topological rings of finite adeles, respectively of adeles, of $\mathbb{Q}$.  The field $\mathbb{R}$ is endowed with its usual absolute value. For every prime number $p$, we normalize the non-archimedean absolute value on $\mathbb{Q}_p$ by $|p|=p^{-1}$ as usual. Hence the map $\mathbb{A}^\times \longrightarrow \mathbb{C}^\times$ defined by $(x_v)_v \longmapsto \prod_v |x_v|$ induces a continuous character $|\,\,|: \mathbb{Q}^\times \backslash \mathbb{A}^\times \longrightarrow \mathbb{C}^\times$. Every continuous character $\nu= \bigotimes'_v \nu_v: \mathbb{Q}^\times \backslash \mathbb{A}^\times \longrightarrow \mathbb{C}^\times$ can be written uniquely $\nu=|\,\,|^s \nu^0$ for some complex number $s$ and some finite order character $\nu^0$ (\cite{bump} Prop. 3.1.2 (ii)). The infinity type of $\nu$ is by definition the complex number $s$ and the sign of $\nu$,  is by definition $\nu_\infty(-1)$.\\


2.3. Let $I_2$ be the identity matrix of size $2$ and let
$$
\psi=
\begin{pmatrix}
 & I_2\\
-I_2 & \\
\end{pmatrix}.
$$
The symplectic group $G=\mathrm{GSp}(4)$ is the reductive linear algebraic group over $\mathbb{Q} $ defined as
$$
G=\{g \in \mathrm{GL}_{4 / \mathbb{Q}} \,|\, ^t\!g \psi g = \nu(g) \psi, \, \nu(g) \in \mathbb{G}_m\}.
$$
Then $\nu: G \longrightarrow \mathbb{G}_m$ is a character and the derived group of $G$ is $\mathrm{Sp}(4)= \mathrm{Ker} \, \nu$. We denote by $T \subset G$ the diagonal maximal torus defined as
$$
T=\{ \mathrm{diag}(\alpha_1, \alpha_2, \alpha_1^{-1} \nu, \alpha_2^{-1} \nu)| \, \alpha_1, \alpha_2, \nu \in \mathbb{G}_m \}
$$
and by $B=TU$ the standard Borel subgroup of upper triangular matrices in $G$. We identify the group $X^*(T)$ of algebraic characters (we will also, as usual, say "weights") of $T$ to the subgroup of $\mathbb{Z}^2 \oplus \mathbb{Z}$ of triples $(k, k', c)$ such that $k+k' \equiv c \pmod 2$ via
$$
\lambda(k, k', c): \mathrm{diag}(\alpha_1, \alpha_2, \alpha_1^{-1} \nu, \alpha_2^{-1} \nu) \longmapsto \alpha_1^k \alpha_2^{k'} \nu^{\frac{c-k-k'}{2}}.
$$
Let $\rho_1=\lambda(1, -1, 0)$ be the short simple root and $\rho_2=\lambda(0, 2, 0)$ be the long simple root. Then the set $R \subset X^*(T)$ of roots of $T$ in $G$ is
$$
R=\{ \pm \rho_1, \pm \rho_2, \pm (\rho_1 + \rho_2), \pm (2 \rho_1+\rho_2)\}
$$
and the subset $R^+ \subset R$ of positive roots with respect to $B$ is
$$
R^+=\{ \rho_1, \rho_2, \rho_1+ \rho_2, 2 \rho_1+ \rho_2 \}.
$$
Then, the set of dominant weights is the set of $\lambda(k, k', c)$ such that $k \geq k' \geq 0$. For any dominant weight $\lambda$, there is an irreducible algebraic representation $V_\lambda$ of $G$ of highest weight $\lambda$, unique up to isomorphism, and all isomorphism classes of irreducible algebraic representations of $G$ are obtained in this way. If $V_\lambda$ is irreducible with highest weight $\lambda(k, k', c)$, the contragredient of $V_\lambda$ has highest weight $\lambda(k, k', -c)$.\\

The Weyl group $W$ of $(G, T)$ is defined as the normalizer of $T$ in $G$ modulo its centralizer. It is a group of order eight such that the images in $W$ of the elements
$$
s_1=
\begin{pmatrix}
  & 1 & &\\
1 &   & &\\
   &   & & 1\\
   &   & 1 & \\
\end{pmatrix},
s_2=\begin{pmatrix}
1  &  & &\\
    &   & & 1\\
    &   & 1 & \\
    & -1  &  & \\
\end{pmatrix}
$$
generate $W$. Then $W$ acts on $X^*(T)$ according to the rule
$$
(w.\lambda)(t)=\lambda(w^{-1}tw)
$$
and we have $s_1. \lambda(k, k', c)=\lambda(k', k, c)$ and $s_2. \lambda(k, k', c)=\lambda(k, -k', c)$ which means that $s_1$ corresponds to the reflection associated to the short simple root $\rho_1$ and $s_2$ to the one associated to the long simple root $\rho_2$.\\

We shall also denote by $G'$ the reductive linear algebraic group $\mathrm{GL}_2 \times_{\mathbb{G}_m} \mathrm{GL}_2$ over $\mathbb{Q}$, where the fiber product is over the determinant. As mentioned in the introduction, we have the embedding $\iota: G' \longrightarrow G$ defined by
$$
\iota\left( \begin{pmatrix}
a & b\\
c & d\\
\end{pmatrix}, \begin{pmatrix}
a' & b'\\
c' & d'\\
\end{pmatrix} \right) =
\begin{pmatrix}
a & & b & \\
 & a' &  & b'\\
c &  & d & \\
 & c' &  & d'\\
\end{pmatrix}.
$$
The subgroup of upper triangular matrices in $G'$ is denoted by $B'$.\\

2.4. Let us explain our normalizations of the Haar measures on adeles groups, in the case of $G(\mbb{A})$. A similar discussion can be done for $G'(\mbb{A})$. Consider the unitary group $\mathrm{U}(2)=\{g \in \mathrm{GL}(2, \mathbb{C}) \,|\, ^t\!\overline{g}g= I_2 \}$ where $\overline{g}$ denotes the complex conjugate of $g$. The map $\kappa: \mathrm{U}(2) \longrightarrow \g{Sp}(4, \mathbb{R})$ defined by $$g=A+iB \longmapsto \begin{pmatrix}
A & B\\
-B & A\\
\end{pmatrix},$$
where $A$ and $B$ denote the real and imaginary parts of $g$, identifies $\g{U}(2)$ with a maximal compact subgroup $K$ of  $\g{Sp}(4, \mathbb{R})$. Let $A_G=\mathbb{R}_+^\times$ be the identity component of the center of $G$, write $K_G$ for the subgroup $A_G K$ of $G(\mbb{R})$, which is maximal compact modulo the center, and let $\xi: A_G \longrightarrow \mathbb{C}^\times$ be a continuous character. Denote again by $\xi: G(\mathbb{A}) \longrightarrow \mathbb{C}^\times$ the extension of $\xi$ to $G(\mbb{A})$ given by $\xi((g_v)_v)=|\nu(g_\infty)|^\frac{1}{2}$. The choice of a generator $\textbf{1}_G$ of the highest exterior power of $\g{Lie}\,G_\mathbb{R}/ \g{Lie}(A_G K)$ determines a left translation invariant measure on $G(\mathbb{R})/(A_G K)$. Together with the Haar measure on $K$ whose total mass is one, this datum determines a left translation invariant measure $dg_\infty$ on $G(\mbb{R})/A_G$. For every prime number $p$, we endow $G(\mbb{Q}_p)$ with the Haar measure $dg_p$ for which $G(\mathbb{Z}_p)$ has volume one. Then, we have the translation invariant measure $dg=\prod_{v \leq \infty} dg_v$ on $G(\mathbb{A})$. To introduce automorphic representations, let $L^2(G(\mathbb{Q})\backslash G(\mathbb{A}), \xi)$ be the space of functions $f: G(\mathbb{Q})\backslash G(\mathbb{A}) \longrightarrow \mathbb{C}$ such that:
\begin{itemize}
\item $\forall z \in A_G, \forall g \in G(\mbb{A}), f(zg)=\xi(z)f(g),$
\item the function $\xi^{-1}f$ is square-integrable on $A_G G(\mathbb{Q}) \backslash G(\mathbb{A})$.
\end{itemize}

The subspace $^0L^2(G(\mathbb{Q})\backslash G(\mathbb{A}), \xi)$ of admissible cuspidal functions is a discrete sum with finite multiplicities of closed irreducible $((\mathfrak{g}, K_G) \times G(\mathbb{A}_f))$-invariant subspaces, which are cuspidal automorphic representations of $G$ by definition. Here, we denote by $\mathfrak{g}$ the Lie algebra of $G(\mbb{R})$. Let $\mathcal{C}^\infty(G(\mbb{Q})\backslash G(\mbb{A}), \xi)$, respectively $\mathcal{C}^\infty_c(G(\mbb{Q})\backslash G(\mbb{A}), \xi)$, be the space of functions $f: G(\mathbb{Q})\backslash G(\mathbb{A}) \longrightarrow \mathbb{C}$ such that:
\begin{itemize}
\item $\forall z \in A_G, \forall g \in G(\mbb{A}), f(zg)=\xi(z)f(g),$
\item the restriction of $f$ to $G(\mbb{R})$ is $\mathcal{C}^\infty$, respectively $\mathcal{C}^\infty$ and compactly supported modulo $A_G$,
\item the restriction of $f$ to $G(\mbb{A}_f)$ is locally constant and compactly supported.
\end{itemize}

Let $\mathcal{C}^\infty_{(2)}(G(\mbb{Q})\backslash G(\mbb{A}), \xi)$ denote the space $L^2(G(\mathbb{Q})\backslash G(\mathbb{A}), \xi) \cap \mathcal{C}^\infty(G(\mbb{Q})\backslash G(\mbb{A}), \xi)$. We have natural inclusions of $((\mathfrak{g}, K_G) \times G(\mathbb{A}_f))$-modules
\begin{equation} \label{inclusion-fonctions}
\mathcal{C}^\infty_c(G(\mbb{Q})\backslash G(\mbb{A}), \xi) \subset \mathcal{C}^\infty_{(2)}(G(\mbb{Q})\backslash G(\mbb{A}), \xi) \subset \mathcal{C}^\infty(G(\mbb{Q})\backslash G(\mbb{A}), \xi).
\end{equation}
Finally, let $\mathcal{C}^\infty_{cusp}(G(\mbb{Q})\backslash G(\mbb{A}), \xi)$ denote $^0L^2(G(\mathbb{Q})\backslash G(\mathbb{A}), \xi) \cap \mathcal{C}^\infty(G(\mbb{Q})\backslash G(\mbb{A}), \xi)$. Smooth truncation to a large compact modulo the center subset induces a map
\begin{equation} \label{troncation}
\mathcal{C}^\infty_{cusp}(G(\mbb{Q})\backslash G(\mbb{A}), \xi) \lra \mathcal{C}^\infty_{c}(G(\mbb{Q})\backslash G(\mbb{A}), \xi).
\end{equation}

2.5. Let $\mathbb{S}=Res_{\mathbb{C}/\mathbb{R}} \mathbb{G}_{m, \mathbb{C}}$ be the Deligne torus. Following Deligne and Pink, our convention for the equivalence of categories between algebraic representations of $\mbb{S}$ in finite dimensional $\mbb{R}$-vector spaces and (semisimple) mixed $\mbb{R}$-Hodge structures is as follows. Let $(\rho, V)$ be such a representation of $\mbb{S}$. Then the summand $V^{p, q}$ of $V_\mbb{C}$ of type $(p, q)$ is the summand on which $\rho(z_1, z_2)$ acts by multiplication by $z_1^{-p}z_2^{-q}$ for any $(z_1, z_2) \in \mbb{S}(\mbb{C})$. In particular, any algebraic representation $V$ of $\mbb{S}$ with central character $c$ corresponds to a pure Hodge structure of weight $-c$. This convention disagrees with the one adopted in \cite{occult}, \cite{taylor} and \cite{weissauer} but agrees with the one adopted in \cite{kings}, \cite{lemma1} and \cite{pink}.

\section{Motives for $\mathrm{GSp}(4)$} \label{Hodge-deRham}

In this section, we are mainly interested in reviewing the connection between the Hodge decomposition of the middle degree cohomology of Siegel threefolds and the discrete series $L$-packets for $G$. Let us start by recalling the classification of discrete series for $G$.

\subsection{Discrete series $L$-packets} \label{discrete_series_classification} In this section and in several other places of the present article, the reader will have to be familiar with the representation theory of compact Lie groups as exposed, for example, in chapter IV of \cite{knapp}.

In section 2.4, we identified the unitary group $\mathrm{U}(2)$ to a maximal compact subgroup $K$ of $\mathrm{Sp}(4, \mbb{R})$ via the isomorphism $\kappa: \mathrm{U}(2) \simeq K$. Let $\mathfrak{k}$ denote the Lie algebra of $K$ and by $\mathfrak{k}_\mathbb{C}$ its complexification. The differential of $\kappa$ induces an isomorphism of Lie algebras $d\kappa: \mathfrak{gl}_{2, \mathbb{C}} \simeq \mathfrak{k}_\mathbb{C}$. Let $\mathfrak{sp}_4$ denote the Lie algebra of $\g{Sp}(4, \mathbb{R})$ and by $\mathfrak{sp}_{4, \mathbb{C}}$ its complexification. A compact Cartan subalgebra of $\mathfrak{sp}_4$ is defined as $\mathfrak{h}=\mathbb{R}T_1 \oplus \mathbb{R}T_2$ where
\begin{eqnarray*}
T_1 &=& d\kappa \left(  \begin{pmatrix}
i &  \\
 &  \\
\end{pmatrix}\right)=\begin{pmatrix}
 &   & 1 & \\
 &  &  & \\
-1 &  &  & \\
  &  &  & \\
\end{pmatrix},\\
 T_2 &=& d\kappa \left(  \begin{pmatrix}
 &  \\
 &  i\\
\end{pmatrix}\right)=\begin{pmatrix}
 &  &  & \\
 &  &  & 1\\
 &  &  & \\
 & -1 &  & \\
\end{pmatrix}.
\end{eqnarray*}
Define a $\mathbb{C}$-basis of $\mathfrak{h}_\mathbb{C}^*$ by $e_1(T_1)=i, e_1(T_2)=0, e_2(T_1)=0, e_2(T_2)=i$. The root system $\Delta$ of the pair $(\mathfrak{sp}_{4, \mathbb{C}}, \mathfrak{h}_\mathbb{C})$ is $
\Delta=\{ \pm 2 e_1, \pm 2 e_2, \pm(e_1 \pm e_2)\}.$
We denote by $\Delta_c$, respectively $\Delta_{nc}$, the set of compact, respectively non-compact roots in $\Delta$. We have $\Delta_c=\{ \pm (e_1-e_2)\}$ and $\Delta_{nc}=\Delta - \Delta_c$. We choose the set of positive roots as $\Delta^+=\{e_1-e_2, 2e_1, e_1+e_2, 2e_2\}$. Then, the set of compact, respectively non-compact, positive roots is $\Delta^+_c=\Delta_c \cap \Delta^+$, respectively $\Delta_{nc}^+=\Delta_{nc} \cap \Delta^+$. For each symmetric matrix $Z \in \mathfrak{gl}_{2, \mathbb{C}}$, define the element $p_\pm(Z)$ of $\mathfrak{sp}_{4, \mathbb{C}}$ by
$$
p_\pm(Z)=\begin{pmatrix}
Z & \pm iZ\\
\pm iZ & -Z\\
\end{pmatrix}.
$$
Let $X_{(\alpha_1, \alpha_2)} \in \mathfrak{sp}_{4, \mathbb{C}}$ be defined as
$$
X_{\pm(2, 0)}=p_\pm \left(  \begin{pmatrix}
1 & \\
 &  \\
\end{pmatrix}\right), X_{\pm(1, 1)}=p_\pm \left(  \begin{pmatrix}
 & 1\\
1 &  \\
\end{pmatrix}\right), X_{\pm(0, 2)}=p_\pm \left(  \begin{pmatrix}
 & \\
 &  1\\
\end{pmatrix}\right).
$$
It follows from an easy computation that $X_{(\alpha_1, \alpha_2)}$ is a root vector corresponding to the non-compact root $(\alpha_1, \alpha_2)=\alpha_1 e_1+ \alpha_2 e_2$. If we set
\begin{equation} \label{pplus-pmoins}
\mathfrak{p}^\pm= \bigoplus_{\alpha \in \Delta_{nc}^+} \mathbb{C} X_{\pm \alpha},
\end{equation}
we have the Cartan decomposition $\mathfrak{sp}_{4, \mathbb{C}}= \mathfrak{k}_\mathbb{C} \oplus \mathfrak{p}^+ \oplus \mathfrak{p}^-$.\\

Integral weights are defined as the $(k, k')=ke_1+k'e_2 \in \mathfrak{h}_\mathbb{C}^*$ with $k, k' \in \mathbb{Z}$ and an integral weight is dominant for $\Delta_c^+$ if $k \geq k'$. Assigning its highest weight to a finite dimensional irreducible complex representation $\tau$ of $K$, we define a bijection between isomorphism classes of finite-dimensional irreducible complex representations of $K$ and dominant integral weights, whose inverse will be denoted by $(k ,k') \longmapsto \tau_{(k, k')}$. Let $(k, k')$ be a dominant integral weight and let $d=k-k'$. Then $\dim_\mbb{C} \tau_{(k, k')}=d+1$. More precisely, there exists a basis $(v_s)_{0 \leq s \leq d}$ of $\tau_{(k, k')}$, such that
\begin{eqnarray*}
\tau_{(k, k')} \left( d\kappa\begin{pmatrix}
1 &  \\
 &  \\
\end{pmatrix} \right) v_s &=& (s+k') v_s,\\
\tau_{(k, k')} \left( d\kappa\begin{pmatrix}
 &  \\
 & 1 \\
\end{pmatrix} \right) v_s &=& (-s+k) v_s,\\
\tau_{(k, k')} \left( d\kappa\begin{pmatrix}
 & 1 \\
 &  \\
\end{pmatrix} \right) v_s &=& (s+1) v_{s+1},\\
\tau_{(k, k')} \left( d\kappa\begin{pmatrix}
 &  \\
1 &  \\
\end{pmatrix} \right) v_s &=& (d-s+1) v_{s-1}\\
\end{eqnarray*}
which we call a standard basis of $\tau_{(k, k')}$. In the identities above, we agree to use the convention $v_{-1}=v_{d+1}=0.$ We will denote by $W_K$ be the Weyl group of $(\mathfrak{k}_\mathbb{C}, \mathfrak{h}_\mathbb{C})$. According to the classification theorem \cite{knapp} Thm. 9.20, we have:

\begin{pro} \label{lpaquet} Let $G(\mbb{R})^+$ be the identity component of $G(\mbb{R})$, let $\xi$ be a character of $A_G$ and let $(k, k') \in \mathfrak{h}_\mathbb{C}^*$  be an integral weight. Assume $k \geq k' \geq 0$. Then, there exist four isomorphism classes $\pi_\infty^H$, $\pi_\infty^W$, $\overline{\pi}_\infty^W$, $\overline{\pi}_\infty^H$ of irreducible discrete series representations of $G(\mathbb{R})^+$ with Harish-Chandra parameter $(k+2, k'+1)$ and central character $\xi$. Furthermore, the restrictions of these representations to $K$ contain as minimal $K$-types the representations $\tau_{(k+3, k'+3)}$, $\tau_{(k+3, -k'-1)}$, $\tau_{(k'+1, -k-3)}$, $\tau_{(-k'-3, -k-3)}$ respectively.
\end{pro}

\begin{rems} \label{conj-ds} For the cohomological considerations that follow, we need to be more precise and explain the specific representatives of the isomorphism classes of discrete series that we choose. Let $N$ be the element of $G(\mbb{R})$ defined as
\begin{eqnarray*}
N &=& \begin{pmatrix}
  & -1 & &\\
-1 &   & &\\
   &   & & 1\\
   &   & 1 & \\
\end{pmatrix}.
\end{eqnarray*}
We shall see that $N$ is related to the action of complex conjugation on the set of complex points of Siegel threefolds (Prop. \ref{milne-shih}). For the adjoint action $\mrm{Ad}: G(\mbb{R}) \lra \mrm{GL}(\mathfrak{h})$, we have
\begin{eqnarray*} \label{Naction}
\mrm{Ad}_{N}(T_1) &=& -T_2,\\
\mrm{Ad}_{N}(T_2) &=& -T_1.
\end{eqnarray*}
Furthermore, as $\nu(N)=-1$, the matrix $N$ normalizes $G(\mbb{R})^+$. It follows from the identities above and from \cite{knapp} Thm. 9.20 that the representation of $G(\mbb{R})^+$ obtained by conjugating $\pi_\infty^H$, respectively $\pi_\infty^W$, by $N$ is isomorphic to $\overline{\pi}_\infty^H$, respectively $\overline{\pi}_\infty^W$. If we fix such isomorphisms, given a vector $\Psi_\infty$  belonging to the space underlying $\pi_\infty^W$, of weight $(u, v)$, we will denote by $\overline{\Psi}_\infty$ the same vector regarded as a vector of the space underlying $\overline{\pi}_\infty^W$. It has weight $(-v, -u)$.
\end{rems}

For the arithmetic applications we aim at, we will need the following result:

\begin{pro} \label{bhr} \cite{bhr} Thm. 3.2.2. Let $\pi=\pi_\infty \otimes \pi_f$ be a cuspidal automorphic representation of $G$ such that $\pi_\infty|_{G(\mbb{R})^+}$ is a discrete series. Then $\pi_f$ is defined over its rationality field, which is a number field $E(\pi_f)$.
\end{pro}

\subsection{Cohomology of Siegel threefolds} \label{coho-siegel} Siegel threefolds are the Shimura varieties associated to the group $G$. Let us briefly recall their definition. Let $\mathbb{S}=Res_{\mathbb{C}/\mathbb{R}} \mathbb{G}_{m, \mathbb{C}}$ be the Deligne torus and let $\mathcal{H}$ be the $G(\mathbb{R})$-conjugacy class of the morphism $h: \mathbb{S} \longrightarrow G_\mathbb{R}$ given on $\mathbb{R}$-points by
$$
x+iy \longmapsto \begin{pmatrix}
x & & y & \\
 & x &  & y\\
-y &  & x & \\
 &  -y &  & x\\
\end{pmatrix}.
$$
The pair $(G, \mathcal{H})$ is a pure Shimura datum in the sense of \cite{pink} 2.1. The cocharacter $\mu$ of $G_\mbb{C}$ associated to the morphism $h$ as in \cite{pink} 1.3 induces
$$
z \longmapsto \begin{pmatrix}
z & &  & \\
 & z &  & \\
 &  & 1 & \\
 &   &  & 1\\
\end{pmatrix}
$$
on complex points, hence it is defined over $\mbb{Q}$. In other terms, the reflex field of $(G, \mathcal{H})$ is $\mbb{Q}$. For any neat compact open subgroup $L$ of $G(\mathbb{A}_f)$, we denote by $S_L$ the Siegel threefold of level $L$. This is a smooth quasi-projective $\mathbb{Q}$-scheme such that, as complex analytic varieties, we have
$$
S_{L, \mathbb{C}}^{an}=G(\mathbb{Q} )\backslash ( \mathcal{H} \times G(\mathbb{A}_f)/L),
$$
where $S_{L, \mathbb{C}}^{an}$ denotes the analytification of the base change of $S_L$ to $\mathbb{C}$. For $g \in G(\mathbb{A}_f)$ and $L$, $L'$ two neat compact open subgroups of $G(\mathbb{A}_f)$ such that $g^{-1}L'g \subset L$, right multiplication by $g$ on $S_{L, \mathbb{C}}^{an}$ descends to a morphism $[g]: S_{L'} \longrightarrow S_L$ of $\mathbb{Q}$-schemes, which is finite \'etale. This implies that there is an action of $G(\mathbb{A}_f)$ on the projective system $(S_L)_L$ indexed by neat compact open subgroups of $G(\mathbb{A}_f)$. In what follows, all compact open subgroups of $G(\mathbb{A}_f)$ will be assumed to be neat and we will not mention it anymore. Because the $S_L$ are Shimura varieties associated to $(G, \mathcal{H})$, for any algebraic representation $E$ of $G$ in a finite dimensional $\mathbb{Q}$-vector space, we have a polarizable variation of $\mathbb{Q}$-Hodge structure, abusively denoted again by $E$, on $S_L$. We take the liberty not to mention the level $L$ in the notation because these variation of Hodge structures are compatible under the pull-back maps induced by the above morphisms $[g]$. We will also denote by $E$ the local system underlying the variation of Hodge structure $E$. This should not lead to confusion.\\

Let $E$ an irreducible algebraic representation of $G$ in a finite dimensional $\mathbb{C}$-vector space, let $\xi$ be the inverse of its central character  and let $L$ be a compact open subgroup of $G(\mathbb{A}_f)$. Let $\mathcal{A}^*_c(S_L, E)$ be the de Rham complex of $\mathcal{C}^\infty$ differential forms with compact support on $S_{L, \mathbb{C}}^{an}$ with values in the local system $E$, let $\mathcal{A}^*_{(2)}(S_L, E)$ be the complex of square integrable differential forms and let $\mathcal{A}^*(S_L, E)$ be the complex of usual differential forms. If $\circ$ is the symbol $c, (2)$ or the empty symbol define
\begin{eqnarray*}
\mathcal{A}_\circ^*(S, E) &=& \varinjlim_L \mathcal{A}_\circ^*(S_L, E).
\end{eqnarray*}
When $\circ$ is $c$ or $(2)$, this definition is legitimate because the transition morphisms $[g]$ are finite \'etale.  Moreover, these complexes carry an action of $G(\mathbb{A}_f)$ induced by the action on $(S_L)_L$ described above. In section 2.4, we introduced the Lie algebra $\mathfrak{g}$ of $G(\mbb{R})$ and the subgroup $K_G=A_G K$ of $G(\mbb{R})$, which is maximal compact modulo the center. For any $(\mathfrak{g}_\mathbb{C}, K_G)$-module $V$, let $C^*(\mathfrak{g}_\mathbb{C}, K_G, V)$ be the $(\mathfrak{g}_\mathbb{C}, K_G)$-complex of $V$ as defined in \cite{borel-wallach} I. According to \cite{borel-wallach} VII $\S$ 2, we have $G(\mbb{A}_f)$-equivariant isomorphisms of complexes
\begin{eqnarray*}
\mathcal{A}_c^*(S, E) &\simeq& C^*(\mathfrak{g}_\mathbb{C}, K_G, E \otimes_\mathbb{C} \mathcal{C}^\infty_c(G(\mbb{Q})\backslash G(\mbb{A}), \xi)),\\
\mathcal{A}^*(S, E) &\simeq& C^*(\mathfrak{g}_\mathbb{C}, K_G, E  \otimes_\mathbb{C} \mathcal{C}^\infty(G(\mbb{Q})\backslash G(\mbb{A}), \xi))
\end{eqnarray*}
which are compatible with the inclusions $\mathcal{A}_c^*(S, E) \subset \mathcal{A}^*(S, E)$ and
$$
C^*(\mathfrak{g}_\mathbb{C}, K_G, E \otimes_\mathbb{C} \mathcal{C}^\infty_c(G(\mbb{Q})\backslash G(\mbb{A}), \xi)) \subset C^*(\mathfrak{g}_\mathbb{C}, K_G, E \otimes_\mathbb{C} \mathcal{C}^\infty(G(\mbb{Q})\backslash G(\mbb{A}), \xi)).
$$
Taking cohomology, we obtain $G(\mbb{A}_f)$-equivariant isomorphisms
\begin{eqnarray*}
H_{dR, c}^*(S, E) &\simeq& H^*(\mathfrak{g}_\mathbb{C}, K_G,  E \otimes_\mathbb{C} \mathcal{C}^\infty_c(G(\mbb{Q})\backslash G(\mbb{A}), \xi)),\\
H^*_{dR}(S, E) &\simeq& H^*(\mathfrak{g}_\mathbb{C}, K_G,  E \otimes_\mathbb{C} \mathcal{C}^\infty(G(\mbb{Q})\backslash G(\mbb{A}), \xi))
\end{eqnarray*}
which are compatible with the maps from cohomology with compact support to cohomology without support. Let $H^*_{(2)}(S, E)$ denote the $L^2$ cohomology of $S$ with coefficients in $E$, i.e. the cohomology of the complex $ \mathcal{A}_{(2)}^*(S, E)$. According to \cite{borel}, we have a $G(\mbb{A}_f)$-equivariant isomorphism
$$
H^*_{(2)}(S, E)=H^*(\mathfrak{g}_\mathbb{C}, K_G,  E \otimes_\mathbb{C} \mathcal{C}^\infty_{(2)}(G(\mbb{Q})\backslash G(\mbb{A}), \xi))
$$
Applying the $(\mathfrak{g}_\mbb{C}, K_G)$-cohomology functor to the maps appearing in (\ref{inclusion-fonctions}) and (\ref{troncation}) of section 2.3, we obtain the maps
\begin{equation} \label{cohomologies}
H^*_{cusp}(S, E) \rightarrow H^*_{dR, c}(S, E) \rightarrow H^*_{(2)}(S, E) \rightarrow H^*_{dR}(S, E)
\end{equation}
where $H^*_{cusp}(S, E)$ denotes $H^*(\mathfrak{g}_\mathbb{C}, K_G,  E \otimes_\mathbb{C} \mathcal{C}^\infty_{cusp}(G(\mbb{Q})\backslash G(\mbb{A}), \xi))$. Let $$H^*_{dR, !}(S, E)=\g{Im}(H_{dR, c}^*(S, E) \longrightarrow H^*_{dR}(S, E)).$$

\begin{pro} \label{comparaisons} \cite{mokrane-tilouine} Prop. 1. The maps (\ref{cohomologies}) induce $G(\mbb{A}_f)$-equivariant isomorphisms
$$
H^*_{cusp}(S, E)=H^*_{(2)}(S, E)=H^*_{dR, !}(S, E)=H^3_{dR, !}(S, E).
$$
\end{pro}

The $((\mathfrak{g}_\mathbb{C}, K_G) \times G(\mbb{A}_f))$-module $\mathcal{C}^\infty_{cusp}(G(\mbb{Q})\backslash G(\mbb{A}), \xi)$ decomposes into a direct sum
$$
\mathcal{C}^\infty_{cusp}(G(\mbb{Q})\backslash G(\mbb{A}), \xi)=\bigoplus_{\pi} m(\pi) \pi
$$
indexed by irreducible cuspidal automorphic representations of $G$ , with finite multiplicities. This induces a decomposition
$$
H^3_{dR, !}(S, E)=\bigoplus_{\pi=\pi_\infty \otimes \pi_f} m(\pi) H^3(\mathfrak{g}_\mathbb{C}, K_G, E \otimes_{\mbb{C}} \pi_\infty) \otimes \pi_f
$$
into irreducible $\mbb{C}[G(\mbb{A}_f)]$-modules.

\begin{defn} \label{p(e)} Let $E$ be an irreducible algebraic representation of $G$. The discrete series $L$-packet $P(E)$ associated to $E$ is the set of isomorphism classes of discrete series of $G(\mbb{R})^+$ whose Harish-Chandra parameter and central character are opposed to the ones of $E$.
\end{defn}

\begin{lem} \label{dpp} Assume that $E$ has highest weight $\lambda(k, k', c)$. Then,
$$
P(E)=\{ \pi_\infty^H, \pi_\infty^W, \overline{\pi}_\infty^W, \overline{\pi}_\infty^H \}
$$
and the restrictions of $\pi_\infty^H, \pi_\infty^W, \overline{\pi}_\infty^W$ and $\overline{\pi}_\infty^H $ to $K$ contain as minimal $K$ types the representations $\tau_{(k+3, k'+3)}$, $\tau_{(k+3, -k'-1)}$, $\tau_{(k'+1, -k-3)}$, $\tau_{(-k'-3, -k-3)}$ respectively.
\end{lem}

\begin{proof} If $E$ is irreducible, with highest weight $\lambda(k, k', c)$, then it has infinitesimal character $(k+2, k'+1)$. So, the statement is a direct consequence of Prop. \ref{lpaquet}.
\end{proof}

The main result of \cite{vogan-zuckermann} implies that the $\pi$ contributing to the above sum, i.e. those for which $H^3(\mathfrak{g}_\mathbb{C}, K_G,  \pi_\infty \otimes_\mathbb{C} E)$ is non-zero, are the ones such that $\pi_\infty|_{G(\mbb{R})^+} \in P(E)$ (see the proof of \cite{mokrane-tilouine} Prop. 1 for more details). As a consequence, we have
\begin{equation} \label{semisimple}
H^3_{dR, !}(S, E)= \bigoplus_{\pi=\pi_\infty \otimes \pi_f\,|\,\pi_\infty \in P(E)} H^3_{dR, !}(S, E)(\pi_f) \otimes \pi_f
\end{equation}
where the sum is indexed by irreducible cuspidal automorphic representations of $G$ whose archimedean component belongs to $P(E)$ and where
\begin{equation} \label{isotypes}
H^3_{dR, !}(S, E)(\pi_f)=\bigoplus_{\pi_\infty \in P(E)} m(\pi_\infty \otimes \pi_f) H^3(\mathfrak{g}_\mathbb{C}, K_G,  E \otimes_{\mbb{C}} \pi_\infty).
\end{equation}
According to \cite{borel-wallach} II. \S 3, Prop. 3.1, for any $\pi_\infty \in P(E)$, the $(\mathfrak{g}_\mbb{C}, K_G)$-complex of $\pi_\infty \otimes_\mbb{C} E$ has zero differential. Hence, we have
$$
H^3(\mathfrak{g}_\mathbb{C}, K_G, E \otimes_{\mbb{C}} \pi_\infty)=\mathrm{Hom}_{K_G}\left( \bigwedge^3  \mathfrak{sp}_{4, \mbb{C}}/\mathfrak{k}_\mbb{C}, E \otimes_{\mbb{C}} \pi_\infty \right).
$$
In this last equality, we are using the fact that the inclusion $\mathfrak{sp}_{4, \mbb{C}} \subset \mathfrak{g}_\mbb{C}$ induces an isomorphism $\mathfrak{sp}_{4, \mbb{C}}/\mathfrak{k}_\mbb{C} \simeq \mathfrak{g}_\mbb{C}/(\mathrm{Lie} K_G)_\mbb{C}$. By the Cartan decomposition $\mathfrak{sp}_{4, \mbb{C}}=\mathfrak{k}_\mbb{C} \oplus \mathfrak{p}^+ \oplus \mathfrak{p}^-$, where $\mathfrak{p}^\pm$ are defined by ($\ref{pplus-pmoins}$), we have
$
\bigwedge^3 \mathfrak{sp}_{4, \mbb{C}}/\mathfrak{k}_\mbb{C}= \bigoplus_{p+q=3} \bigwedge^p \mathfrak{p}^+ \otimes_\mbb{C} \bigwedge^q \mathfrak{p}^-.
$
As the weights for the adjoint representation of $\mathfrak{h}_\mbb{C}$ on $ \bigwedge^p \mathfrak{p}^+ \otimes_\mbb{C} \bigwedge^q \mathfrak{p}^-$ are the sums of $p$ distinct weights of $\mathfrak{p}^+$ and of $q$ distinct weights of $\mathfrak{p}^-$, the reader will easily deduce the following decompositions
\begin{eqnarray*}
\bigwedge^3 \mathfrak{p}^+ &=& \tau_{(3, 3)},\\
\bigwedge^2 \mathfrak{p}^+ \otimes_\mbb{C} \mathfrak{p}^- &=& \tau_{(3, -1)} \oplus \tau_{(2, 0)} \oplus \tau_{(1, 1)}\\
\mathfrak{p}^+ \otimes_\mbb{C} \bigwedge^2 \mathfrak{p}^- &=& \tau_{(1, -3)} \oplus \tau_{(0, -2)} \oplus \tau_{(-1, -1)},\\
\bigwedge^3 \mathfrak{p}^- &=& \tau_{(-3, -3)}\\
\end{eqnarray*}
into irreducible $\mbb{C}[K]$-modules  from the basic facts on irreducible representations of $K$ reviewed in section \ref{discrete_series_classification}.  If $E$ has highest weight $\lambda(k, k', c)$, its infinitesimal character is $(k+2, k'+1)$.

\begin{pro} \label{dimension} The $\mbb{C}$-vector spaces
\begin{eqnarray*}
H^3(\mathfrak{g}_\mathbb{C}, K_G,  \pi_\infty^H \otimes_\mathbb{C} E) &=& \mathrm{Hom}_{K_G}\left( \bigwedge^3  \mathfrak{sp}_{4, \mbb{C}}/\mathfrak{k}_\mbb{C}, E \otimes_{\mbb{C}} \pi_\infty^H \right),\\
H^3(\mathfrak{g}_\mathbb{C}, K_G,  \pi_\infty^W \otimes_\mathbb{C} E) &=& \mathrm{Hom}_{K_G}\left( \bigwedge^3  \mathfrak{sp}_{4, \mbb{C}}/\mathfrak{k}_\mbb{C}, E \otimes_{\mbb{C}} \pi_\infty^W \right),\\
H^3(\mathfrak{g}_\mathbb{C}, K_G,  \overline{\pi}_\infty^W \otimes_\mathbb{C} E) &=& \mathrm{Hom}_{K_G}\left( \bigwedge^3  \mathfrak{sp}_{4, \mbb{C}}/\mathfrak{k}_\mbb{C}, E \otimes_{\mbb{C}} \overline{\pi}_\infty^W \right),\\
H^3(\mathfrak{g}_\mathbb{C}, K_G,  \overline{\pi}_\infty^H \otimes_\mathbb{C} E) &=& \mathrm{Hom}_{K_G}\left( \bigwedge^3  \mathfrak{sp}_{4, \mbb{C}}/\mathfrak{k}_\mbb{C}, E \otimes_{\mbb{C}} \overline{\pi}_\infty^H \right)
\end{eqnarray*}
have dimension one.
\end{pro}

\begin{proof}
This is a particular case of \cite{borel-wallach} II. Prop. 3.1 and Thm. 5.3.
\end{proof}

From now on, let us assume that the irreducible algebraic representation $E$ takes values in a finite dimensional $\mbb{Q}$-vector space. In particular, the results explained above can be applied to the complexification $E_\mbb{C}$ of $E$. For any cuspidal automorphic representation $\pi =\pi_\infty \otimes \pi_f$ of $G$ such that $\pi_\infty \in P(E)$, denote again by $\pi_f$ the model of $\pi_f$ defined over its rationality field, which is the number field $E(\pi_f)$ (see Prop. \ref{bhr}). Following \cite{harris2} 2.6.2 and 2.6.3, define
\begin{equation*}
M_{dR}(\pi_f, E)=\mathrm{Hom}_{\mbb{Q}[G(\mbb{A}_f)]}(Res_{E(\pi_f)/\mbb{Q}}\pi_f, H^3_{dR, !}(S, E)).
\end{equation*}
and similarly
\begin{equation*}
M_{B}(\pi_f, E)=\mathrm{Hom}_{\mbb{Q}[G(\mbb{A}_f)]}(Res_{E(\pi_f)/\mbb{Q}}\pi_f, H^3_{B, !}(S, E))
\end{equation*}
where $H^3_{B, !}(S, E)$ denotes the Betti cohomology with coefficients in $E$. These are $\mbb{Q}$-vector spaces endowed with a $\mbb{Q}$-linear action of $E(\pi_f)$. Extending scalars from $\mbb{Q}$ to $\mbb{C}$, according to the identities (\ref{semisimple}) and (\ref{isotypes}), we have
\begin{equation*}
M_{dR}(\pi_f, E)_\mbb{C}= \bigoplus_{\sigma: E(\pi_f) \rightarrow \mbb{C}} \bigoplus_{\pi_\infty \in P(E)} m(\pi_\infty \otimes \pi_f) H^3(\mathfrak{g}_\mathbb{C}, K_G,  E_\mbb{C} \otimes_\mathbb{C}  \pi_\infty ).
\end{equation*}
For any $\pi_f$, the $\mbb{C}$-vector space
$$
\bigoplus_{\pi_\infty \in P(E)} m(\pi_\infty \otimes \pi_f) H^3(\mathfrak{g}_\mathbb{C}, K_G, E_\mbb{C} \otimes_\mathbb{C}  \pi_\infty),
$$
has finite dimension $$m(\pi_\infty^H \otimes \pi_f)+m(\pi_\infty^W \otimes \pi_f)+m(\overline{\pi}_\infty^W \otimes \pi_f)+m(\overline{\pi}_\infty^H \otimes \pi_f)$$
according to Prop. \ref{dimension}. As a consequence, the dimension of $M_{dR}(\pi_f, E)$ as an $E(\pi_f)$-vector space is
$$
m(\pi_\infty^H \otimes \pi_f)+m(\pi_\infty^W \otimes \pi_f)+m(\overline{\pi}_\infty^W \otimes \pi_f)+m(\overline{\pi}_\infty^H \otimes \pi_f).
$$
The same holds for $M_{B}(\pi_f, E)$ because of the comparison isomorphism
$$
\begin{CD}
I_\infty: M_{B}(\pi_f, E)_\mbb{C} @>\sim>>  M_{dR}(\pi_f, E)_\mbb{C}.
\end{CD}
$$

It follows from Saito's formalism of mixed Hodge modules that for any open compact subgroup $L$ of $G(\mbb{A}_f)$, the interior cohomology $H^3_!(S_L, E)$ underlies a $\mbb{Q}$-Hodge structure. With the convention adopted here, and explained in section 2.4, this $\mbb{Q}$-Hodge structure is in fact pure of weight $3-c$, where $x \longmapsto x^c$ is the central character of $E$. Hence, we obtain a pure $\mbb{Q}$-Hodge structure on $M_B(\pi_f, E)$:

\begin{pro} \label{dec-hodge} Let $t=\frac{c-k-k'}{2}$. The Hodge decomposition of $M_B(\pi_f, E)$ is
$$
M_B(\pi_f, E)_\mbb{C}=M_B^{3-t, -k-k'-t} \oplus M_B^{2-k'-t, 1-k-t} \oplus M_B^{1-k-t, 2-k'-t} \oplus M_B^{-k-k'-t, 3-t}
$$
where
\begin{eqnarray*}
M_B^{3-t, -k-k'-t} &=& \bigoplus_{\sigma: E(\pi_f) \rightarrow \mbb{C}} m(\pi_\infty^H \otimes \pi_f) H^3(\mathfrak{g}_\mathbb{C}, K_G,  E \otimes_{\mbb{C}} \pi_\infty^H ),\\
M_B^{2-k'-t, 1-k-t} &=& \bigoplus_{\sigma: E(\pi_f) \rightarrow \mbb{C}} m(\pi_\infty^W \otimes \pi_f) H^3(\mathfrak{g}_\mathbb{C}, K_G,  E \otimes_{\mbb{C}} \pi_\infty^W ),\\
M_B^{1-k-t, 2-k'-t} &=& \bigoplus_{\sigma: E(\pi_f) \rightarrow \mbb{C}} m(\overline{\pi}_\infty^W \otimes \pi_f) H^3(\mathfrak{g}_\mathbb{C}, K_G,  E \otimes_{\mbb{C}} \overline{\pi}_\infty^W ),\\
M_B^{-k-k'-t, 3-t} &=& \bigoplus_{\sigma: E(\pi_f) \rightarrow \mbb{C}} m(\overline{\pi}_\infty^H \otimes \pi_f) H^3(\mathfrak{g}_\mathbb{C}, K_G,  E \otimes_{\mbb{C}} \overline{\pi}_\infty^H ).
\end{eqnarray*}
\end{pro}

\begin{proof} A reference for the statement on Hodge types is \cite{occult} (1.4). However, let us remark that \cite{occult} uses the sign convention opposite to ours (see section 2.4). The result then follows from the fact that  the Hodge decomposition of $L^2$ cohomology given by harmonic forms coincides with the Hodge decomposition given by the theory of mixed Hodge modules (Prop. \ref{comparaisons} and \cite{boundarycohomology} Thm. 5.4).
\end{proof}

The following definition is taken from \cite{beilinson2} \S 7.

\begin{defn} Let $A$ be a subfield of $\mathbb{R}$. A real mixed $A$-Hodge structure is a mixed $A$-Hodge structure whose underlying $A$-vector space is endowed with an involution $F_\infty$ stabilizing the weight filtration and whose $\mbb{C}$-antilinear complexification $F_\infty \otimes \tau$, where $\tau$ denotes the complex conjugation, stabilizes de Hodge filtration. We call $F_\infty \otimes \tau$ the de Rham involution.
\end{defn}

Let $\mathrm{MHS}_A^+$ denote the abelian category of real mixed $A$-Hodge structures.

\begin{defn} \label{mhs-coeff} Let $F$ be a number field. A real mixed $A$-Hodge structure with coefficients in $F$ is a pair $(M, s)$ where $M$ is an object of $\mathrm{MHS}_A^+$  and $s: F \lra \mrm{End_{\mathrm{MHS}_A^+}}(M)$ is a ring homomorphism.
\end{defn}

Let $\mathrm{MHS}_{A, F}^+$ denote the abelian category of real mixed $A$-Hodge structures with coefficients in $F$.

\begin{pro} \label{hodge-deRham-pif} Let $F_\infty$ be the involution on $M_B(\pi_f, E)$ induced by the complex conjugation on $S(\mbb{C})$ and on $E$. Then $(M_B(\pi_f, E), F_\infty)$ is an object of $\mathrm{MHS}_{\mbb{Q}, E(\pi_f)}^+$.
\end{pro}

\begin{pro} \label{ratF} \cite{harris1} Cor. 2.3.1. The Hodge filtration $F^* M_B(\pi_f, E)_\mbb{C}$, which is defined by
$$
F^p M_B(\pi_f, E)_\mbb{C}=\bigoplus_{p' \geq p} M_B^{p', q},
$$
is the image of the complexification of a filtration  $F^* M_{dR}(\pi_f, E)$ of $M_{dR}(\pi_f, E)$ under the comparison isomorphism $I_\infty^{-1}$.
\end{pro}

\begin{pro} \label{milne-shih} Let $N \in G(\mbb{R})$ be as in Rem. \ref{conj-ds}. Then, the involution $F_\infty$ of $M_B(\pi_f, E)_\mbb{C}$ is induced by the action of $N$ on $C^*(\mathfrak{g}_\mbb{C}, K_G, E \otimes_\mbb{C} \mathcal{C}^\infty(G(\mbb{Q}) \backslash G(\mbb{A})))$ defined by
$$
f \in \mrm{Hom}_K\left(\bigwedge^* \mathfrak{g}_\mbb{C} / \mathfrak{k}_\mbb{C}, E \otimes_\mbb{C} \mathcal{C}^\infty(G(\mbb{Q}) \backslash G(\mbb{A}))  \right) \longmapsto \left( X \longmapsto Nf(\mrm{Ad}_N(X))\right).
$$
\end{pro}

\begin{proof} Note that
$$
N=\begin{pmatrix}
  & 1 & &\\
1 &   & &\\
   &   & & 1\\
   &   & 1 & \\
\end{pmatrix} \begin{pmatrix}
-1 & &  & \\
 & -1 &  & \\
 &  & 1 & \\
 &   &  & 1\\
\end{pmatrix}.
$$
As the cocharacter $\mu: \mbb{G}_{m, \mbb{C}} \lra G_\mbb{C}$ associated to the morphism $h: \mbb{S} \lra G_\mbb{R}$ defined above induces
$$
z \longmapsto \begin{pmatrix}
z & &  & \\
 & z &  & \\
 &  & 1 & \\
 &   &  & 1\\
\end{pmatrix}
$$
on complex points, the matrix $N$ satisfies the assumptions of \cite{milne-shih} Lem. 3.2. So the statement follows from the proof of Langlands conjecture on the action of complex conjugation on the set of complex points of a Shimura varieties which, in the case of Siegel varieties, is explained in \cite{milne-shih} Rem. 3.3 (c).
\end{proof}

\section{Computation of the regulator} \label{sect-calc-reg}

Let $p, q, k$ and $k'$ be integers as in the introduction and let $\mathcal{B}_p$, respectively $\mathcal{B}_q$, be the source of Beilinson's Eisenstein symbol (\ref{symbole-eisenstein}) of weight $p$, respectively $q$. Assume that $p, q, k$ and $k'$ satisfy the assumptions of Thm. \ref{lemma1}. Then, we have the extension class
$$
Eis_\mathcal{H}^{p, q, W}(\phi_f \otimes \phi'_f) \in \mathrm{Ext}^1_{\mathrm{MHS}_\mathbb{R}^+}(\mathbb{R}(0), H^3_!(S, W)_\mbb{R}) \subset H^4_\mathcal{H}(S/\mbb{R}, W)
$$
for any $\phi_f \otimes \phi'_f \in \mathcal{B}_p \otimes_\mbb{Q} \mathcal{B}_q$. Let $Eis_\mathcal{D}^{p, q, W}(\phi_f \otimes \phi'_f)$ denote the image of $Eis_\mathcal{H}^{p, q, W}(\phi_f \otimes \phi'_f)$ by the natural map
$$
\begin{CD}
H^4_\mathcal{H}(S/\mbb{R}, W) @>>> H^4_\mathcal{D}(S/\mbb{R}, W)
\end{CD}
$$
from absolute Hodge to Deligne-Beilinson cohomology (\cite{beilinson2} 5.7, \cite{jannsen} $\S$ 2). Thanks to the work of Jannsen \cite{jannsen}, the Deligne-Beilinson cohomology groups can be explicitly described by pairs of currents $(S, T)$ (Prop. \ref{explicitDBhomology}) and this will permit us to give an explicit description of $Eis_\mathcal{D}^{p, q, W}(\phi_f \otimes \phi'_f)$ (Prop. \ref{eisenstein-deligne}). In section \ref{section-Poincare}, we introduce Deligne rational structure $\mathcal{D}$ on $\mrm{Ext}^1$ and explain how the image $\mathcal{K}$ of $Eis_\mathcal{H}^{p, q, W}$ can be compared to $\mathcal{D}$ via the computation of a Poincar\'e duality pairing related to $\mathcal{K}$ and another related to $\mathcal{D}$  (Lem. \ref{dualite}). This idea goes back to Beilinson \cite{beilinson} (see also \cite{kings} 6.1). In section \ref{regulateur-integrale}, the pairing related to $\mathcal{K}$ is shown to be equal to an explicit adelic integral. Before this, let us introduce some relative motives.

\subsection{The relative motives} \label{sect-relatives-motives} In this section, we give a definition of motivic cohomology of the Shimura varieties we are interested in, with coefficients in sheaves such as $Sym^p V_2 \boxtimes Sym^q V_2$ or $W$. This relies on the work of Ancona \cite{ancona} and Cisinski-D\'eglise \cite{cisinski-deglise}, whose ideas were initiated by Voevodsky et al. \cite{voevodsky}. Ideally, motivic cohomology should be defined as a space of extensions in categories of mixed motivic sheaves, but these categories have not been discovered yet.\\

For the necessary background on relative motives of abelian schemes and relative Weil cohomologies, see \cite{ancona} section 2 and 3 respectively. Let $A/S$ be the universal abelian surface over the Siegel threefold and let $A^{k+k'}$ be the $(k+k')$-th fold fiber product over $S$. Let $R(A^{k+k'})$ denote the relative motive over $S$ associated to $A^{k+k'}$.

\begin{pro} \label{motifdeW} Let $k, k'$ and $c$ be such that $k \geq k' \geq 0$ and $c \equiv k+k' \pmod 2$. Let $t$ denote the integer $(k+k'+c)/2.$ Let $W$ be an irreducible  algebraic representation of $G$ of highest weight $\lambda(k, k', c)$. Then, there exists a relative Chow motive over the Siegel threefold which is a direct factor of $R(A^{k+k'})(t)$ and whose Betti realization is the variation of Hodge structure $W$.
\end{pro}

\begin{proof} Let $r: G \longrightarrow \mrm{GL}(4)$ be the standard representation of $G$. It follows from Weyl's invariants theory that $W$ is a direct factor, defined by a certain explicit Schur projector, of the representation $r^{\otimes (k+k')} \otimes \nu^{\otimes t}$ (see \cite{fulton-harris} $\S 17.3$). Hence, the statement is a direct consequence of \cite{ancona} Thm. 1.3.
\end{proof}

Let $\mrm{DM}_{B, c}(S)$ is the triangulated category of constructible Beilinson motives over $S$ with $\mbb{Q}$-coefficients as defined in \cite{cisinski-deglise} Def. 15.1.1. Then $W$ is an object of $\mrm{DM}_{B, c}(S)$. Let $\textbf{1}_S$ be the unit object of $\mrm{DM}_{B, c}(S)$. Motivic cohomology with coefficients in $W$ is the $\mbb{Q}$-vector space defined by
$$
H^*_\mathcal{M}(S, W)=\mrm{Hom}_{\mrm{DM}_{B, c}(S)}(\textbf{1}_S, W[*]).
$$
The compatibility of this definition with the $K$-theoretical one follows from \cite{cisinski-deglise} Cor. 14.2.14. For integers $p \geq 0$ and $q \geq 0$, we can define $H^*_\mathcal{M}(M \times M, (Sym^p V_2 \boxtimes Sym^q V_2)(2))$ similarly. Here $M$ denotes a modular curve. Moreover, if $p, q, k$ and $k'$ verify the conditions stated in the introduction, the relative motive $(Sym^p V_2 \boxtimes Sym^q V_2)(2)$ over $M \times M$ is naturally a direct factor of $\iota^*W$, as Ancona's construction is functorial. As a consequence, the motivic cohomology space $H^*_\mathcal{M}(M \times M, (Sym^p V_2 \boxtimes Sym^q V_2)(2))$ is a naturally a direct factor of $H^*_\mathcal{M}(M \times M, \iota^*W(-1))$. As the triangulated category of constructible Beilinson motives has the formalism of Grothendieck six functors, duality (\cite{cisinski-deglise} Thm. 15.2.4) and the absolute purity isomorphism (\cite{lehalleur} Prop. 1.7), we have the Gysin morphism
$$
\begin{CD}
H^*_\mathcal{M}(M \times M, \iota^*W(-1)) @>>> H^{*+2}_\mathcal{M}(S, W).
\end{CD}
$$
In this setting, the definition of the regulator in Deligne-Beilinson cohomology and the compatibility with the previous one has been explained in \cite{scholbach}.

\subsection{Explicit description of the cohomology classes} \label{sousectionexplicit} Let us start by reviewing basic facts about Deligne-Beilinson cohomology. In what follows, we shall consider Deligne-Beilinson cohomology with coefficients in algebraic representations of the group underlying a given Shimura variety. This means that, like in \cite{kings} 2.3, we consider Deligne-Beilinson cohomology of the corresponding relative motives.\\

Let $Sch(\mbb{Q})$ be the category of smooth quasi-projective $\mbb{Q}$-schemes. Let $X$ be an object of $Sch(\mbb{Q})$ and let $n$ be an integer. For a definition of the real Deligne-Beilinson cohomology $H^*_\mathcal{D}(X/\mathbb{R}, \mbb{R}(n))$, the reader is referred to \cite{nekovar} 7. We also have the real absolute Hodge cohomology $H^*_\mathcal{H}(X/\mathbb{R}, \mbb{R}(n))$ of $X$ with coefficients in $\mbb{R}(n)=(2\pi i)^n \mbb{R}$ as defined in \cite{huber-wildeshaus1} Def. A.2.6 and there is a canonical map
\begin{equation}
\begin{CD}
H^m_\mathcal{H}(X/\mathbb{R}, \mbb{R}(n)) @>>> H^m_\mathcal{D}(X/\mathbb{R}, \mbb{R}(n)).
\end{CD}
\end{equation}
Let $\mathcal{S}^m(X/\mbb{R}, \mbb{R}(n))$ be the vector space of $\mathcal{C}^\infty$ differential forms on  $X(\mbb{C})$ on which the map $F_\infty$ induced by complex conjugation on $X(\mbb{C})$ acts by multiplication by $(-1)^n$.  Let $\mathcal{S}^m_c(X/\mbb{R}, \mbb{R}(n))$ be the compactly supported differential forms belonging to $\mathcal{S}^m(X/\mbb{R}, \mbb{R}(n))$. Let us consider a smooth projective compactification $j: X \longrightarrow X^*$ and let $i: Y \longrightarrow X^*$ be the complementary reduced closed embedding. Assume that $Y$ is a normal crossing divisor and let $\Omega^m_X \langle Y \rangle$ be the $\mbb{C}$-vector space of holomorphic differentials of degree $m$ on $X$ with logarithmic singularities along $Y$, endowed with its Hodge filtration (see \cite{deligne1} 3.1 and 3.2.2). For any integer $n$, let $\pi_n: \mathbb{C} \longrightarrow \mathbb{R}(n)$ denote the map $z \longmapsto  \frac{1}{2}(z+(-1)^n\overline{z})$.

\begin{pro} \label{explicitDB} \cite{nekovar} 7.3. For any integer $m$, we have a canonical isomorphism of $\mbb{R}$-vector spaces
$$
H^{m}_\mathcal{D}(X/\mathbb{R}, \mbb{R}(m))=\frac{\{(\phi, \omega) \in \mathcal{S}^{m-1}(X/\mbb{R}, \mbb{R}(m)) \times  \Omega^m_X \langle Y \rangle \,|\, d\phi=\pi_m(\omega)\}}{d\mathcal{S}^{m-2}(X/\mbb{R}, \mbb{R}(m))}.
$$
\end{pro}

The Eisenstein symbol (\cite{beilinson} \S 3) is a $\mbb{Q}$-linear map 
$$
\begin{CD}
Eis^n_\mathcal{M}: \mathcal{B}_n @>>> H^{n+1}_\mathcal{M}(E^n, \mbb{Q}(n+1)) 
\end{CD}
$$
where we denote by $H^{n+1}_\mathcal{M}(E^n, \mbb{Q}(n+1)) $ the inductive limit over compact open subgroups $K$ of $\mathrm{GL}_2(\mathbb{A}_f)$ of the motivic cohomology $H^{n+1}_\mathcal{M}(E^n_K, \mbb{Q}(n+1)) $ of the $n$-th fold fiber product of the universal elliptic curve $E_K/M_K$ over the modular curve of level $K$. The notation $\mathcal{B}_n$ stands for the space of locally constant $\mbb{Q}$-valued functions $\phi_f$ on $\mrm{GL}_2(\mbb{A}_f)$ such that for any $b \in \mbb{A}_f$, any $a, d \in \mbb{Q}$ such that $ad>0$ and any $g \in \mrm{GL}_2(\mbb{A}_f)$, we have
\begin{eqnarray}
\phi_f \left( \begin{pmatrix}
a & b\\
 & d\\
\end{pmatrix}g \right) &=& a^{-(n+1)}d \phi_f(g), \label{transformation1}\\
\phi_f \left( \begin{pmatrix}
-1 & \\
 & 1\\
\end{pmatrix}g \right) &=& \phi_f(g) \label{transformation2}
\end{eqnarray}
and which are invariant by right translation under $\begin{pmatrix}
\widehat{\mbb{Z}}^\times & \\
  & 1\\
\end{pmatrix}$. The source of the Eisenstein symbol is indentified to a space of $\mbb{Q}$-valued functions $\mathcal{F}^n$ on $\mrm{GL}_2(\mbb{A}_f)$ in \cite{beilinson} p.7. Note that the map $\psi_f \mapsto (\phi_f: g \mapsto \phi_f(g)=\psi_f(det(g)^{-1}\,^t\!g)$ defines a $\mrm{GL}_2(\mbb{A}_f)$-equivariant isomorphism $\mathcal{F}^n \simeq \mathcal{B}_n$ when $\mrm{GL}_2(\mbb{A}_f)$ acts on $\mathcal{F}^n$, resp. $\mathcal{B}_n$, by left, resp. right, translation. By definition of the motivic sheaf $\mrm{Sym}^n V_2(1)$, the motivic cohomology $H^1_\mathcal{M}(M, \mrm{Sym}^n V_2(1))$ is a direct factor of $H^{n+1}_\mathcal{M}(E^n, \mbb{Q}(n+1))$. Furthermore, the Eisenstein symbol factors through the natural inclusion $H^1_\mathcal{M}(M, \mrm{Sym}^n V_2(1)) \subset H^{n+1}_\mathcal{M}(E^n, \mbb{Q}(n+1))$ and we denote again by $Eis^n_\mathcal{M}$ the induced map $\mathcal{B}_n \longrightarrow H^1_\mathcal{M}(M, \mrm{Sym}^n V_2(1))$. The following lemma will be very useful later.

\begin{lem} \label{eiseinstein-ext-scalaires} Let $\mathcal{B}_{n, \overline{\mbb{Q}}}$ be the vector space $\mathcal{B}_n \otimes_\mbb{Q} \overline{\mbb{Q}}$ with the action of $\mathrm{GL}_2(\mathbb{A}_f)$ by right translation. For any finite order Hecke character $\nu$ let $\mathcal{I}_n(\nu)$ denote the space of locally constant functions $f: \mathrm{GL}_2(\mathbb{A}_f) \longrightarrow \overline{\mbb{Q}}$ such that
$$
\forall a ,d \in \mbb{Q}^\times_+, \forall \alpha, \delta \in \widehat{\mbb{Z}}^\times, \forall b \in \mbb{A}_f, \forall g \in  \mathrm{GL}_2(\mathbb{A}_f), f \left( \begin{pmatrix}
a \alpha & b\\
 & d \delta\\
\end{pmatrix}g\right)=a^{-(n+1)}d \nu(\delta)f(g)
$$
and which are invariant by $\begin{pmatrix}
\widehat{\mbb{Z}}^\times & \\
  & 1\\
\end{pmatrix}$ when $\mathcal{I}_n(\nu)$ is endowed with the action of $\mathrm{GL}_2(\mathbb{A}_f)$ by right translation. Then, there is a $\mathrm{GL}_2(\mathbb{A}_f)$-equivariant decomposition
$$
\mathcal{B}_{n, \overline{\mbb{Q}}}= \bigoplus_{\mrm{sgn}(\nu)=(-1)^n} \mathcal{I}_n(\nu)
$$
where the sum is indexed by all finite order Hecke characters $\nu$ of sign $(-1)^n$.
\end{lem}

\begin{proof} Let $\mrm{T}_2$ denote the diagonal maximal torus of $\mrm{GL}_2$. We are interested in the action of $\mrm{T}_2(\mbb{A}_f)$ on the space $\mathcal{B}_{n, \overline{\mbb{Q}}}$ by left translation. Because of the decomposition $\mbb{A}_f^\times=\mbb{Q}^\times_+ \widehat{\mbb{Z}}^\times$, we are reduced to study the action of $\mrm{T}_2(\widehat{\mbb{Z}})$ thanks to the equality (\ref{transformation1}). The Iwasawa decomposition $\mathrm{GL}_2(\mathbb{A}_f)=\mrm{B}_2(\mbb{A}_f) \mrm{GL}_2(\widehat{\mbb{Z}})$ and the fact that functions in $\mathcal{B}_{n, \overline{\mbb{Q}}}$ are locally constant implies that $\mathcal{B}_{n, \overline{\mbb{Q}}}$ is a union of finite dimensional $\overline{\mbb{Q}}$-vector spaces $V$ which are stable under $\mrm{T}_2(\widehat{\mbb{Z}})$. By continuity, the action is trivial on an open subgroup of $\mrm{T}_2(\widehat{\mbb{Z}})$ and so $V$ is a sum of finite order characters
$$
\begin{pmatrix}
 \alpha & \\
 & \delta\\
\end{pmatrix} \longmapsto \chi(\alpha) \nu(\delta).
$$
By $\begin{pmatrix}
\widehat{\mbb{Z}}^\times & \\
  & 1\\
\end{pmatrix}$-invariance,  the character $\chi$ has to be trivial and by condition (\ref{transformation1}), the character $\nu$ has to be of sign $(-1)^n.$
\end{proof}

Let us recall now the explicit description of the image of the Eisenstein symbol in Deligne-Beilinson cohomology, via the isomorphism of Prop. \ref{explicitDB}. Here, we follow \cite{kings} 6.3 very closely.\\

Let us regard the circle $\mrm{U}(1)$ as a maximal compact subgroup of $\mathrm{SL}_2(\mathbb{R})$ in the usual way. We will also sometimes consider the maximal torus $Z_2(\mbb{R})^+ \mrm{U}(1)$, where $Z_2(\mbb{R})^+$ denotes the identity component of the center of $\mathrm{GL}_2(\mbb{R})$. In what follows, we shall implicitly use the isomorphism between the de Rham complex on modular curves and a $(\mathfrak{gl}_{2, \mbb{C}}, Z_2(\mbb{R})^+ \mrm{U}(1))$-complex which is analogous to the one stated in section \ref{coho-siegel} for Siegel threefolds. If $\mathfrak{k}_2$ denotes the Lie algebra of $\mrm{U}(1)$, we have the Cartan decomposition $\mathfrak{sl}_{2, \mathbb{C}}=\mathfrak{k}_{2, \mathbb{C}} \oplus \mathfrak{p'}^+ \oplus \mathfrak{p'}^-$ where
$$
\mathfrak{p'}^\pm = \left\{ \begin{pmatrix}
z & \pm iz\\
\pm iz & -z\\
\end{pmatrix} \in \mathfrak{sl}_{2, \mathbb{C}} \,|\, z \in \mathbb{C} \right\}.
$$
Let $v^\pm \in \mathfrak{p'}^\pm$ denote the vector $v^\pm=\frac{1}{2}\begin{pmatrix}
1 & \pm i\\
\pm i & -1\\
\end{pmatrix} \in \mathfrak{p'}^\pm$. Let $(X, Y)$ be a basis of the standard representation $V_2$ of $\mrm{GL}(2)$ such that $\begin{pmatrix}
a & b\\
c & d\\
\end{pmatrix} \in \mathrm{GL}_2(\mathbb{Q})$ acts by
\begin{eqnarray*}
\begin{pmatrix}
a & b\\
c & d\\
\end{pmatrix}X=aX+cY,\\
\begin{pmatrix}
a & b\\
c & d\\
\end{pmatrix}Y=bX+dY.
\end{eqnarray*}
We regard  $\mathrm{Sym}^n V_{2, \mathbb{C}}$ as the space of homogeneous polynomials of degree $n$ in the variables $X$ and $Y$, with coefficients in $\mathbb{C}$. For any integer $0 \leq j \leq n$, let $b_j^{(n)}  \in \mathrm{Sym}^n V_{2, \mathbb{C}}$ be the vector $b_j^{(n)}=(iX-Y)^{j} (iX+Y)^{n-j}$. The family $\left( b_j^{(n)}\right)_{0 \leq j \leq n}$ is a basis of $\mathrm{Sym}^n V_{2, \mathbb{C}}$ and we will denote by $\left( a_j^{(n)}\right)_{0 \leq j \leq n}$ the dual basis.

\begin{lem} \label{poinds-a_j} Let $n, c$ be two integers such that $n \equiv c \pmod 2$. Let $\lambda'(n, c)$ be the character of the torus $Z_2(\mathbb{R})^+ \mrm{U}(1)$ defined by
$$
\lambda'(n, c): \begin{pmatrix}
x & y\\
-y & x\\
\end{pmatrix} \longmapsto (x+iy)^n (x^2+y^2)^\frac{c-n}{2}.
$$
Then, the vector $a_j^{(n)}$ has weight $\lambda'(n-2j, -n)$.
\end{lem}

\begin{proof} This follows from a trivial computation.
\end{proof}

Let $\mrm{B}_2$ denote the standard Borel of $\mrm{GL}_2$, and by $\mrm{GL}_2(\mbb{R})^+$, respectively $\mrm{B}_2(\mbb{R})^+$, the identity component of $\mrm{GL}_2(\mbb{R})$, respectively $\mrm{B}_2(\mbb{R})$. Given $n, c$ as above, we denote by $\lambda(n, c)$ the algebraic character of the diagonal maximal torus of $\mrm{GL}_2$ defined by
$$
\lambda(d, c): \begin{pmatrix}
\alpha & \\
 & \alpha^{-1} \nu\\
\end{pmatrix} \longmapsto \alpha^n \nu^\frac{c-n}{2}.
$$
For any integer $r$ such that $r \equiv n \pmod 2$, the function $\phi^{n}_{r} \in \mrm{ind}_{\mrm{B}_2(\mbb{R})^+}^{\mrm{GL}_2(\mbb{R})^+} \lambda(n+2, n)$ is defined as
$$
\phi^{n}_{r}=(z_2-\overline{z}_2)^{n+1} w_2^\frac{n-r}{2} \overline{w}_2^\frac{n+r}{2}
$$
where
$$
z_2\begin{pmatrix}
a & b\\
c & d\\
\end{pmatrix}=\frac{ai+b}{ci+d}, w_2\begin{pmatrix}
a & b\\
c & d\\
\end{pmatrix}=ci+d.
$$
Consider
$$
\Theta_n = \frac{(2\pi  i)^{n+1}}{2(n+1)} \sum_{j=0}^n  a_j^{(n)} \otimes \phi^{n}_{2j-n} \in \mathrm{Sym}^n V_{2, \mathbb{C}}^\vee \otimes \mrm{ind}_{\mrm{B}_2(\mbb{R})^+}^{\mrm{GL}_2(\mbb{R})^+} \lambda(n+2, n).
$$
Given $\phi_f \in \mathcal{B}_n$, we can consider the vector valued series of functions
$$
Eis^n_\mathcal{H}(\phi_f)=\sum_{\gamma \in \mrm{B}_2(\mbb{Q}) \backslash \mrm{GL}_2(\mbb{Q})} \gamma^*(\Theta_n \otimes \phi_f),
$$
which is absolutely convergent provided $n \geq 1$. Because of the limitations given by the hypothesis of Thm. \ref{lemma1}, we only need to work with the Eisenstein symbol $Eis^n_\mathcal{M}$ for $n \geq 1$ in this work.

\begin{lem} \label{cc_omega}  Let $$\omega_n^\pm \in \mrm{Hom}_{\mrm{U}(1)} \left( \mathfrak{p'}^+ \oplus \mathfrak{p'}^-, \mathrm{Sym}^n V_{2, \mathbb{C}}^\vee \otimes \mrm{ind}_{\mrm{B}_2(\mbb{R})^+}^{\mrm{GL}_2(\mbb{R})^+} \lambda(n+2, n)\right)$$ be the differential form defined by
\begin{eqnarray*}
\omega_n^+(v^+) &=& (2 \pi i)^{n+1} a_n^{(n)} \otimes \phi_{n+2}^n ,\\
\omega_n^+(v^-) &=& 0,\\
\omega_n^-(v^+) &=& 0,\\
\omega_n^-(v^-) &=& (2 \pi i)^{n+1} a_0^{(n)} \otimes \phi_{-n-2}^n.
\end{eqnarray*}
Let $\overline{\omega}_n^\pm$ be the differential form deduced from $\omega_n^\pm$ by applying the complex conjugation on $\mathrm{Sym}^n V_{2, \mathbb{C}}^\vee \otimes \mrm{ind}_{\mrm{B}_2(\mbb{R})^+}^{\mrm{GL}_2(\mbb{R})^+} \lambda(n+2, n)$. Then $\overline{\omega}_n^\pm=(-1)^n \omega_n^\mp$.
\end{lem}

\begin{proof} This follows from the identities $\overline{\phi}_{n+2}^n=(-1)^{n+1} \phi_{-n-2}^n$ and $\overline{a_j^{(n)}}=(-1)^n a_j^{(n)}$.
\end{proof}

For any $\phi_f \in \mathcal{B}_n$, the infinite series
$$
Eis_B^n(\phi_f)=\sum_{\gamma \in \mrm{B}_2(\mbb{Q}) \backslash \mrm{GL}_2(\mbb{Q})} \gamma^*(\omega_n^+ \otimes \phi_f)
$$
is absolutely convergent and defines a vector valued closed holomorphic differential one-form on $M$, i.e. en element of
$$
\mathrm{Hom}_{\mrm{U}(1)}\left( \bigwedge^1(\mathfrak{p'}^+ \oplus \mathfrak{p'}^-),  \mathrm{Sym}^n V_{2, \mathbb{C}}^\vee \otimes \mathcal{C}^\infty(\mathrm{GL}_2(\mathbb{Q}) \backslash \mathrm{GL}_2(\mathbb{A}_f)) \right),
$$
as explained in \cite{kings} 6.3. Let $E \lra M$ denote the universal elliptic curve over $M$. By definition of the relative motive associated to the standard representation $V_2$ of $\mrm{GL}(2)$, the cohomology $H^1_\mathcal{D}(M/\mbb{R}, \mathrm{Sym}^n V_2(1))$ is a direct factor of $H^{n+1}_\mathcal{D}(E^n/\mbb{R}, \mbb{R}(n+1))$, where $E^n$ denotes the $n$-th fold fiber product over $M$.

\begin{pro} \label{explicitEisenstein} \cite{kings} (6.3.5). Let $Eis^n_\mathcal{D}: \mathcal{B}_n \longrightarrow H^1_\mathcal{D}(M/\mbb{R}, \mathrm{Sym}^n V_2(1))$ be the composite of $Eis^n_\mathcal{M}$ and of the regulator
$$
\begin{CD}
H^1_\mathcal{M}(M, \mathrm{Sym}^n V_2(1)) @>>> H^1_\mathcal{D}(M/\mbb{R}, \mathrm{Sym}^n V_2(1)).
\end{CD}
$$
If $n \geq 1$ then for any $\phi_f \in \mathcal{B}_n$, the class of $Eis^n_\mathcal{D}(\phi_f)$ is represented by $(Eis^n_\mathcal{H}(\phi_f), Eis^n_B(\phi_f))$ via the isomorphism of Prop. \ref{explicitDB}.
\end{pro}

By functoriality of the natural map between absolute Hodge and Deligne-Beilinson cohomology, the Deligne-Beilisnon cohomology classes $Eis_\mathcal{D}^{p, q, W}(\phi_f \otimes \phi'_f)$ we are interested in are the image under the Gysin morphism associated to the closed embedding
$$
\begin{CD}
M \times M @>\iota>> S
\end{CD}
$$
of cup-products of Eisenstein classes $Eis_\mathcal{D}^p(\phi_f) \sqcup Eis_\mathcal{D}^q(\phi'_f)$. Hence, to get an explicit description of the classes $Eis_\mathcal{D}^{p, q, W}(\phi_f \otimes \phi'_f)$, we need an explicit description of the cup-product and of the Gysin morphism.

\begin{pro} \label{products} Let $X$ and $Y$ be two objects of $Sch(\mbb{Q})$. Let  $p_X: X \times Y \longrightarrow X$ and $p_Y: X \times Y \longrightarrow Y$ be the canonical projections. Then, via the isomorphism of Prop. \ref{explicitDB}, the external cup-product
$$
\sqcup: H^{m}_\mathcal{D}(X/\mathbb{R}, \mbb{R}(m)) \otimes H^{m'}_\mathcal{D}(X/\mathbb{R}, \mbb{R}(m')) \longrightarrow H^{m+m'}_\mathcal{D}(X/\mathbb{R}, \mbb{R}(m+m'))
$$
is
$$
(\phi, \omega) \sqcup (\phi', \omega')=(p_X^*\phi \wedge p_Y^*( \pi_{m'} \omega')+(-1)^m p_X^*(\pi_m \omega) \wedge p_Y^*\phi', p_X^* \omega \wedge p_Y^* \omega')
$$
for any $m, m'$.
\end{pro}

\begin{proof} The external cup-product is by definition $x \sqcup  y = p_X^*(x) \cup p_Y^*(y)$, where $\cup$ denotes the usual cup-product. Hence, the statement follows from the explicit formulas for the usual cup-product given in \cite{deninger-scholl} 2.5 (see also \cite{esnault-viehweg} 3.10).
\end{proof}

To give an explicit description of the Gysin morphism, we need to introduce currents and Deligne-Beilinson homology. For $X \in Sch(\mbb{Q})$ and any integer $m$, let $\mathcal{T}^\circ(X/\mbb{R}, \mathbb{R}(m))$ be the complex of $\mbb{R}(m)$-valued currents on which the map $F_\infty$, induced by complex conjugation on $X(\mbb{C})$, acts by multiplication by $(-1)^m$. Let us consider as above a smooth projective compactification $j: X \longrightarrow X^*$ and let $i: Y \longrightarrow X^*$ be the complementary reduced closed embedding. Assume $Y$ is a normal crossing divisor. Let $\mathcal{T}_{log}^\circ(X/\mbb{R}, \mbb{C})$ be the complex of currents with logarithmic singularities along $Y$, endowed with the Hodge filtration $(F^r\mathcal{T}_{log}^\circ(X/\mbb{R}, \mbb{C}))_r$. Details on these notions are give in \cite{jannsen} 1.4.

\begin{pro} \label{explicitDBhomology} \cite{kings} Lem. 6.3.9. Let $i$ and $j$ be two integers. The real Deligne-Beilinson homology $H_i^\mathcal{D}(X/\mbb{R}, \mbb{R}(j))$ is the $\mbb{R}$-vector space
$$
H_i^\mathcal{D}(X/\mbb{R}, \mbb{R}(j))=\frac{\{ (S, T) \,|\, dS=\pi_{j-1}T \}}{\{ d(\widetilde{S}, \widetilde{T}) \}}
$$
where
$$
(S, T) \in \mathcal{T}^{-i-1}(X/\mbb{R}, \mathbb{R}(j-1)) \oplus F^j \mathcal{T}_{log}^{-i}(X/\mbb{R}, \mbb{C})
$$
and
$$
(\widetilde{S}, \widetilde{T}) \in \mathcal{T}^{-i-2}(X/\mbb{R}, \mathbb{R}(j-1)) \oplus F^j \mathcal{T}_{log}^{-i-1}(X/\mbb{R}, \mbb{C}).
$$
\end{pro}

As currents are covariant for proper maps, the proposition above gives an explicit description of the Gysin morphism in Deligne-Beilinson cohomology.

\begin{pro}\label{gysin} The following statements hold:
\begin{itemize}
\item \cite{jannsen} Thm. 1.15. Let $X$ be an object of $Sch(\mbb{Q})$ of pure dimension $d_X$. There is a canonical isomorphism between Deligne-Beilinson homology and cohomology
$$
H_i^\mathcal{D}(X/\mbb{R}, \mbb{R}(j)) \simeq H_\mathcal{D}^{2d_X+i}(X/\mbb{R}, \mbb{R}(d_X+j)).
$$
\item \cite{kings} Lem. 6.3.10. Let $Y$ be an object of $Sch(\mbb{Q})$ of pure dimension $d_Y$ and let $i: Y \longrightarrow X$ be a closed embedding of codimension $c=d_X-d_Y$. Then, via the isomorphism above and the explicit description of Deligne-Beilinson homology classes given in Prop. \ref{explicitDBhomology}, the Gysin morphism
$$
i_*: H^m_\mathcal{D}(Y/\mathbb{R}, \mbb{R}(n)) \longrightarrow H^{m+2c}_\mathcal{D}(X/\mathbb{R}, \mbb{R}(n+c))
$$
is induced by the map $(S, T) \longmapsto (i_*S, i_*T)$.
\end{itemize}
\end{pro}

Let us fix an orientation on the complex manifold $X(\mbb{C})$ and let $\phi \in \mathcal{S}^i(X/\mbb{R}, \mbb{R}(j))$. Following \cite{jannsen} $\S 1$, let $T_\phi \in \mathcal{T}^{i-2d_X}(X/\mbb{R}, \mbb{R}(j-d_X))$ denote the current defined by
\begin{equation} \label{diff-form-to-current}
\omega \in \mathcal{S}^i_c(X/\mbb{R}, \mbb{R}(d_X-j)) \longmapsto T_\phi(\omega)=\frac{1}{(2\pi i)^{d_X}}\int_{X(\mbb{C})} \omega \wedge \phi.
\end{equation}

\begin{pro} \label{eisenstein-deligne}  Let
$$
\begin{CD}
Eis^{p, q, W}_\mathcal{D}: \mathcal{B}_p \otimes_\mbb{Q} \mathcal{B}_q @>>> H^{4}_\mathcal{D}(S/\mathbb{R}, W)
\end{CD}
$$
be the composite of the map (\ref{eisM}) and of the regulator
$$
\begin{CD}
H^{4}_\mathcal{M}(S, W) @>>> H^{4}_\mathcal{D}(S/\mathbb{R}, W)
\end{CD}
$$
in Deligne-Beilinson cohomology. For $j=1, 2$ let $p_j: M \times M \lra M$ denote the $j$-th projection. Then
$$
Eis^{p, q, W}_\mathcal{D}(\phi_f \otimes \phi'_f)=(\iota_*T_{P^{p, q}_\mathcal{H}(\phi_f \otimes \phi'_f)}, \iota_*T_{P^{p, q}_B(\phi_f \otimes \phi'_f)})
$$
where $P^{p, q}_\mathcal{H}(\phi_f \otimes \phi'_f)$ is defined by
$$
P^{p, q}_\mathcal{H}(\phi_f \otimes \phi'_f)=p_1^*  Eis^p_\mathcal{H}(\phi_f) \wedge p_2^*(\pi_q Eis^q_B(\phi'_f))+(-1)^p p_1^*(\pi_p Eis^p_B(\phi_f)) \wedge p_2^* Eis^q_\mathcal{H}(\phi'_f)
$$
and $P^{p, q}_B(\phi_f \otimes \phi'_f)$ is defined by
$$
P^{p, q}_B(\phi_f \otimes \phi'_f)=p_1^* Eis_B^p(\phi_f) \wedge p_2^* Eis_B^q(\phi_f).
$$
\end{pro}

\begin{proof}
The statement is a direct consequence of Prop. \ref{explicitEisenstein}, Prop. \ref{products} and Prop. \ref{gysin}.
\end{proof}

\subsection{The use of Poincar\'e duality} \label{section-Poincare} This section explains how the Poincar\'e duality pairing can be used to compute the regulator. This idea is due to Beilinson \cite{beilinson} (see also \cite{kings} 6.1). Let us start with a general result.

\begin{lem}   \label{suite-ex-pour-ext1} Let $E$ be a number field and let $M$ be an object of $\mathrm{MHS}_{\mathbb{R}, E}^+$ (Def. \ref{mhs-coeff}) which is pure weight $w<0$. Let $M_{dR}$ be the sub $E \otimes_\mbb{Q} \mbb{R}$-module of $M_\mbb{C}$ where the de Rham involution acts trivially and let $M^-$ be the submodule of $M$ where $F_\infty$ acts by multiplication by $-1$. Let $M^-(-1)=1/(2\pi i)M^-$. Then, there is an exact sequence of $E \otimes_\mbb{Q} \mbb{R}$-modules
$$
0 \longrightarrow F^0 M_{dR} \longrightarrow M^-(-1) \longrightarrow \mathrm{Ext}^1_{\mathrm{MHS}_\mathbb{R}^+}(\mathbb{R}(0), M) \longrightarrow 0
$$
where the second map is the composite  of the natural inclusions $$F^0 M_{dR} \longrightarrow M_{dR} \longrightarrow M_\mbb{C}$$
and of the projections
$$
M_\mbb{C} \longrightarrow M(-1) \longrightarrow M^-(-1)
$$
defined by $v \longmapsto \frac{1}{2}(v-\overline{v})$ and $v \longmapsto \frac{1}{2}(v-F_\infty(v))$ respectively.
\end{lem}

\begin{proof} Immediate consequence of \cite{nekovar} (2.1) and (2.5).
\end{proof}

\begin{cor} \label{suite_ex_ext1} Let $W$ be an irreducible algebraic representation of $G$ of highest weight $\lambda(k, k', c)$, with $w=3-c<0$. Let $\pi_f$ be the non-archimedean part of an irreducible cuspidal automorphic representation $\pi$ of $G$ whose archimedean component $\pi_\infty$ belongs to the discrete series $L$-packet $P(W)$. Then, there is an exact sequence
$$
0 \longrightarrow F^0 M_{dR}(\pi_f, W)_{\mathbb{R}} \longrightarrow  M_{B}(\pi_f, W)_\mathbb{R}^{-}(-1) \longrightarrow \mathrm{Ext}^1_{\mathrm{MHS}_\mathbb{R}^+}(\mathbb{R}(0),  M_{B}(\pi_f, W)_\mathbb{R}) \longrightarrow 0,
$$
where the second map is as above.
\end{cor}

\begin{proof}For any compact open subgroup $L$ of $G(\mbb{A}_f)$, let $F^0 H^3_{dR, !}(S_L, W)_\mbb{R}$ denote the $\mathbb{R}$-subspace of $F^0 H^3_{B, !}(S_L, W)_{\mbb{C}}$ of vectors which are fixed by the de Rham involution and by $F^0 H^3_{dR, !}(S, W)_\mbb{R}$ the colimit of the $F^0 H^3_{dR, !}(S_L, W)_\mbb{R}$ over all compact open subgroups $L$ of $G(\mbb{A}_f)$. As filtered colimits of vector spaces preserve exact sequences, the lemma above implies that we have a $G(\mathbb{A}_f)$-equivariant exact sequence
$$
0 \longrightarrow F^0 H^3_{dR, !}(S, W)_\mbb{R} \longrightarrow H^3_{B, !}(S, W)_\mbb{R}^-(-1) \longrightarrow \mathrm{Ext}^1_{\mathrm{MHS}_\mathbb{R}^+}(\mathbb{R}(0),  H^3_{B, !}(S, W)_\mbb{R}) \longrightarrow 0
$$
of $\mbb{R}$-vector spaces. As $\pi_f$ is irreducible, applying the functor
$$
X \longmapsto \mathrm{Hom}_{\mbb{Q}[G(\mbb{A}_f)]}(Res_{E(\pi_f)/\mbb{Q}} \pi_f, X)
$$ to the above exact sequence, we still get an exact sequence, which is the one of the statement of the corollary.
\end{proof}

The first term, respectively the second term, of the exact sequence above is obtained applying $\otimes_\mbb{Q} \mbb{R}$ to the finite dimensional $E(\pi_f)$-vector space $F^0 M_{dR}(\pi_f, W)$, respectively $M_{B}(\pi_f, W)^{-}(-1)$ (see Prop. \ref{ratF} and Prop. \ref{hodge-deRham-pif}). Let $F^0 M_{dR}(\pi_f, W)^*$ be the dual of $F^0 M_{dR}(\pi_f, W)$. As a consequence, the one dimensional $E(\pi_f)$-vector space
$$
\mathcal{B}(\pi_f, W)=det_{E(\pi_f)}  F^0 M_{dR}(\pi_f, W)^* \otimes_{E(\pi_f)} det_{E(\pi_f)} M_{B}(\pi_f, W)_\mathbb{R}^{-}(-1)
$$
is an $E(\pi_f)$-structure of the $E(\pi_f) \otimes \mbb{R}$-module $$det_{E(\pi_f) \otimes_\mbb{Q} \mbb{R}} \mathrm{Ext}^1_{\mathrm{MHS}_\mathbb{R}^+}(\mathbb{R}(0),  M_{B}(\pi_f, W)_\mathbb{R}).$$ By definition $\mathcal{B}(\pi_f, W)$ is the Beilinson $E(\pi_f)$-structure, which occurs in one of the two equivalent formulations of Beilinson's conjecture (see \cite{nekovar} 6.1). Because we are dealing with a partial $L$-function, we do not want to use the functional equation in this work. As a consequence, we prefer to work with the Deligne $E(\pi_f)$-structure. 

\begin{defn} \label{deligne-q-structure} Let $\delta(\pi_f, W) \in (E(\pi_f) \otimes_\mbb{Q} \mbb{C})^\times$ be the determinant of the comparison isomorphism $M_B(\pi_f, W)_\mbb{C} \longrightarrow M_{dR}(\pi_f, W)_\mbb{C}$ computed in basis defined over $E(\pi_f)$ on both sides. The Deligne $E(\pi_f)$-structure of $det_{E(\pi_f) \otimes_\mbb{Q} \mbb{R}} \mathrm{Ext}^1_{\mathrm{MHS}_\mathbb{R}^+}(\mathbb{R}(0),  M_{B}(\pi_f, W)_\mathbb{R})$ is
$$
\mathcal{D}(\pi_f, W)=(2 \pi i)^{\dim_{E(\pi_f)} M_B(\pi_f, W)^-} \delta(\pi_f, W)^{-1} \mathcal{B}(\pi_f, W).
$$
\end{defn}

\begin{rems} Of course, this definition does not depend on the choice of the basis.
\end{rems}

The ranks of the $E(\pi_f) \otimes_\mbb{Q} \mbb{R}$-modules of the exact sequence of Cor. \ref{suite_ex_ext1} are related to multiplicities of automorphic representations in the discrete series $L$-packet $P(W)$.

\begin{lem} \label{dimensions} The following equalities hold:
\begin{itemize}
\item we have $rk_{E(\pi_f) \otimes_\mbb{Q} \mbb{R}} M_{B}(\pi_f, W)_\mathbb{R}^{-}(-1)=m(\pi_\infty^H \otimes \pi_f)+m(\pi_\infty^W \otimes \pi_f),$
\item we have $rk_{E(\pi_f) \otimes_\mbb{Q} \mbb{R}} F^0 M_{dR}(\pi_f, W)_{\mathbb{R}}=m(\pi_\infty^H \otimes \pi_f),$
\item we have $rk_{E(\pi_f) \otimes_\mbb{Q} \mbb{R}} \mathrm{Ext}^1_{\mathrm{MHS}_\mathbb{R}^+}(\mathbb{R}(0),  M_{B}(\pi_f, W)_\mathbb{R})=m(\pi_\infty^W \otimes \pi_f)$.
\end{itemize}
\end{lem}

\begin{proof} Thanks to the exact sequence of Cor. \ref{suite_ex_ext1}, it is enough to prove the first two statements. Write $t=\frac{p+q+6-k-k'}{2}$. Then, the first statement easily follows from the Hodge decomposition
$$
M_B(\pi_f, W)_\mbb{C}=M_B^{3-t, -k-k'-t} \oplus M_B^{2-k'-t, 1-k-t} \oplus M_B^{1-k-t, 2-k'-t} \oplus M_B^{-k-k'-t, 3-t},
$$
where $M_B^{3-t, -k-k'-t}, M_B^{2-k'-t, 1-k-t}, M_B^{1-k-t, 2-k'-t}$  and $M_B^{-k-k'-t, 3-t}$ are $E(\pi_f) \otimes \mbb{C}$-modules of respective ranks $m(\pi_\infty \otimes \pi_f), m(\pi_\infty \otimes \pi_f), m(\overline{\pi}_\infty^W \otimes \pi_f)$ and $m(\overline{\pi}_\infty^H \otimes \pi_f)$ (Prop. \ref{dimension} and Prop. \ref{dec-hodge}), and from the fact that $F_\infty$ exchanges $M_B^{r, s}$ and $M_B^{s, r}$. The proof of the second statement follows from the fact that the conditions on $p, q, k$ and $k'$ stated in the introduction imply that $2-k'-t<0$, hence that
$$
F^0 M_{dR}(\pi_f, W)_{\mathbb{C}}=M^{3-t, -k-k'-t}_B.
$$
\end{proof}

\begin{defn} \label{stable} Let $E$ be an irreducible algebraic representation $E$ of $G$. Let $\pi=\pi_\infty \otimes \pi_f$ be an irreducible cuspidal automorphic representation such that $\pi_\infty$ belongs to the discrete series $L$-packet $P(E)$ (Def. \ref{p(e)}). The automorphic representation $\pi$ is stable at infinity if  $m(\pi'_\infty \otimes \pi_f)=1$ for any $\pi'_\infty \in P(E)$.
\end{defn}

\begin{hyp} \label{hypothesis} In the rest of the paper, we assume that the considered cuspidal automorphic representation $\pi$ is stable at infinity.
\end{hyp}

\begin{rems} \label{abondance} It follows from Arthur's classification \cite{arthur} that, for most of $\pi$ whose $\pi_\infty$ is a discrete series, $\pi$ is stable. Specific examples constructed via theta lifts of Hilbert modular forms over real quadratic fields are discussed in  \cite{mokrane-tilouine} 7.3.
\end{rems}

Note that Hyp. \ref{hypothesis} and Lem. \ref{dimensions} imply that $\mathrm{Ext}^1_{\mathrm{MHS}_\mathbb{R}^+}(\mathbb{R}(0),  M_{B}(\pi_f, W)_\mathbb{R})$ is a rank one $E(\pi_f) \otimes_\mbb{Q} \mbb{R}$-module. From now on, we consider integers $p, q, k, k'$  and a coefficient system $W$ of highest weight $\lambda(k, k', p+q+6)$ as in the statement of Thm. \ref{lemma1}. Let $\mathcal{K}(p, q, W)$ denote the sub $\mbb{Q}[G(\mbb{A}_f)]$-module of
$$
\mathrm{Ext}^1_{\mathrm{MHS}_\mathbb{R}^+}(\mathbb{R}(0), H^3_!(S, W)_\mbb{R})= \varinjlim_L \mathrm{Ext}^1_{\mathrm{MHS}_\mathbb{R}^+}(\mathbb{R}(0), H^3_!(S_L, W)_\mbb{R})
$$
generated by the image of $Eis^{p, q, W}_\mathcal{H}$ and let $\mathcal{K}(\pi_f, W)$ be defined as
$$
\mathcal{K}(\pi_f, W)=\mathrm{Hom}_{{\mbb{Q}}[G(\mbb{A}_f)]} \left( Res_{E(\pi_f)/\mbb{Q}}\pi_f, \mathcal{K}(p, q, W) \right).
$$
This is an $E(\pi_f)$-submodule of $\mathrm{Ext}^1_{\mathrm{MHS}_\mathbb{R}^+}(\mathbb{R}(0),  M_{B}(\pi_f, W)_\mathbb{R})$. We have two $E(\pi_f)$-submodules of $\mathrm{Ext}^1_{\mathrm{MHS}_\mathbb{R}^+}(\mathbb{R}(0),  M_{B}(\pi_f, W)_\mathbb{R})$ "geometrically defined". The elementary one, namely $\mathcal{D}(\pi_f, W)$, is defined in terms of the de Rham and Betti cohomology. The sophisticated one, namely $\mathcal{K}(\pi_f, W)$, is defined in terms of the regulator. Of course, whereas $\mathcal{D}(\pi_f, W)$ is non-zero by definition, we do not know at this point whether $\mathcal{K}(\pi_f, W)$ is zero or not. 

\begin{defn} \label{equivalence} If $\mu$ and $\mu'$ are two elements of $E(\pi_f) \otimes_\mbb{Q} \mathbb{C}$, denote $\mu \sim \mu'$ if there exists $\lambda \in E(\pi_f)^\times$ such that $\mu =\lambda \mu'$.
\end{defn}

\begin{lem} \label{facilezzz} Let $v_\mathcal{D}$ be a non-zero vector of $\mathcal{D}(\pi_f, W)$ and $v_\mathcal{K}$ be a vector of $\mathcal{K}(\pi_f, W)$. Let $\widetilde{v_\mathcal{D}} \in M_B(\pi_f, W)_\mbb{R}^-(-1)$, respectively $\widetilde{v_\mathcal{K}}$, be an element mapped to $v_\mathcal{D}$, respectively $v_\mathcal{K}$, by the third map of the exact sequence of Cor. \ref{suite_ex_ext1}. Then
$$
\mathcal{K}(\pi_f, W)=\frac{\psi(\widetilde{v_\mathcal{K}})}{\psi(\widetilde{v_\mathcal{D}})}\mathcal{D}(\pi_f, W)
$$
for any $E(\pi_f) \otimes_\mbb{Q} \mbb{R}$-linear map $\psi: M_B(\pi_f, W)_\mbb{R}^-(-1) \longrightarrow E(\pi_f) \otimes_\mbb{Q} \mbb{C}$ whose kernel contains $ F^0 M_{dR}(\pi_f, W)_{\mathbb{R}}$.
\end{lem}

\begin{proof} Trivial.
\end{proof}

Our goal is now to compute $\psi(\widetilde{v_\mathcal{K}})$ and $\psi(\widetilde{v_\mathcal{D}})$ for a well chosen $\psi$ as above. To this end, let us recall the properties of the Poincar\'e duality pairing for Siegel threefolds. As the representation $W$ has highest weight $\lambda(k, k', p+q+6)$, its contragredient has highest weight $\lambda(k, k', -p-q-6)$. In other words, we have a perfect bilinear pairing $W \otimes W \longrightarrow \mathbb{Q}(p+q+6)$ where $\mbb{Q}(p+q+6)$ denotes the one dimensional $\mathbb{Q}$-vector space on which $G$ acts by the $(p+q+6)$-th power of the multiplier character $\nu$. This pairing induces a $G(\mathbb{A}_f)$-equivariant pairing
$$
\begin{CD}
H^3_{B, !}(S, W) \otimes H^3_{B, !}(S, W(-p-q-3)) @>>> \mbb{Q}(0)
\end{CD}
$$
which becomes perfect after restriction to the vectors which are invariant by a compact open subgroup of $G(\mathbb{A}_f)$, when $\mbb{Q}(0)$ is given the action of $G(\mbb{A}_f)$ by $|\nu|^{-3}$ (see \cite{taylor} p. 295). This induces a morphism of Hodge structures
$$
\begin{CD}
M_B(\pi_f, W) \otimes M_B(\check{\pi}_f|\nu|^{-3}, W(-p-q-3)) @>\langle \,\,,\,\, \rangle_B>> E(\pi_f)(0).
\end{CD}
$$
Recall that according to Prop. \ref{dec-hodge}, we have the Hodge decomposition
$$
M_B(\check{\pi}_f|\nu|^{-3}, W(-p-q-3))_\mbb{C}
$$
$$
=M_B^{3-t', -k-k'-t'} \oplus M_B^{2-k'-t', 1-k-t'} \oplus M_B^{1-k-t', 2-k'-t'} \oplus M_B^{-k-k'-t', 3-t'}
$$
where $t'=-\frac{k+k'+p+q}{2}$.

\begin{lem} \label{dualite} Let $\omega \in M_B(\check{\pi}_f|\nu|^{-3}, W(-p-q-3))_\mbb{C}$.
Let
$$
\begin{CD}
M_B(\pi_f, W)_\mbb{R}^-(-1) @> \langle \omega, \,\, \rangle_B >> E(\pi_f) \otimes \mbb{C}
\end{CD}
$$
be defined as the composition of the natural inclusion $M_B(\pi_f, W)_\mbb{R}^-(-1) \rightarrow M_B(\pi_f, W)_\mbb{C}$ and of the pairing with $\omega$. Assume that $\omega$ belongs to $M_B^{2-k'-t', 1-k-t'} \oplus M_B^{1-k-t', 2-k'-t'}$. Then, with notations of Lem. \ref{facilezzz}, we have
$$
\mathcal{K}(\pi_f, W)=\frac{\langle \omega,\widetilde{v}_\mathcal{K} \rangle_B}{\langle \omega,\widetilde{v}_\mathcal{D} \rangle_B} \mathcal{D}(\pi_f, W).
$$
\end{lem}

\begin{proof} According to Lem. \ref{facilezzz}, it is enough to know that the kernel of $x \longmapsto \langle \omega, x \rangle_B$ contains $F^0 M_{dR}(\pi_f, W)_{\mathbb{R}}$. Write $t=\frac{p+q+6-k-k'}{2}$. We have the Hodge decomposition
$$
M_B(\pi_f, W)_\mbb{C}=M_B^{3-t, -k-k'-t} \oplus M_B^{2-k'-t, 1-k-t} \oplus M_B^{1-k-t, 2-k'-t} \oplus M_B^{-k-k'-t, 3-t}
$$
where $3-t=(k+k'-p-q)/2 \geq 0$ and $2-k'-t=(k-k'-p-q-2)/2<0$ according to the inequalities relating $p, q, k$ and $k'$ stated in the introduction. As a consequence, we have $F^0 M_{dR}(\pi_f, W)_{\mathbb{C}}=\bigoplus_{p \geq 0} M_B^{p, q}=M_B^{3-t, -k-k'-t}$. This implies that the image of the inclusion
$$
\begin{CD}
F^0 M_{dR}(\pi_f, W)_{\mathbb{R}} @>>> M_B(\pi_f, W)_\mbb{R}^-(-1)
\end{CD}
$$
of the exact sequence of Cor. \ref{suite_ex_ext1} is included in
$$
M_B^{3-t, -k-k'-t} \oplus \overline{M_B^{3-t, -k-k'-t}}=M_B^{3-t, -k-k'-t} \oplus M_B^{-k-k'-t, 3-t}.
$$
Hence, the statement follows from the assumption on the Hodge types of $\omega$ and fact that $\langle \,\,, \,\, \rangle_B$ is a morphism of Hodge structures.
\end{proof}

\subsection{From the regulator to a global integral} \label{regulateur-integrale} In this section, we shall explain how to associate a class $\omega$ whose Hodge types are as in Lem. \ref{dualite} to some cusp forms $\Psi$ on $G$. When $\omega$ is associated to $\Psi$, we shall see that the pairing $\langle \omega,\widetilde{v}_\mathcal{K} \rangle_B$ is an integral on $G'(\mbb{Q}) Z'(\mbb{A}) \backslash G'(\mathbb{A})$ of a function of the shape $\Psi E_\mathcal{H}$, where $E_\mathcal{H}$ is an Eisenstein series on $G'$ related to the Deligne-Beilinson realization $Eis_\mathcal{D}^{p, q, W}(\phi_f \otimes \phi'_f)$ of the motivic classes $Eis_\mathcal{M}^{p, q, W}(\phi_f \otimes \phi'_f)$.

\begin{rems} \label{iso} At several places in this article, we consider a $G(\mbb{A}_f)$-module $\mathcal{V}$ and its $\pi_f$-isotypical component $\mathcal{V}(\pi_f)=\mathrm{Hom}_{G(\mbb{A}_f)}(\pi_f, \mathcal{V}).$ Let us choose a vector $x \in \pi_f$, that will be fixed until the end. Then, the linear map
$\mathcal{V}(\pi_f) \lra \mathcal{V}, \phi \longmapsto \phi(x)$ is injective because $\pi_f$ is irreducible. In what follows, we will often regard $\mathcal{V}(\pi_f)$ as a subspace of $\mathcal{V}$, via the choice of $x$. 
\end{rems}

According to Thm. \ref{lemma1}, for any $\phi_f \otimes \phi'_f \in \mathcal{B}_p \otimes_\mbb{Q} \mathcal{B}_q$, we have the extension class $$Eis_\mathcal{H}^{p, q, W}(\phi_f \otimes \phi'_f) \in \mathrm{Ext}^1_{\mathrm{MHS}_\mathbb{R}^+}(\mathbb{R}(0),  H^3_{B, !}(S, W)_\mbb{R})$$ whose image in the Deligne-Beilinson cohomology $H^4_\mathcal{D}(S/\mbb{R}, W)$ coincides with the class of the pair of currents $(\iota_*T_{P^{p, q}_\mathcal{H}(\phi_f \otimes \phi'_f)}, \iota_*T_{P^{p, q}_B(\phi_f \otimes \phi'_f)})$ (Prop. \ref{eisenstein-deligne}). Due to the above remark and Lem. \ref{dualite}, we need to compute a pairing of the shape $\langle \omega,\widetilde{v}_\mathcal{K} \rangle_B$ where $\widetilde{v}_\mathcal{K}$ is a class in $H^3_{B, !}(S, W)_\mbb{R}^-(-1)$ which is mapped to $Eis_\mathcal{H}^{p, q, W}(\phi_f \otimes \phi'_f)$ by the surjection
$$
\begin{CD}
H^3_{B, !}(S, W)_\mbb{R}^-(-1) @>>> \mathrm{Ext}^1_{\mathrm{MHS}_\mathbb{R}^+}(\mathbb{R}(0),  H^3_{B, !}(S, W)_\mbb{R}).
\end{CD}
$$

In the following lemma, whose aim is to describe as explicitly as necessary such a lifting $\widetilde{v}_\mathcal{K}$, we use two standard facts. The first is that Betti cohomology can be computed by closed currents. The second is the so called Liebermann's trick: let $p: A \lra S$ be an abelian scheme and let $W$ be a sheaf on $S$  which is a direct factor of $(R^1 p_* \mbb{Q}(0))^{\otimes n}(m)$, for some integers $m$ and $n$. Then, the cohomology $H^i(S, W)$  is a direct factor of $H^{i+n}(A^n, \mbb{Q}(t))$, for some $t$, where $A^n$ denotes the $n$-th fiber product of $A$ over $S$. Let $r: G \longrightarrow \mrm{GL}(4)$ be the standard representation of $G$. It follows from Weyl's invariants theory that any irreducible algebraic representation $W$ of $G$ of highest weight $\lambda(k, k', c)$ is a direct factor, defined by a certain explicit Schur projector, of the representation $r^{\otimes (k+k')}$ twisted by a power of $\nu$ (see \cite{fulton-harris} $\S 17.3$). As the variation of Hodge structure associated to $r$ is $R^1p_*\mbb{Q}(1)$, the variation of Hodge structure associated to such a $W$ is a direct factor of a Tate twist of $(R^1 \pi_* \mbb{Q}(0))^{\otimes (k+k')}$ where $p: A \lra S$ is the universal abelian surface over the Siegel threefold. It is not necessary for us to be more precise. However, the reader might consult \cite{ancona} for much more precise and general statements and their proofs.

\begin{lem} \label{relevement} Let $\phi_f \otimes \phi'_f \in \mathcal{B}_p \otimes_\mbb{Q} \mathcal{B}_q$. Let $A$ be the universal abelian surface over $S$, of infinite level, and let $A^{k+k'}$ be the $(k+k')$-th fold fiber product over $S$. Let $\widetilde{A}^{k+k'}$ be a smooth projective toroidal compactification of $A^{k+k'}$ such that the complement $\widetilde{A}^{k+k'}-A^{k+k'}$ is a normal crossing divisor. Then, there exists a closed current $\rho \in \mathcal{T}(\widetilde{A}^{k+k'}/ \mbb{R}, \mbb{R}(\circ))$, for some integer $\circ$, such that:
\begin{itemize}
\item the restriction of the cohomology class $[\rho]$ of $\rho$ to $A^{k+k'}$ belongs to $H^3_{B, !}(S, W)_\mbb{R}^-(-1)$,
\item the class $[\rho]$ is mapped to the extension class $Eis_\mathcal{H}^{p, q, W}(\phi_f \otimes \phi'_f)$ by the third map
$$
\begin{CD}
H^3_{B, !}(S, W)_\mbb{R}^-(-1) \lra \mathrm{Ext}^1_{\mathrm{MHS}_\mathbb{R}^+}(\mathbb{R}(0),  H^3_{B, !}(S, W)_\mbb{R})
\end{CD}
$$
of the exact sequence of the proof of Cor. \ref{suite_ex_ext1},
\item the pairs of currents $\left(\iota_*T_{P^{p, q}_\mathcal{H}(\phi_f \otimes \phi'_f)}, \iota_*T_{P^{p, q}_B(\phi_f \otimes \phi'_f)} \right)$ and $\left( \rho, 0 \right)$ represent the same cohomology class in $H^{4}_\mathcal{D}(S/\mathbb{R}, W)$ (see Prop. \ref{eisenstein-deligne}).
\end{itemize}
\end{lem}

\begin{proof} The natural map $H^*_c(A^{k+k'}) \lra H^*(A^{k+k'})$ factors through $H^*(\widetilde{A}^{k+k'}) \lra H^*(A^{k+k'})$. Hence, for the first and the second statement, we can take $\rho$ to be any closed current on $\widetilde{A}^{k+k'}$ representing a lifting  of $Eis^{p,q, W}_\mathcal{H}(\phi_f \otimes \phi'_f)$ by the map
$$
\begin{CD}
H^3_{B, !}(S, W)_\mbb{R}(-1)^+ @>>> \mathrm{Ext}^1_{\mathrm{MHS}_\mathbb{R}^+}(\mathbb{R}(0),  H^3_{B, !}(S, W)_\mbb{R}).
\end{CD}
$$
Note that a current on $\widetilde{A}^{k+k'}$ can be regarded as a current on ${A}^{k+k'}$, so that the third statement is meaningful. As $\widetilde{A}^{k+k'}$ is smooth and projective, it follows from \cite{jannsen} 4.5.1 that there exists a natural map $H^{*}_B(\widetilde{A}^{k+k'}, \mbb{R}(\circ))^- \lra H^{*+1}_\mathcal{D}(\widetilde{A}^{k+k'}/\mbb{R}, \mbb{R}(\circ+1))$. This map is induced by the map sending a closed current $\tau$ to the pair $(\tau, 0)$ via the description of Deligne-Beilinson homology classes by pairs of currents (Prop. \ref{explicitDBhomology}) and the isomorphism of the Deligne-Beilinson cohomology and homology (Prop. \ref{gysin}). Composing with the restriction map $$H_\mathcal{D}(\widetilde{A}^{k+k'}/\mbb{R}, \mbb{R}(\circ +1)) \lra H_\mathcal{D}(A^{k+k'}/\mbb{R}, \mbb{R}(\circ +1)),$$ we obtain the upper map of a commutative diagram
$$
\begin{CD}
H^{k+k'+3}_B(\widetilde{A}^{k+k'}, \mbb{R}(\circ))^- @>>> H^{k+k'+4}_\mathcal{D}(A^{k+k'}/\mbb{R}, \mbb{R}(\circ+1))\\
@VVV                                                                        @VVV\\
H^3_{B, !}(S, W)(-1)                    @>>>  H^4_\mathcal{D}(S/\mbb{R}, W)
\end{CD}
$$
where the vertical maps are induced by the Schur projectors defining $W$. Hence, the third statement follows from the second.
\end{proof}

To compute the pairing $\langle \omega, [\rho] \rangle_B$, we shall use the notion of rapidly decreasing and slowly increasing differential form and the fact that the Poincar\'e duality pairing can be represented by a pairing between rapidly decreasing and slowly increasing differential forms. For definitions, we refer the reader to \cite{harris0} 1.3 (see also \cite{borel0} 3.2) and for the statement about the Poincar\'e duality pairing, we refer the reader to \cite{harris0} Prop. 1.4.4. (c).\\

The differential forms $Eis^p_B(\phi_f)$ and $Eis^q_B(\phi'_f)$ are slowly increasing as explained in \cite{kings} (6.3.1) and \cite{kings} p. 120. Moreover, according to Prop. \ref{eisenstein-deligne}, we have
$$
P^{p, q}_\mathcal{H}(\phi_f \otimes \phi'_f)=p_1^*  Eis^p_\mathcal{H}(\phi_f) \wedge p_2^*(\pi_q Eis^q_B(\phi'_f))+(-1)^p p_1^*(\pi_p Eis^p_B(\phi_f)) \wedge p_2^* Eis^q_\mathcal{H}(\phi'_f).
$$
Hence $P^{p, q}_\mathcal{H}(\phi_f \otimes \phi'_f)$ is slowly increasing. On the other hand, according to Prop. \ref{comparaisons}, the cohomology class $\omega$ can be computed by cuspidal cohomology. So, let $\Omega$ be a cuspidal differential form representing $\omega$. As $\Omega$ is rapidly decreasing by definition, the differential form $\iota^*\Omega \wedge P^{p, q}_\mathcal{H}(\phi_f \otimes \phi'_f)$, where $\iota^*\Omega$ denotes the restriction of $\Omega$ to $M \times M$, is rapidly decreasing. This differential form has values in the vector space underlying the algebraic representation
$$
\iota^*W(-p-q-3)_\mbb{C} \otimes (\mrm{Sym}^p V_{2, \mbb{C}}^\vee \boxtimes \mrm{Sym}^q V_{2, \mbb{C}}^\vee)
$$
of $G'$. Recall that the irreducible representation $\mrm{Sym}^p V_{2} \boxtimes \mrm{Sym}^q V_{2}$ occurs in the isotypical decomposition of $\iota^*W(-3)$  by our choice of $p, q, k$ and $k'$. Hence, we have the $G'$-equivariant pairing
$$
\begin{CD}
\iota^*W(-p-q-3)_\mbb{C} \otimes (\mrm{Sym}^p V_{2, \mbb{C}}^\vee \boxtimes \mrm{Sym}^q V_{2, \mbb{C}}^\vee) @>\langle \,\,,\,\,\rangle >> \mathbb{C}(-p-q)
\end{CD}
$$
defined as the composite of the natural projection $$\iota^*W(-p-q-3) \longrightarrow (\mrm{Sym}^p V_{2, \mbb{C}}^\vee \boxtimes \mrm{Sym}^q V_{2, \mbb{C}}^\vee)$$ and of the natural pairing
\begin{equation} \label{accouplement}
\begin{CD}
(\mrm{Sym}^p V_{2, \mbb{C}}^\vee \boxtimes \mrm{Sym}^q V_{2, \mbb{C}}^\vee) \otimes (\mrm{Sym}^p V_{2, \mbb{C}}^\vee \boxtimes \mrm{Sym}^q V_{2, \mbb{C}}^\vee) @>>> \mathbb{C}(-p-q),
\end{CD}
\end{equation}
which is given by
$$
\langle a_r^p \boxtimes a_s^q, a_{r'}^p \boxtimes a_{s'}^q \rangle= \left\{
   \begin{array}{ll}
       0 & \mbox{if } r+r' \neq p  \mbox{ or } s+s' \neq q \\
      (-1)^{r+s}(2i)^{-p-q} \binom{p}{r}\binom{q}{s}  & \mbox{else.}
   \end{array}
\right.
$$

Let $A_{G'}=\mbb{R}_+^\times$ be the identity component of the center of $G'(\mbb{R})$ and let $\mathfrak{g}'$, respectively $\mathfrak{k}'$, be the Lie algebra of $G'(\mathbb{R})$, respectively of its subgroup $A_{G'}(\mathrm{U}(1) \times_{\mathbb{R}^\times} \mrm{U}(1))$. Then, with notations of the beginning of section \ref{sousectionexplicit}, we have
$$
\mathfrak{g}'_\mathbb{C}/ \mathfrak{k}'_\mathbb{C}=(\mathfrak{p'}^+ \oplus \mathfrak{p'}^-) \oplus (\mathfrak{p'}^+ \oplus \mathfrak{p'}^-).
$$
Let $\textbf{1}$ be the generator of the highest exterior power $\bigwedge^4 \mathfrak{g}'_\mathbb{C}/ \mathfrak{k}'_\mathbb{C}$ defined by
$$
\textbf{1}=(v^+, 0) \wedge (0, v^+) \wedge (v^-, 0) \wedge (0, v^-)
$$
where $v^\pm=\frac{1}{2}\begin{pmatrix}
1 & \pm i\\
\pm i & -1\\
\end{pmatrix} \in \mathfrak{p'}^\pm$. Evaluating the differential form $\iota^*\Omega \wedge P^{p, q}_\mathcal{H}(\phi_f \otimes \phi'_f)$ at the tangent vector $\textbf{1}$, we get the rapidly decreasing  vector valued function
$$
(\iota^*\Omega \wedge P^{p, q}_\mathcal{H}(\phi_f \otimes \phi'_f))(\textbf{1})  \in \iota^*W(-p-q-3)_\mbb{C} \otimes (\mrm{Sym}^p V_{2, \mbb{C}}^\vee \boxtimes \mrm{Sym}^q V_{2, \mbb{C}}^\vee) \otimes \mathcal{C}^\infty_{rd}(G'(\mathbb{Q})\backslash G'(\mathbb{A})).
$$
Here $\mathcal{C}^\infty_{rd}(G'(\mathbb{Q})\backslash G'(\mathbb{A}))$ denotes the space of rapidly decreasing functions defined in \cite{harris0} p. 48. Composing with the pairing defined above, finally, we get the rapidly decreasing function
$$
\langle (\iota^*\Omega \wedge P^{p, q}_\mathcal{H}(\phi_f \otimes \phi'_f))(\textbf{1}) \rangle \in \mathcal{C}^\infty_{rd}(G'(\mathbb{Q})\backslash G'(\mathbb{A})).
$$

\begin{pro} \label{crucial} Let $\omega$ be a vector of  $M_{B}(\check{\pi}_f |\nu|^{-3}, W(-p-q-3))_\mbb{C}$ satisfying the condition on Hodge types of Lem. \ref{dualite}. Let $\Omega$ be a cuspidal differential form representing $\omega$, let $\rho$ be a closed current as in the statement of Lem. \ref{relevement} and let $[\rho]$ be its cohomology class. Let $dg$ be the measure on $G'(\mbb{Q})Z'(\mbb{A})\backslash G'(\mathbb{A})$ associated to $\textbf{1}$ (see section 2.4). Then,
$$
\langle \omega, [\rho] \rangle_B=\frac{1}{(2 \pi i)^{p+q+2}}\int_{G'(\mbb{Q})Z'(\mbb{A})\backslash G'(\mathbb{A})} \langle (\iota^*\Omega \wedge P^{p, q}_\mathcal{H}(\phi_f \otimes \phi'_f))(\textbf{1})\rangle(g) |det g|^3 dg.
$$
\end{pro}

\begin{proof}  There exists a rapidly decreasing differential form $\eta$ such that $\Omega'=\Omega-d\eta$ is compactly supported (\cite{borel0} Cor. 5.5). We claim that the pairing $\langle \iota^*d\eta, P^{p, q}_\mathcal{H}(\phi_f \otimes \phi'_f) \rangle$ is zero. To prove this, recall that according to the third statement of Lem. \ref{relevement}, the classes of the pairs of currents $\left( \iota_*T_{P^{p, q}_\mathcal{H}(\phi_f \otimes \phi'_f)}, \iota_*T_{P^{p, q}_B(\phi_f \otimes \phi'_f)} \right)$ and $(\rho, 0)$ coincide in $H^4_\mathcal{D}(S/\mbb{R}, W)$. According to Prop. \ref{explicitDBhomology}, this implies the existence of a current  $\rho'$ on the open part $A^{k+k'}$ such that $$\rho = \iota_*T_{P^{p, q}_\mathcal{H}(\phi_f \otimes \phi'_f)}+d\rho'.$$
As  $d\rho=0$, for any compactly supported differential form $\eta_c$ of suitable degree on $S$, we have
$$
\iota_*T_{P^{p, q}_\mathcal{H}(\phi_f \otimes \phi'_f)}(d\eta_c)=(\rho-d\rho')(d\eta_c)=0.
$$
The differential form $\iota^*d\eta \wedge P^{p, q}_\mathcal{H}(\phi_f \otimes \phi'_f)$ on $M \times M$ is rapidly decreasing, hence it extends to a differential form on a smooth compactification $\widetilde{M \times M}$, which is zero on the boundary. As a consequence, we can see the pairing $\langle \iota^*d\eta, P^{p, q}_\mathcal{H}(\phi_f \otimes \phi'_f) \rangle$ as an integral
$$
\int_{\widetilde{M \times M}} {\iota^*d\eta \wedge P^{p, q}_\mathcal{H}(\phi_f \otimes \phi'_f)}
$$
over the compactification. To prove that this integral is zero, approximate $\eta$ by a form $\eta_c$ whose support is compact in $S$. We have
$$
\int_{\widetilde{M \times M}} {\iota^*d\eta \wedge P^{p, q}_\mathcal{H}(\phi_f \otimes \phi'_f)}
$$
$$=\int_{\widetilde{M \times M}} {\iota^*d\eta_c \wedge P^{p, q}_\mathcal{H}(\phi_f \otimes \phi'_f)}+\int_{\widetilde{M \times M}} {\iota^*d(\eta-\eta_c) \wedge P^{p, q}_\mathcal{H}(\phi_f \otimes \phi'_f)}.
$$
On the one hand, the first integral coincides with $\iota_*T_{P^{p, q}_\mathcal{H}(\phi_f \otimes \phi'_f)}(d\eta_c)$, hence is equal to zero. On the other hand, as the manifold $\widetilde{M \times M}$ is compact, the second integral can be made arbitrarily small if the support of $\eta_c$ is sufficiently close to the support of $\eta$. Hence,
$$
\langle \iota^*d\eta, P^{p, q}_\mathcal{H}(\phi_f \otimes \phi'_f) \rangle=0.
$$
As a consequence,
\begin{eqnarray*}
\langle \omega, [\rho] \rangle_B &=& \rho(\Omega')\\
&=& \iota_*T_{P^{p, q}_\mathcal{H}(\phi_f \otimes \phi'_f)}(\Omega')+ d\rho'(\Omega')\\
&=& \iota_*T_{P^{p, q}_\mathcal{H}(\phi_f \otimes \phi'_f)}(\Omega')\\
&=& \langle \iota^*\Omega', P^{p, q}_\mathcal{H}(\phi_f \otimes \phi'_f) \rangle\\
&=& \langle \iota^*\Omega, P^{p, q}_\mathcal{H}(\phi_f \otimes \phi'_f) \rangle.
\end{eqnarray*}
The second equality follows from the definition of $\rho'$, the third from the fact that $\Omega'$ is closed, the fourth from the definition of the current associated to a differential form (\ref{diff-form-to-current}), and the last from the vanishing of $\langle \iota^*d\eta, P^{p, q}_\mathcal{H}(\phi_f \otimes \phi'_f) \rangle$. The fact that $\langle \iota^*\Omega, P^{p, q}_\mathcal{H}(\phi_f \otimes \phi'_f) \rangle$ is computed by the integral given in the statement of the proposition follows from \cite{harris0} Prop. 1.4.4. (c). This completes the proof.
\end{proof}

Our goal is now to  give an explicit formula for the integrand in the above integral. The first step is to explain precisely how to associate differential forms on $S$ to cusp forms on $G$. Once we have the results of section \ref{Hodge-deRham} at our disposal, this association relies on rather elementary representation theoretic considerations.\\

Let $T'$ be the maximal compact subtorus of $\mathrm{Sp}(4, \mathbb{R})$ defined by
$$
T'= \left\{ \begin{pmatrix}
x &  & y & \\
 & x' &  & y'\\
-y &  & x & \\
 & -y' &  & x' \\
\end{pmatrix} \,\,|\,\, x^2+y^2=x'^2+y'^2=1 \right\}.
$$
The Lie algebra of $T'$ is the compact Cartan subalgebra of $\mathfrak{sp}_4$ that we denoted by $\mathfrak{h}$ in section \ref{discrete_series_classification}. Let $A_G=\mbb{R}^\times_+$ be the identity component of the center of $G(\mbb{R})$. For integers $n, n', c$ such that $n+n' \equiv c \pmod 2$, let $\lambda'(n, n', c): A_G T' \longrightarrow \mathbb{C}^\times$ denote the character defined by
$$
\begin{pmatrix}
x &  & y & \\
 & x' &  & y'\\
-y &  & x & \\
 & -y' &  & x' \\
\end{pmatrix} \longmapsto (x+iy)^n (x'+iy')^{n'}(x^2+y^2)^\frac{c-n-n'}{2},
$$
and by $\lambda'(n, n')$ the restriction of $\lambda'(n, n', c)$ to $T'$. Note that the simple root $e_1-e_2$, respectively $2e_2$, defined in section \ref{discrete_series_classification}, coincides with the differential at the identity matrix of the restriction to $T'$ of the character $\lambda'(1, -1, 0)$, respectively $\lambda(0, 2, 0)$.

\begin{lem} \label{lambda-lambda'} Let $E$ be an irreducible algebraic representation of $G$ in a finite dimensional $\mbb{Q}$-vector space and let $w \in E$ be a vector of weight $\lambda(u, u', c)$. Let
\begin{equation*} \label{J}
J= \frac{1}{\sqrt{2}} \begin{pmatrix}
1 &  & i & \\
 & 1 &  & i\\
i &  & 1 & \\
 & i &  & 1 \\
\end{pmatrix} \in \mathrm{Sp}(4, \mathbb{C}).
\end{equation*}
Let $v=Jw \in E_\mbb{C}$ and let $\overline{v}=Nv$, where $N \in G(\mbb{R})$ is defined in Rem. \ref{conj-ds}. Then, for the action of the torus $A_G T' \subset G(\mbb{R})$, the vector $v$, respectively $\overline{v}$, has weight $\lambda'(u, u', c)$, respectively $\lambda'(-u', -u, c)$.
\end{lem}

\begin{proof}
The statement on the weight of $v$ is a straightforward consequence of the fact that for any $x, x', y, y' \in \mathbb{R}$, we have
\begin{eqnarray*}
J^{-1}\begin{pmatrix}
x &  & y & \\
 & x' &  & y'\\
-y &  & x & \\
 & -y' &  & x' \\
\end{pmatrix} J &=& \overline{J} \begin{pmatrix}
x &  & y & \\
 & x' &  & y'\\
-y &  & x & \\
 & -y' &  & x' \\
\end{pmatrix} J\\
&=& \begin{pmatrix}
x+iy &  &  & \\
 & x'+iy' &  & \\
 &  & x-iy & \\
 &  &  & x'-iy' \\
\end{pmatrix}.
\end{eqnarray*}
Hence, the statement on the weight of $\overline{v}$ follows from the identities
\begin{eqnarray*}
\begin{pmatrix}
x &  & y & \\
 & x' &  & y'\\
-y &  & x & \\
 & -y' &  & x' \\
\end{pmatrix} \overline{v} &=& N N^{-1} \begin{pmatrix}
x &  & y & \\
 & x' &  & y'\\
-y &  & x & \\
 & -y' &  & x' \\
\end{pmatrix} N v\\
&=& N \begin{pmatrix}
x' &  & -y' & \\
 & x &  & -y\\
y' &  & x' & \\
 & y &  & x \\
\end{pmatrix} v\\
&=& \lambda(-u', -u, c) \left( \begin{pmatrix}
x &  & y & \\
 & x' &  & y'\\
-y &  & x & \\
 & -y' &  & x' \\
\end{pmatrix} \right) \overline{v}.
\end{eqnarray*}
\end{proof}

\begin{lem} \label{def.omegaw}
Let $X_{(1, -1)} \in \mathfrak{k}_\mathbb{C}$ be defined by
$$
X_{(1, -1)}=d \kappa \left(  \begin{pmatrix}
 & 1 \\
 &  \\
\end{pmatrix}\right)= \begin{pmatrix}
 & 1 &  & -i \\
-1 & & -i & \\
 & i &  & 1\\
i &  & -1 & \\
\end{pmatrix}.
$$
Let $\pi_\infty^W$ and $\overline{\pi}_\infty^W \in P(W(-p-q-3))$ be as in Lem. \ref{dpp}. Let $\Psi_\infty \in \pi_\infty^W$ be a lowest weight vector of the minimal $K$-type $\tau_{(k+3,-k'-1)}$ of $\pi_\infty^W$ and let $\overline{\Psi}_\infty \in \overline{\pi}_\infty^W$ be the vector associated to $\Psi_\infty$ as in Rem. \ref{conj-ds}. Let $w \in W(-p-q-3)$ be a vector of weight $\lambda(-k, k', -p-q)$ and let $v, \overline{v} \in W(-p-q-3)_\mbb{C}$ be the vectors associated to $w$ as in Lem \ref{lambda-lambda'}. Let
\begin{eqnarray*}
X_{(2, 0)} \wedge X_{(1, 1)} \otimes X_{(0, -2)} &\in& \bigwedge^2 \mathfrak{p}^+ \otimes \mathfrak{p}^-,\\
X_{(0, -2)} \wedge X_{(-1, -1)} \otimes X_{(2, 0)} &\in& \bigwedge^2 \mathfrak{p}^- \otimes_\mbb{C} \mathfrak{p}^+
\end{eqnarray*}
where the $X_{(u, u')}$ are the root vectors defined in section \ref{discrete_series_classification}. Then:
\begin{itemize}
\item there exists a unique non-zero map
$$
\Omega(\Psi_\infty) \in \mathrm{Hom}_{K_G}\left( \bigwedge^2 \mathfrak{p}^+ \otimes_\mbb{C} \mathfrak{p}^-,  W(-p-q-3) \otimes_\mbb{C} \pi_\infty^W \right)
$$
such that
$$
\Omega(\Psi_\infty)(X_{(2, 0)} \wedge X_{(1, 1)} \otimes X_{(0, -2)})=\sum_{i = 0}^{k+k'-1} (-1)^i  X_{(1, -1)}^i v \otimes {X}^{k+k'+4-i}_{(1, -1)} \Psi_\infty,
$$
\item for the action of $N$ on the de Rham complex defined in Prop. \ref{milne-shih}, we have
$$
N\Omega(\Psi_\infty) \in \mrm{Hom}_{K_G}\left( \bigwedge^2 \mathfrak{p}^- \otimes_\mbb{C} \mathfrak{p}^+,  W(-p-q-3) \otimes_\mbb{C} \overline{\pi}_\infty^W \right)
$$
and
$$
(N\Omega(\Psi_\infty))(X_{(0, -2)} \wedge X_{(-1, -1)} \otimes X_{(2, 0)})= \sum_{i = 0}^{k+k'-1} (-1)^i  X_{(1, -1)}^i \overline{v} \otimes {X}^{k+k'+4-i}_{(1, -1)} \overline{\Psi}_\infty.
$$
\end{itemize}
\end{lem}

\begin{proof}
Note that $X_{(1, -1)}$ is a root vector corresponding to the positive compact root and recall that we have the isotypical decomposition
$$
\bigwedge^2 \mathfrak{p}^+ \otimes_\mbb{C} \mathfrak{p}^-= \tau_{(3, -1)} \oplus \tau_{(2, 0)} \oplus \tau_{(1, 1)}.
$$
Hence, the vector  $X_{(2, 0)} \wedge X_{(1, 1)} \otimes X_{(0, -2)}$ is a highest weight vector. As a consequence, to define the restriction of a $K_G$-equivariant map $\Omega: \bigwedge^2 \mathfrak{p}^+ \otimes_\mbb{C} \mathfrak{p}^- \lra W(-p-q-3) \otimes_\mbb{C} \pi_\infty^W$ to $\tau_{(3, -1)}$ amounts to give the image of $X_{(2, 0)} \wedge X_{(1, 1)} \otimes X_{(0, -2)}$, under the condition that $\Omega(X_{(2, 0)} \wedge X_{(1, 1)} \otimes X_{(0, -2)})$ must have the same highest weight as $X_{(2, 0)} \wedge X_{(1, 1)} \otimes X_{(0, -2)}$. For every $0 \leq i \leq k+k'+4$, the vector $X_{(1, -1)}^{k+k'+4-i} \Psi_\infty$ has weight $\lambda'(k+3, -k'-1, c)-i\lambda'(1, -1, 0)$ where $c=p+q+6$. Similarly, the vector $X_{(1, -1)}^i v$ has weight $\lambda'(-k ,k', -c)+i\lambda'(1, -1, 0)$. As a consequence, the vector $\sum_{i = 0}^{k+k'-1} (-1)^i  X_{(1, -1)}^i v \otimes X_{(1, -1)}^{k+k'+4-i} \Psi_\infty$ has weight $\lambda'(3, -1, 0)$, which is the weight of $X_{(2, 0)} \wedge X_{(1, 1)} \otimes X_{(0, -2)}$. Furthermore
\begin{eqnarray*}
X_{(1, -1)} \sum_{i = 0}^{k+k'-1} (-1)^i X_{(1, -1)}^{i} v \otimes  X_{(1, -1)}^{k+k'+4-i} \psi_\infty &=& \sum_{i = 0}^{k+k'-1} (-1)^i  X_{(1, -1)}^{i+1} v \otimes X_{(1, -1)}^{k+k'+4-i} \Psi_\infty \\
&& + \sum_{i = 0}^{k+k'-1} (-1)^i  X_{(1, -1)}^{i} v \otimes X_{(1, -1)}^{k+k'+4-i+1} \Psi_\infty\\
&=& \sum_{i = 0}^{k+k'-2} (-1)^i  X_{(1, -1)}^{i+1} v \otimes X_{(1, -1)}^{k+k'+4-i} \Psi_\infty \\
&& + \sum_{i = 1}^{k+k'-1} (-1)^i  X_{(1, -1)}^{i} v \otimes X_{(1, -1)}^{k+k'+4-i+1} \Psi_\infty\\
&=& 0.
\end{eqnarray*}
In other terms, $\sum_{i=0}^{k+k'-1} (-1)^i X_{(1, -1)}^i v \otimes  X_{(1, -1)}^{k+k'+4-i} \Psi_\infty$ is a highest weight vector of the $K_G$-module $W(-p-q-3) \otimes_\mathbb{C} \pi_\infty^W.$ As a consequence, there exists a unique non-zero $K_G$-equivariant map $\tau_{(3, -1)} \longrightarrow W(-p-q-3) \otimes_\mathbb{C} \pi_\infty^W$ sending $X_{(2, 0)} \wedge X_{(1, 1)} \otimes X_{(0, -2)}$ to $\sum_{i=0}^{k+k'-1} (-1)^i X_{(1, -1)}^i v \otimes  X_{(1, -1)}^{k+k'+4-i} \Psi_\infty.$ Thanks to Prop. \ref{dimension}, this proves the first statement. Let us prove the second: by definiton of the action of $N$ on the de Rham complex (Prop. \ref{milne-shih}), we have
\begin{eqnarray*}
(N\Omega(\Psi_\infty))(X_{(0, -2)} \wedge X_{(-1, -1)} \otimes X_{(2, 0)})=N(\Omega(\Psi_\infty)(\mrm{Ad}_{N^{-1}}(X_{(0, -2)} \wedge X_{(-1, -1)} \otimes X_{(2, 0)})).
\end{eqnarray*}
An easy computation shows that
$$
\mrm{Ad}_{N^{-1}}(X_{(0, -2)} \wedge X_{(-1, -1)} \otimes X_{(2, 0)})=X_{(2, 0)} \wedge X_{(1, 1)} \otimes X_{(0, -2)}.
$$
Hence,
\begin{eqnarray*}
(N\Omega(\Psi_\infty))(X_{(0, -2)} \wedge X_{(-1, -1)} \otimes X_{(2, 0)}) &=& N(\Omega(\Psi_\infty)(X_{(2, 0)} \wedge X_{(1, 1)} \otimes X_{(0, -2)}))\\
&=& N\left( \sum_{i = 0}^{k+k'-1} (-1)^i  X_{(1, -1)}^i v \otimes {X}^{k+k'+4-i}_{(1, -1)}, \Psi_\infty \right)\\
&=& \sum_{i = 0}^{k+k'-1} (-1)^i  X_{(1, -1)}^i \overline{v} \otimes X_{(1, -1)}^{k+k'+4-i}, \overline{\Psi}_\infty.
\end{eqnarray*}
The last identity follows from the fact that, by definition of a $(\mathfrak{g}, K)$-module $\mathcal{V}$, we have
$$
kXv=\mrm{Ad}_k(X) kv
$$
for any $k \in K, X \in \mathfrak{g}$ and $v \in \mathcal{V}$ and from the equality $\mrm{Ad}_N(X_{(1, -1)})=X_{(1, -1)}$.
\end{proof}

\begin{lem} \label{valeurs_Omega} Let $\phi_f \otimes \phi'_f \in \mathcal{B}_p \otimes_\mbb{Q} \mathcal{B}_q$. Let $\Psi=\Psi_\infty \otimes \Psi_f$ be a factorizable cusp form on $G$. Assume that $\Psi_\infty$ satisfies the conditions of Lem. \ref{def.omegaw}. Let $A_{k, k', i, j}, B_{k, k', i}$ and $C_{k, k', i}$ denote the integers
\begin{eqnarray*}
A_{k, k', i, j} &=& \frac{(k+k'+4-i)!}{(k+k'+4-(i+j))!}\frac{(i+j)!}{(i-j)!}\frac{(k+k'-i+j)!}{(k+k'-i)!},\\
B_{k, k', i} &=& (i+1)(k+k'+4-i),\\
C_{k, k', i} &=&  i(k+k'-i+1).
\end{eqnarray*}
Then,
$$
(\iota^*\left( \Omega(\Psi_\infty) \otimes \Psi_f \right) \wedge P^{p, q}_\mathcal{H}(\phi_f \otimes \phi'_f))(\textbf{1})=
$$
$$
\frac{3}{160} \frac{(2 \pi i)^{p+q+2}}{(p+1)} \sum_{i = 0}^{k+k'-1} \sum_{j=0}^3  (-1)^i \binom{3}{j}A_{k, k', i, j} {X}^{k+k'+4-(i+j)}_{(1, -1)} \Psi \otimes X_{(1, -1)}^{i-j} v$$
$$
\otimes  \sum_{r=0}^p \sum_{(\gamma, \gamma') \in B'(\mathbb{Q})\backslash G'(\mathbb{Q})} a_r^{(p)} \otimes a_0^{(q)} \otimes \gamma^*(\phi^{p}_{2r-p} \otimes \phi_f)  \gamma'^*(\phi^q_{-q-2} \otimes \phi'_f)
$$
$$
-\frac{(-1)^p(2 \pi i)^{p+q+2}}{8(q+1)}  \sum_{i = 0}^{k+k'-1}  (-1)^i B_{k, k', i} X_{(1, -1)}^{k+k'+4-(i+1)} \Psi \otimes X_{(1, -1)}^i v
$$
$$
\otimes   \sum_{s=0}^q \sum_{(\gamma, \gamma') \in B'(\mathbb{Q})\backslash G'(\mathbb{Q})} a_0^{(p)} \otimes a_s^{(q)} \otimes\gamma^*(\phi^p_{-p-2} \otimes \phi_f) \gamma'^*(\phi^{q}_{2s-q} \otimes \phi_f')
$$
$$
-\frac{(-1)^p(2 \pi i)^{p+q+2}}{8(q+1)}  \sum_{i = 1}^{k+k'-1} (-1)^i C_{k, k', i} X_{(1, -1)}^{k+k'+4-i} \Psi \otimes X_{(1, -1)}^{i-1} v
$$
$$
\otimes    \sum_{s=0}^q \sum_{(\gamma, \gamma') \in B'(\mathbb{Q})\backslash G'(\mathbb{Q})} a_0^{(p)} \otimes a_s^{(q)} \otimes\gamma^*(\phi^p_{-p-2} \otimes \phi_f) \gamma'^*(\phi^{q}_{2s-q} \otimes \phi_f'),
$$
and
$$
(\iota^* F_\infty \left( \Omega(\Psi_\infty) \otimes \Psi_f \right) \wedge P^{p, q}_\mathcal{H}(\phi_f \otimes \phi'_f))(\textbf{1})=
$$
$$
\frac{3}{160} \frac{(-1)^p (2 \pi i)^{p+q+2}}{(q+1)} \sum_{i = 0}^{k+k'-1} \sum_{j=0}^3  (-1)^i \binom{3}{j}A_{k, k', i, j} {X}^{k+k'+4-(i+j)}_{(1, -1)} \overline{\Psi} \otimes X_{(1, -1)}^{i-j} \overline{v}$$
$$
\otimes  \sum_{s=0}^q \sum_{(\gamma, \gamma') \in B'(\mathbb{Q})\backslash G'(\mathbb{Q})} a_p^{(p)} \otimes a_s^{(q)} \otimes \gamma^*(\phi^{p}_{p+2} \otimes \phi_f)  \gamma'^*(\phi^q_{2s-q} \otimes \phi'_f)
$$
$$
-\frac{(2 \pi i)^{p+q+2}}{8(q+1)}  \sum_{i = 0}^{k+k'-1}  (-1)^i B_{k, k', i} X_{(1, -1)}^{k+k'+4-(i+1)} \overline{\Psi} \otimes X_{(1, -1)}^i \overline{v}
$$
$$
\otimes   \sum_{r=0}^p \sum_{(\gamma, \gamma') \in B'(\mathbb{Q})\backslash G'(\mathbb{Q})} a_r^{(p)} \otimes a_q^{(q)} \otimes\gamma^*(\phi^p_{2r-p} \otimes \phi_f) \gamma'^*(\phi^{q}_{q+2} \otimes \phi_f')
$$
$$
-\frac{(2 \pi i)^{p+q+2}}{8(q+1)}  \sum_{i = 1}^{k+k'-1} (-1)^i C_{k, k', i} X_{(1, -1)}^{k+k'+4-i} \overline{\Psi} \otimes X_{(1, -1)}^{i-1} \overline{v}
$$
$$
\otimes    \sum_{r=0}^p \sum_{(\gamma, \gamma') \in B'(\mathbb{Q})\backslash G'(\mathbb{Q})} a_r^{(p)} \otimes a_q^{(q)} \otimes\gamma^*(\phi^p_{2r-p} \otimes \phi_f) \gamma'^*(\phi^{q}_{q+2} \otimes \phi_f').
$$
\end{lem}

\begin{proof} Let $\Omega$ denote $\Omega(\Psi_\infty) \otimes \Psi_f$ and
\begin{eqnarray*}
e_1 &=& (v^+, 0),\\
e_2 &=& (0, v^+),\\
e_3 &=& (v^-, 0),\\
e_4 &=& (0, v^-),
\end{eqnarray*}
which are vectors of $\mathfrak{g}'_\mbb{C}/\mathfrak{k}'_\mbb{C}=(\mathfrak{p'}^+ \oplus \mathfrak{p'}^-)^2$. Let  $\mathcal{S}_4$  be the symmetric group on four elements. For $\sigma \in \mathcal{S}_4$, let $\epsilon(\sigma)$ denote the signature of $\sigma$. Then,
$$
(\iota^*\Omega \wedge P^{p, q}_\mathcal{H}(\phi_f \otimes \phi'_f))(\textbf{1})= \sum_{\sigma \in \mathcal{S}_4} \epsilon(\sigma) \Omega(\iota_*(e_{\sigma(1)} \wedge e_{\sigma(2)} \wedge e_{\sigma(3)})) \otimes P^{p, q}_\mathcal{H}(\phi_f \otimes \phi'_f)(e_{\sigma(4)}).
$$
By definition, we have $\Omega(\iota_*(e_{\sigma(1)} \wedge e_{\sigma(2)} \wedge e_{\sigma(3)}))=0$ whenever $e_{\sigma(1)} \wedge e_{\sigma(2)} \wedge e_{\sigma(3)}$ is not $\pm (v^+, 0) \wedge (0, v^+) \wedge (0, v^-)$ or $\pm (v^+, 0) \wedge (0, v^+) \wedge (v^-, 0)$. So we have to compute $\Omega(\iota_*(v^+, 0) \wedge (0, v^+) \wedge (0, v^-))$ and $\Omega(\iota_*(v^+, 0) \wedge (0, v^+) \wedge (v^-, 0))$. Recall the decomposition
$$
\bigwedge^2 \mathfrak{p}^+ \otimes \mathfrak{p}^-=\tau_{(3, -1)} \oplus \tau_{(2, 0)} \oplus \tau_{(1, 1)}
$$
into irreducible $\mbb{C}[K]$-modules. By construction (Lem. \ref{def.omegaw}), $\Omega$ factors through the projection $\bigwedge^2 \mathfrak{p}^+ \otimes \mathfrak{p}^- \lra \tau_{(3, -1)}$ and sends the highest weight vector $X_{(2, 0)} \wedge X_{(1, 1)} \otimes X_{(0, -2)}$ of $\tau_{(3, -1)}$ to
$$
\sum_{i=0}^{k+k'-1} (-1)^i {X}^{k+k'+4-i}_{(1, -1)} \Psi \otimes X_{(1, -1)}^i v.
$$
Hence, we need to compute the image of $\iota_*(v^+, 0) \wedge (0, v^+) \wedge (0, v^-)$ by the projection above. The identities
\begin{eqnarray*}
\iota_*(v^\pm, 0) &=& X_{\pm(2, 0)},\\
\iota_*(0, v^\pm) &=& X_{\pm(0, 2)}
\end{eqnarray*}
imply
$$
\iota_*(v^+, 0) \wedge (0, v^+) \wedge (0, v^-)=X_{(2, 0)} \wedge X_{(0, 2)} \otimes X_{(0, -2)},
$$
which has weight $(2, 0)=(3, -1)+(-1, 1)$. Let us write
$$
\iota_*(v^+, 0) \wedge (0, v^+) \wedge (0, v^-)=\alpha x_{(3, -1)}+\beta x_{(2, 0)}+\gamma x_{(1, 1)}
$$
where $x_{(3, -1)} \in \tau_{(3, -1)}$, $x_{(2, 0)} \in \tau_{(2, 0)}, x_{(1, 1)} \in \tau_{(1, 1)}$ are weight vectors and $\alpha, \beta, \gamma \in \mbb{C}$. Because weights of $T'$ in irreducible representations of $K$ are multiplicity free, we can assume that $x_{(3, -1)}=\mrm{ad}_{X_{(-1, 1)}} \left( X_{(2, 0)} \wedge X_{(1, 1)} \otimes X_{(0, -2)} \right)$ where $\mrm{ad}$ denotes the adjoint representation $\mathfrak{g}_\mbb{C} \lra \mrm{End}(\mathfrak{g}_\mbb{C})$. As $\tau_{(1, 1)}$ does not contain vectors of weight $(2, 0)$, we have $\gamma=0$. Furthermore, as $\tau_{(2, 0)}$ has highest weight $(2, 0)$, we have
$$
\mrm{ad}_{X_{(1, -1)}} \left( \iota_*(v^+, 0) \wedge (0, v^+) \wedge (0, v^-) \right)=\alpha \mrm{ad}_{X_{(1, -1)}} \left( x_{(3, -1)} \right)
$$
i.e.
$$
\mrm{ad}_{X_{(1, -1)}} \left( X_{(2, 0)} \wedge X_{(0, 2)} \otimes X_{(0, -2)} \right)=\alpha \mrm{ad}_{X_{(1, -1)}} \left( X_{(2, 0)} \wedge X_{(1, 1)} \otimes X_{(0, -2)} \right).
$$
An easy computation shows that $\alpha=\frac{1}{4}$. As  a consequence, we have
\begin{eqnarray*}
\Omega(\iota_*(v^+, 0) \wedge (0, v^+) \wedge (0, v^-)) &=& \Omega\left(\frac{1}{4} \mrm{ad}_{X_{(-1, 1)}} \left( X_{(2, 0)} \wedge X_{(1, 1)} \otimes X_{(0, -2)} \right)\right)\\
&=& \frac{1}{4}X_{(-1, 1)}\sum_{i=0}^{k+k'-1} (-1)^i {X}^{k+k'+4-i}_{(1, -1)} \Psi \otimes X_{(1, -1)}^i v.
\end{eqnarray*}
It follows from the definition of standard basis (section \ref{discrete_series_classification}) that
\begin{eqnarray*}
X_{(-1, 1)}^m X_{(1, -1)}^{n} \Psi_\infty &=& \frac{n!}{(n-m)!} \frac{(k+k'+4-n+m)!}{(k+k'+4-n)!} X_{(1, -1)}^{n-m} \Psi_\infty,\\
X_{(-1, 1)}^m X_{(1, -1)}^n v &=& \frac{n!}{(n-m)!} \frac{(k+k'-n+m)!}{(k+k'-n)!} X_{(1, -1)}^{n-m} v.
\end{eqnarray*}
Hence
\begin{eqnarray*}
&& \Omega(\iota_*(v^+, 0) \wedge (0, v^+) \wedge (0, v^-))\\
&&= \frac{1}{4} \sum_{i=0}^{k+k'-1} (-1)^i (k+k'+4-i)(i+1) X_{(1, -1)}^{k+k'+4-(i+1)} \Psi \otimes X_{(1, -1)}^i v\\
&& + \frac{1}{4} \sum_{i=1}^{k+k'-1} (-1)^i i(k+k'-i+1) X_{(1, -1)}^{k+k'+4-i} \Psi \otimes X_{(1, -1)}^{i-1} v.
\end{eqnarray*}
To compute $\Omega(\iota_*(v^+, 0) \wedge (0, v^+) \wedge (v^-, 0))$, note that the vector
$$
\iota_*(v^+, 0) \wedge (0, v^+) \wedge (v^-, 0)=X_{(2, 0)} \wedge X_{(0, 2)} \otimes X_{(-2, 0)}
$$
has weight $(0, 2)=(3, -1)+3(-1, 1)$. A computation as above
shows that the image of the vector $\iota_*(v^+, 0) \wedge (0, v^+) \wedge (v^-, 0)$ by the natural projection $\bigwedge^2 \mathfrak{p}^+ \otimes \mathfrak{p}^- \lra \tau_{(3, -1)}$ is equal to $\frac{3}{80} \mrm{ad}_{X_{(-1, 1)}}^3 \left( X_{(2, 0)} \wedge X_{(1, 1)} \otimes X_{(0, -2)} \right)$. Hence
\begin{eqnarray*}
&& \Omega(\iota_*(v^+, 0) \wedge (0, v^+) \wedge (v^-, 0))\\
&=& \Omega\left(\frac{3}{80} \mrm{ad}_{X_{(-1, 1)}}^3 \left( X_{(2, 0)} \wedge X_{(1, 1)} \otimes X_{(0, -2)} \right) \right)\\
&=& \frac{3}{80}X_{(-1, 1)}^3\sum_{i \geq 0} (-1)^i {X}^{k+k'+4-i}_{(1, -1)} \Psi \otimes X_{(1, -1)}^i v\\
&=& \frac{3}{80}\sum_{i = 0}^{k+k'-1} \sum_{j=0}^3 (-1)^i \binom{3}{j}A_{k,k',i,j} {X}^{k+k'+4-(i+j)}_{(1, -1)} \Psi \otimes X_{(1, -1)}^{i-j} v.
\end{eqnarray*}
To finish the proof, we have to compute $P^{p, q}_\mathcal{H}(\phi_f \otimes \phi'_f))(v^-, 0)$ and $P^{p, q}_\mathcal{H}(\phi_f \otimes \phi'_f))(0, v^-)$. The differential form $P^{p, q}_\mathcal{H}(\phi_f \otimes \phi'_f))$ is defined in Prop. \ref{eisenstein-deligne}. Unravelling the definitions, we get
\begin{eqnarray*}
P^{p, q}_\mathcal{H}(\phi_f \otimes \phi'_f))(v^-, 0) &=& (-1)^p \sum_{(\gamma, \gamma') \in B'(\mathbb{Q})\backslash G'(\mathbb{Q})} \gamma^*\pi_p(\omega_p^+ \otimes \phi_f)(v^-) \gamma'^*(\Theta_q \otimes \phi_f').
\end{eqnarray*}
According to Lem. \ref{cc_omega}, we have $\pi_p(\omega_p^+ \otimes \phi_f) = \frac{1}{2}(\omega_p^+ + \omega_p^-) \otimes \phi_f$. Hence
$$
P^{p, q}_\mathcal{H}(\phi_f \otimes \phi'_f))(v^-, 0)
$$
$$= \frac{(-1)^p(2 \pi i)^{p+1}(2\pi  i)^{q+1}}{4(q+1)} \sum_{s=0}^q \sum_{(\gamma, \gamma') \in B'(\mathbb{Q})\backslash G'(\mathbb{Q})} a_0^{(p)} \otimes a_s^{(q)} \otimes\gamma^*(\phi^p_{-p-2} \otimes \phi_f) \gamma'^*(\phi^{q}_{2s-q} \otimes \phi_f').
$$
Similarly,
$$
P^{p, q}_\mathcal{H}(\phi_f \otimes \phi'_f))(0, v^-)
$$
$$
=\frac{(2 \pi i)^{p+1}(2 \pi i)^{q+1}}{4(p+1)}  \sum_{r=0}^p \sum_{(\gamma, \gamma') \in B'(\mathbb{Q})\backslash G'(\mathbb{Q})} a_r^{(p)} \otimes a_0^{(q)} \otimes \gamma^*(\phi^{p}_{2r-p} \otimes \phi_f)  \gamma'^*(\phi^q_{-q-2} \otimes \phi'_f).
$$
The proof of the first statement follows by putting the previous computations together. The proof of the second statement is a direct consequence of the first, of the identities
\begin{eqnarray*}
\mrm{Ad}_{N^{-1}}(v^\pm, 0) &=& (0, v^\mp),\\
\mrm{Ad}_{N^{-1}}(0, v^\pm) &=& (v^\mp, 0)
\end{eqnarray*}
and of the formula for $P^{p, q}_\mathcal{H}(\phi_f \otimes \phi'_f))(v^+, 0)$ and $P^{p, q}_\mathcal{H}(\phi_f \otimes \phi'_f))(0, v^+)$ deduced as above from Prop. 4.9.
\end{proof}

According to Lem. \ref{valeurs_Omega}, to compute the integral $$\int_{G'(\mbb{Q}) Z'(\mbb{A}) \backslash G'(\mathbb{A})} \langle (\iota^*\left( \Omega(\Psi_\infty) \otimes \psi_f \right)\wedge P^{p, q}_\mathcal{H}(\phi_f \otimes \phi'_f))(\textbf{1}) \rangle (g) |det g|^3 dg,$$ we need to compute pairings of the shape $\langle X^i_{(1, -1)}v ,  a_r^{(p)} \otimes a_0^{(q)} \rangle, \langle X^i_{(1, -1)}v ,  a_0^{(p)} \otimes a_s^{(q)} \rangle$ and of the shape $\langle X^i_{(1, -1)} \overline{v},  a_p^{(p)} \otimes a_s^{(q)} \rangle, \langle X^i_{(1, -1)}v,  a_r^{(p)} \otimes a_q^{(q)} \rangle$. In fact, a lot of these pairings vanish for weight reasons.

\begin{lem} \label{nullite-accouplement} Let $i, r, s$ be integers such that $i \geq 0, p \geq r \geq 0, q \geq s \geq 0$. Then, the following statements hold
\begin{itemize}
\item if $\langle X^i_{(1, -1)}v,  a_r^{(p)} \otimes a_0^{(q)} \rangle \neq 0$ then $ i=k'+q, r=(-k+k'+p+q)/2$,
\item if $\langle X^i_{(1, -1)}v,  a_0^{(p)} \otimes a_s^{(q)} \rangle \neq 0$ then $i=k-p, s=(-k+k'+p+q)/2$,
\item if $\langle X^i_{(1, -1)}\overline{v},  a_p^{(p)} \otimes a_s^{(q)} \rangle \neq 0$ then $ i=k'+p, s=(-k+k'+p+q)/2$,
\item if $\langle X^i_{(1, -1)}\overline{v},  a_r^{(p)} \otimes a_q^{(q)} \rangle \neq 0$ then $i=k+q, r=(-k+k'+p-q)/2$.
\end{itemize}
\end{lem}

\begin{proof} The vector $a_r^{(p)} \boxtimes a_s^{(q)}$ has weight $\lambda'(p-2r, q-2s)$ for the action of $\mrm{U}(1)^2=T'$ (Lem. \ref{poinds-a_j}), the vector $ X^i_{(1, -1)}v$, respectively $X^i_{(1, -1)}\overline{v}$, has weight $\lambda'(-k, k')+i\lambda'(1, -1)$, respectively  $\lambda'(-k', k)+i\lambda'(1, -1)$. Hence, the statement follows from the fact that, if two weight vectors pair non-trivially, they have opposite weights.
\end{proof}

\begin{cor} \label{resume} Let $\Psi=\Psi_\infty \otimes \Psi_f$ be as above, let $w \in W(-p-q-3)$ be a vector of weight $\lambda(-k, k', -p-q)$ and let $\omega$ be the vector in $M_B(\check{\pi}_f|\nu|^{-3}, W(-p-q-3))_\mbb{C}$ associated to $\Psi$ and $w$ by Lem. \ref{def.omegaw}. Let $\phi_f \otimes \phi'_f \in \mathcal{B}_p \otimes_\mbb{Q} \mathcal{B}_q$ and let $\rho$ be a closed current lifting the extension class $Eis_\mathcal{H}^{p, q, k, k'}(\phi_f \otimes \phi'_f)$ as in Lem. \ref{relevement}. Let $\Xi_{n, r, s}(\phi_f, \phi'_f)$ and $\overline{\Xi}_{n, r, s}(\phi_f, \phi'_f)$ denote the functions on $G'(\mbb{A})$ defined by
\begin{eqnarray*}
\Xi_{n, r, s}(\phi_f, \phi'_f) &=&X_{(1, -1)}^{n} \Psi \sum_{(\gamma, \gamma')} \gamma^*(\phi_{r}^p \otimes \phi_f) \gamma'^*(\phi_{s}^q \otimes \phi_f'),\\
\overline{\Xi}_{n, r, s}(\phi_f, \phi'_f) &=&X_{(1, -1)}^{n} \overline{\Psi} \sum_{(\gamma, \gamma')} \gamma^*(\phi_{r}^p \otimes \phi_f) \gamma'^*(\phi_{s}^q \otimes \phi_f')
\end{eqnarray*}
where the sums are indexed by all $(\gamma, \gamma') \in B'(\mbb{Q}) \backslash G'(\mbb{Q})$. Let
\begin{eqnarray*}
C_1 &=& \langle X^{k'+q}_{(1, -1)}v,  a_r^{(p)} \otimes a_0^{(q)} \rangle,\\
C_2 &=& \langle X^{k-p}_{(1, -1)}v,  a_r^{(p)} \otimes a_0^{(q)} \rangle,\\
C_3 &=& \langle X^{k'+p}_{(1, -1)}\overline{v},  a_p^{(p)} \otimes a_s^{(q)} \rangle,\\
C_4 &=& \langle X^{k+q}_{(1, -1)}\overline{v},  a_p^{(p)} \otimes a_s^{(q)} \rangle.
\end{eqnarray*}
Then, the pairing $\langle \omega, [\rho] \rangle_B$ equals
\begin{eqnarray*}
&& C_1 \frac{3}{160(p+1)} \sum_{j=0}^3(-1)^{k'+q+j}\binom{3}{j} A_{k, k', k'+q+j, j}\int \Xi_{k-q-2j+4, -k+k'+q, -q-2}(\phi_f, \phi'_f)\\
&-& C_2 \frac{(-1)^k}{8(q+1)}(B_{k, k', k-p}-C_{k, k', k-p+1}) \int \Xi_{k'+p+3, -p-2, -k+k'+p}(\phi_f, \phi'_f)\\
&+& C_3 \frac{3}{160 (q+1)} \sum_{j=0}^3 (-1)^{k'+j} \binom{3}{j} A_{k, k', k'+p+j, j} \int \overline{\Xi}_{k-p-2j+4, p+2, -k+k'+p}(\phi_f, \phi'_f)\\
&-& C_4 \frac{ (-1)^{k'+p+1}}{8(q+1)} (B_{k, k', k'+p}-C_{k, k', k'-p+1}) \int \overline{\Xi}_{k-p+3, -k+k'-q, q+2}(\phi_f, \phi'_f),
\end{eqnarray*}
where the numbers $A_{k, k', i, j}, B_{k, k', i}$ and $C_{k, k', i}$ are defined in Lem. \ref{valeurs_Omega} and where the integrals are over $G'(\mbb{Q}) Z'(\mbb{A}) \backslash G'(\mathbb{A})$.
\end{cor}

\begin{proof} Direct consequence of Prop. \ref{crucial}, Lem. \ref{valeurs_Omega} and Lem. \ref{nullite-accouplement}.
\end{proof}

To compute the constants $C_1, C_2, C_3$ and $C_4$, let us start with a trivial remark. The family $\left( a_r^{(p)} \boxtimes a_s^{(q)} \right)_{0 \leq r \leq p, 0 \leq s \leq q}$ is a weight basis of $\mrm{Sym}^p V_{2, \mbb{C}}^\vee \boxtimes \mrm{Sym}^q V_{2, \mbb{C}}^\vee$ for the action of $\mrm{U}(1)^2$ and, by definition, the vector $v$ is a weight vector for the maximal compact torus $T' \subset G(\mbb{R})$. Hence, the vector $X^i_{(1, -1)}v \in W(-p-q-3)$ is again a weight vector and, because $\mrm{U}(1)^2=T'$, its image under the $\mrm{U}(1)^2$-equivariant projection
$$
\varrho: \iota^*W(-p-q-3) \lra \mrm{Sym}^p V_{2, \mbb{C}}^\vee \boxtimes \mrm{Sym}^q V_{2, \mbb{C}}^\vee
$$
is equal to $\lambda_i(v) a_{r_i}^{(p)} \boxtimes a_{s_i}^{(q)}$ for some complex number $\lambda_i(v)$ and some integers $r_i, s_i$.\\

Given $i$, is $\lambda_i(v)$ zero or not ? Note that we have the liberty to multiply the vector $w$ which enters the definition of the cohomology class $\omega$ associated to the cusp form $\Psi$ (Lem \ref{def.omegaw}) by any non-zero complex number. Of course, this has the effect to multiply $\omega$ by a scalar, but does not change the ratio $\frac{\langle \omega,\widetilde{v}_\mathcal{K} \rangle_B}{\langle \omega,\widetilde{v}_\mathcal{D} \rangle_B}$ that has to be computed (see Lem. \ref{dualite}). If $\lambda_i(v) \neq 0$, we can replace the vector $w$ by $\lambda_i(v)^{-1} w$. This has the effect to replace the vector $v=Jw$ by $\lambda_i(v)^{-1} v$. For this choice, $\varrho(X^i_{(1, -1)}v)=a_{r_i}^{(p)} \boxtimes a_{s_i}^{(q)}$ and so, the computation of the pairing follow (see (\ref{accouplement})).

\begin{lem} \label{technique} Let $p=k-1$ and $q=k'-1$. Then, for any non zero vector $v \in W(-p-q-3)$ of weight  $\lambda'(-k, k', -p-q)$, the image of $X_{(1, -1)}v$ under the $G'$-equivariant projection
$$
\varrho: \iota^*W(-p-q-3) \lra \mrm{Sym}^p V_{2, \mbb{C}}^\vee \boxtimes \mrm{Sym}^q V_{2, \mbb{C}}^\vee
$$
is non-zero.
\end{lem}

\begin{proof} 
Thanks to the explanations preceding Lem. \ref{def.omegaw} and by symmetry under the Weyl group, we are reduced to prove the following statement. If $X_{-(\rho_1+\rho_2)} \in \mathfrak{g}$ is a root vector associated to $-(\rho_1+\rho_2)$ and if $v \in W(-p-q-3)$ is a highest weight vector for the action of the diagonal maximal torus $T$, then $\rho(X_{-(\rho_1+\rho_2)}v) \neq 0$. Let $X_{-\rho_1}, X_{-\rho_2} \in \mathfrak{g}$ be root vectors associated to the simple negative roots. We can assume that $X_{-\rho_2} \in \mathfrak{g}'$ and that $X_{-(\rho_1+\rho_2)}=[X_{-\rho_1}, X_{-\rho_2}]$. The vector $X_{-\rho_1}v$ is a highest weight vector of the representation $\mrm{Sym}^{k-1} V_{2, \mbb{C}}^\vee \boxtimes \mrm{Sym}^{k'-1} V_{2, \mbb{C}}^\vee$ which is a subrepresentation of $\iota^*W(-k-k'-1)$. Hence, we have $\varrho(X_{-\rho_2} X_{-\rho_1}v)=0$. It is well known that $W(-k-k'-1)$ is generated by applying monomials in the variables $X_{-\rho_1}$ and $X_{-\rho_2}$ to $v$. Furthermore, $X_{-\rho_2}X_{-\rho_1}$ and $X_{-\rho_1}X_{-\rho_2}$ are the only two such monomials whose application to $v$ gives weight $\lambda(k-1, k'-1, -k-k'-2)$, which is the highest weight of the representation $\mrm{Sym}^{k-1} V_{2, \mbb{C}}^\vee \boxtimes \mrm{Sym}^{k'-1} V_{2, \mbb{C}}^\vee$. So, the fact that $\varrho(X_{-\rho_2} X_{-\rho_1}v)=0$ implies that $\varrho(X_{-\rho_1} X_{-\rho_2}v) \neq 0$ which implies $\varrho(X_{-(\rho_1+\rho_2)}v) \neq 0$.
\end{proof}

\begin{rems} \label{molev} In \cite{molev} Thm. 9.6.2, Molev constructs basis of irreducible representations $W$ of $G$ and is able to derive explicit formulas for the action of generators of the Lie algebra of $G$ on the vectors of the basis. The non-vanishing of $\lambda_i(v)$ can be derived from his result at the price of a very lengthy (but elementary) calculation. Thanks to a program implemeting Molev's result and written in Python by Molin, the author verifed numerically the non-vanishing of $\lambda_i(v)$ for some small values of $p, q, k$ and $k'$.
\end{rems}

\begin{cor} \label{cor-technique} Assume $p=k-1$ and $q=k'-1$. There exists $v \in W(-p-q-3)$ of weight $\lambda'(-k, k', -p-q)$ such that $\langle X_{(1, -1)}v, a_0^{(p)} \boxtimes a_q^{(q)} \rangle =(-1)^p (2i)^{-k-k'-2}$.
\end{cor}

\section{Computation of the integral} \label{calcul-de-l-integrale} Given a cuspidal automorphic representation $\pi$ of $G$, the spinor $L$-function is defined as the partial Euler product
$$
L_V(s, \pi, r)=\prod_{v \notin V} L(s, \pi_v, r)
$$
where $V$ denotes the set of places where $\pi$ is ramified and where $r: {}^LG^0 \simeq G \lra \mrm{GL}(4)$ is the natural inclusion. This Euler product is absolutely convergent for $Re\,s$ big enough. In \cite{piatetski-shapiro}, the analytic continuation and a functional equation of $L_V(s, \pi, r)$ are deduced from the analytic continuation and a functional equation of a family of Eisenstein series $E^\Phi$, via an integral representation. In this section, as a direct consequence of Prop. \ref{comparaison}, we show that for suitable choice of data, the integrals that appear in the statement of Cor. \ref{resume} coincide with a special value of such an integral representation. Via Bessel models, the integrals of \cite{piatetski-shapiro} expand into Euler products and we perform the relevant local unramified computation in Prop. \ref{integrale-nonram}. The ramified non-archimedean integrals are considered in Prop. \ref{integrales-ram}. The archimedean integral is computed Prop. \ref{integrale-archi}.\\


\subsection{Comparison with Piatetski-Shapiro's integral} Let $d t_\infty$ be the Lebesgue measure on the additive group $\mbb{R}$. If $v$ is a non-archimedean place of $\mbb{Q}$, let $d t_v$ be the measure for which $\mbb{Z}_v$ has volume one. Following \cite{tate}, let $d^\times t_v$ be the Haar measure on $\mbb{Q}_v^\times$ defined by
$$
d^\times t_v = \left\{
\begin{array}{ll}
        \frac{dt_v}{|t_v|} & \mbox{if } v \mbox{ is archimedean},\\
        \frac{p}{p-1} \frac{dt_v}{|t_v|} & \mbox{if } v \mbox{ is } p\mbox{-adic}.
\end{array}
\right.
$$
Let $d t$, respectively $d^\times t$, denote the product measure $\prod_v d t_v$ on $\mbb{A}$, respectively the product measure $\prod_v d^\times t_v$ on $\mbb{A}^\times$.

\begin{pro} \label{PS1} Let $\mu, \nu_1, \nu_2: \mbb{Q}^\times \backslash \mbb{A}^\times \lra \mbb{C}^\times$ be continuous characters and let $s \in \mbb{C}$. Let $\chi_{\mu, \nu_1, \nu_2, s}$ be the character of $B'(\mbb{A})$ defined by
$$
\chi_{\mu, \nu_1, \nu_2, s} \left( \begin{pmatrix}
a_1 & b_1\\
 & d_1\\
\end{pmatrix}, \begin{pmatrix}
a_2 & b_2\\
 & d_2 \\
\end{pmatrix}\right)=\mu(a_1/d_2)|a_1/d_2|^{s+\frac{1}{2}}\nu_1^{-1}(d_1) \nu_2^{-1}(d_2).
$$
Then, for any Schwartz-Bruhat function $\Phi $ on $\mbb{A}^4$, the following statements hold:
\begin{itemize}
\item the function on $G'(\mbb{A})$ defined by
$$
(g_1, g_2) \longmapsto f^\Phi(g_1, g_2, \mu, \nu_1, \nu_2, s)=
$$
$$
\mu(\det g_1) |\det g_1|^{s+\frac{1}{2}} \int_{\mbb{A}^\times}\int_{\mbb{A}^\times} \Phi((0, t_1)g_1, (0, t_2)g_2) |t_1t_2|^{s+\frac{1}{2}} \mu(t_1 t_2) \nu_1(t_1) \nu_2(t_2) d^\times t_1 d^\times t_2
$$
belongs to $\mrm{ind}_{B'(\mbb{A})}^{G'(\mbb{A})} \chi_{\mu, \nu_1, \nu_2, s},$
\item the Eisenstein series
$$
E^\Phi(g_1, g_2, \mu, \nu_1, \nu_2, s)= \sum_{(\gamma, \gamma') \in B'(\mathbb{Q})\backslash G'(\mathbb{Q})} f^\Phi \left(\gamma g_1, \gamma' g_2, \mu, \nu_1, \nu_2, s \right)
$$
is absolutely convergent for $\mbox{Re}\, s$ big enough and satisfies a functional equation.
\end{itemize}
\end{pro}

\begin{proof} The first statement follows from a trivial computation and the second from  \cite{piatetski-shapiro} Thm. 5.1.
\end{proof}

When $\mu$ is the trivial character, we denote $\chi_{\mu, \nu_1, \nu_2, s}$ by $\chi_{\nu_1, \nu_2, s}$ for simplicity.\\


Let us recall that for $d \equiv c \pmod 2$, the algebraic character $\lambda(d, c)$ of the diagonal maximal torus of $\mrm{GL}_2$ is defined by
$$
\lambda(d, c): \begin{pmatrix}
\alpha & \\
  & \alpha^{-1} \nu \\
\end{pmatrix} \longmapsto \alpha^{d} \nu^\frac{c-d}{2}.
$$

\begin{lem} \label{comp-char}  Let $\nu_1^0$ and $\nu_2^0$ be two finite order Hecke characters of respective signs $(-1)^p$ and $(-1)^q$. Let $\nu_1$ denote the Hecke character $|\,|^{-q} \nu_1^0$ and $\nu_2$ denote the Hecke character $|\,|^{-p} \nu_2^0$. Then, the following statements are satisfied:
\begin{itemize}
\item the restriction of the archimedean part of $\chi_{\nu_1, \nu_2, p+q+3/2}$ to the identity component of the diagonal maximal torus of $G'(\mbb{R})$ is $\lambda(p+2, p) \boxtimes \lambda(q+2, q)$,
\item the non-archimedean part $\chi_{f}$ of $\chi_{\nu_1, \nu_2, p+q+3/2}$ verifies
$$
\chi_{f}\left(\begin{pmatrix}
a_1 \alpha_1 & b_1 \\
 & d_1 \delta_1\\
\end{pmatrix}, \begin{pmatrix}
a_2 \alpha_2 & b_2\\
 & d_2 \delta_2\\
\end{pmatrix} \right)=a_1^{-(p+1)}d_1 a_2^{-(q+1)}d_2 \nu_1^0(\delta_1)^{-1} \nu_2^0(\delta_2)^{-1}.
$$
for any $a_1, d_1, a_2, d_2 \in \mbb{Q}^+, \alpha_1, \delta_1, \alpha_2, \delta_2 \in \widehat{\mbb{Z}}^\times$ such that $a_1 d_1 \alpha_1 \delta_1=a_2 d_2 \alpha_2 \delta_2$ and any $b_1, b_2 \in \mbb{A}_f$.
\end{itemize}
\end{lem}

\begin{proof} Let
$$
\left( \begin{pmatrix}
a_1 & b_1\\
 & d_1\\
\end{pmatrix}, \begin{pmatrix}
a_2 & b_2\\
 & d_2 \\
\end{pmatrix}\right) \in G'(\mbb{A}).
$$
Because $a_1d_1=a_2d_2$, we have the identities
\begin{eqnarray*}
\chi_{\nu_1, \nu_2, p+q+3/2} \left( \begin{pmatrix}
a_1 & b_1\\
  & d_1\\
\end{pmatrix}, \begin{pmatrix}
a_2 & b_2\\
  & d_2\\
\end{pmatrix}\right) &=& |a_1/d_2|^{p+q+2}|d_1|^{q}|d_2|^{p}\nu_1^0(d_1)^{-1} \nu_2^0(d_2)^{-1}\\
&=& |a_1|^{p+2} |a_1d_1|^{-1} \nu_1^0(d_1)^{-1} |a_2|^{q+2} |a_2 d_2|^{-1} \nu_2^0(d_2)^{-1}.
\end{eqnarray*}
For a diagonal element
$$
\left( \begin{pmatrix}
a_1 & \\
  & d_1\\
\end{pmatrix}, \begin{pmatrix}
a_2 & \\
  & d_2\\
\end{pmatrix}\right) \in G'(\mbb{R})^+,
$$
we have $a_1 d_1=a_2 d_2 >0$, hence
$$
\chi_{\nu_1, \nu_2, p+q+3/2} \left( \begin{pmatrix}
a_1 & \\
  & d_1\\
\end{pmatrix}, \begin{pmatrix}
a_2 & \\
  & d_2\\
\end{pmatrix}\right)
$$
\begin{eqnarray*}
&=& |a_1|^{p+2} |a_1d_1|^{-1} \mrm{sgn}(d_1)^p |a_2|^{q+2} |a_2 d_2|^{-1} \mrm{sgn}(d_2)^q\\
&=& |a_1|^{p+2} (a_1d_1)^{-1} \mrm{sgn}(a_1)^p |a_2|^{q+2} (a_2 d_2)^{-1} \mrm{sgn}(a_2)^q\\
&=& a_1^{p+2} (a_1d_1)^{-1} a_2^{q+2} (a_2 d_2)^{-1}\\
&=& (\lambda(p+2, p) \boxtimes \lambda(q+2, q)) \left(  \begin{pmatrix}
a_1 & \\
  & d_1\\
\end{pmatrix}, \begin{pmatrix}
a_2 & \\
  & d_2\\
\end{pmatrix}\right).
\end{eqnarray*}
This proves the first statement. Let $\chi_\infty$ denote the archimedean part of $\chi_{\nu_1, \nu_2, p+q+3/2}$. Then, the second statement follows from the equalities
$$
\chi_{f}\left(\begin{pmatrix}
a_1 \alpha_1 & b_1 \\
 & d_1 \delta_1\\
\end{pmatrix}, \begin{pmatrix}
a_2 \alpha_2 & b_2\\
 & d_2 \delta_2\\
\end{pmatrix} \right)
$$
\begin{eqnarray*}
&=& \chi_{f}\left(\begin{pmatrix}
a_1  & \\
 & d_1 \\
\end{pmatrix}, \begin{pmatrix}
a_2 & \\
 & d_2 \\
\end{pmatrix} \right) \chi_{f}\left(\begin{pmatrix}
\alpha_1 & b_1/a_1 \\
 & \delta_1\\
\end{pmatrix}, \begin{pmatrix}
\alpha_2 & b_2/a_2\\
 & \delta_2\\
\end{pmatrix} \right) \\
&=& \chi_\infty \left(\begin{pmatrix}
a_1  & \\
 & d_1 \\
\end{pmatrix}, \begin{pmatrix}
a_2 & \\
 & d_2 \\
\end{pmatrix} \right)^{-1} \chi_{f}\left(\begin{pmatrix}
\alpha_1 & b_1/a_1 \\
 & \delta_1\\
\end{pmatrix}, \begin{pmatrix}
\alpha_2 & b_2/a_2\\
 & \delta_2\\
\end{pmatrix} \right)\\
&=& (a_1^{p+2} (a_1 d_1)^{-1} a_2^{q+2} (a_2 d_2)^{-1})^{-1}\nu_1^0(\delta_1)^{-1} \nu_2^0(\delta_2)^{-1}.
\end{eqnarray*}
\end{proof}

The following result obviously implies that, for suitable choice of data, the Eisenstein series appearing in the integrals of Cor. \ref{resume} coincide with a special value of the one defined above.

\begin{pro} \label{comparaison} Let $\nu_1^0$ and $\nu_2^0$ be two finite order Hecke characters of respective signs $(-1)^p$ and $(-1)^q$. Let $\nu_1$ denote the Hecke character $|\,|^{-q} \nu_1^0$ and $\nu_2$ denote the Hecke character $|\,|^{-p} \nu_2^0$. Let $\Phi_{1, f}=\prod_{v<\infty} \Phi_{1, v}$ and $\Phi_{2, f}=\prod_{v<\infty} \Phi_{2, v}$ be factorizable Schwartz-Bruhat functions on $\mbb{A}_f^2$ such that for $j=1, 2$ and any $v$, the function $\Phi_{j, v}$ is $\overline{\mbb{Q}}$-valued.  Let $S_j$ be the set of places where $\nu_j$ and $\Phi_{j, f}$ are ramified. Then, there exist $\phi_f \in \mathcal{B}_{p, \overline{\mbb{Q}}}$ and $\phi'_f \in \mathcal{B}_{q, \overline{\mbb{Q}}}$ such that the following statement is verified. For any integers $r \equiv p \pmod 2$ and $s \equiv q \pmod 2$, let $\Phi$ be  the Schwartz-Bruhat function on $\mbb{A}^4$ defined by
$$
\Phi(x_1, y_1, x_2, y_2) =
$$
$$
\frac{(-1)^{p+q+r+s}(2i)^{p+q+2}\pi^{2(p+q)}}{\left( (p+q-1)!\right)^2 L_{S_1}(p+q+2, \nu_1) L_{S_2}(p+q+2, \nu_2)} \Phi_{1}(x_1, y_1) \Phi_{2}(x_2, y_2).
$$
Here, we denote by $\Phi_1$ and $\Phi_2$ the factorizable Schwartz-Bruhat functions $\Phi_1 = \prod_v \Phi_{1, v}$ and $\Phi_2 = \prod_v \Phi_{2, v}$ on $\mbb{A}^2$ where
\begin{eqnarray*}
\Phi_{1, \infty}(x_1, y_1) &=& (ix_1+y_1)^\frac{p-r}{2} (ix_1-y_1)^\frac{p+r}{2} e^{-\pi(x_1^2+y_1^2)},\\
\Phi_{2, \infty}(x_2, y_2) &=& (ix_2+y_2)^\frac{q-s}{2} (ix_2-y_2)^\frac{q+s}{2} e^{-\pi(x_2^2+y_2^2)}.
\end{eqnarray*}
Then, for any $(g_1, g_2) \in G'(\mbb{A})$, we have
$$
(\phi^p_{r} \otimes \phi_f)(g_1) (\phi^{q}_{s} \otimes \phi_f')(g_2)=f^\Phi((g_1, g_2), 1, \nu_1, \nu_2, p+q+3/2).
$$
\end{pro}

\begin{proof} For any place $v$ and any $g_v \in \mrm{GL}_2(\mbb{Q}_v)$ let
$$
Z_v^{\Phi_{j, v}}(g_v, \nu_{j, v}, s)=\int_{\mbb{Q}_v^\times} \Phi_{j, v}((0, t)g_v) |t|^{s+\frac{1}{2}} \nu_{j, v}(t)d^\times t.
$$
Then, we have the following factorization into an Euler product of local Tate integrals
$$
\left( \frac{(-1)^{p+q+r+s}(2i)^{p+q+2}\pi^{2(p+q)}}{\left( (p+q-1)!\right)^2 L_{S_1}(p+q+2, \nu_1) L_{S_2}(p+q+2, \nu_2)} \right)^{-1} f^\Phi(g_1, g_2, 1, \nu_1, \nu_2, p+q+3/2)
$$
$$
=|\mrm{det} g_1|^{p+q+2} \prod_{j=1}^2  \left( \prod_{v} Z_v^{\Phi_{j, v}}(g_{j, v}, \nu_{j, v}, p+q+3/2) \right).
$$
At the archimedean place, we have
$$
Z_\infty^{\Phi_{1, \infty}}(1, \nu_{1, \infty}, p+q+3/2) Z_\infty^{\Phi_{2, \infty}}(1, \nu_{2, \infty}, p+q+3/2)= (-1)^\frac{p+q+r+s}{2} \pi^{-2(p+q)} \Gamma(p+q)^2
$$
where $\Gamma$ is the gamma function. The function $\Phi_\infty$ has weight $\lambda'(r, s)$ for the action of $\mrm{U}(1)^2$, which is the same weight as $\phi_r^p \phi_s^q$. The Iwasawa decomposition for $G'(\mbb{R})$ implies that the weights for $\mrm{U}(1)^2$ in $\mrm{ind}_{B'(\mbb{R})}^{G'(\mbb{R})}(\lambda(p+2, p) \boxtimes \lambda(q+2, q))$ are multiplicity free.  Hence, it follows from the first statement of Prop. \ref{PS1} and from Lem. \ref{comp-char}, that the archimedean part
$$
\frac{(-1)^{p+q+r+s}(2i)^{p+q+2}\pi^{2(p+q)}}{\left( (p+q-1)!\right)^2} Z_\infty^{\Phi_{1, \infty}}(g_{1, \infty}, \nu_{1, \infty}, p+q+3/2) Z_\infty^{\Phi_{2, \infty}}(g_{2, \infty}, \nu_{2, \infty}, p+q+3/2)
$$
of $f^\Phi$ is proportional to $\phi_r^p \phi_s^q$. By the choice of our normalization factor, the two are in fact equal. We claim that the non-archimedean part $f^{\Phi_f}$ belongs to $\mathcal{B}_{p, \overline{\mbb{Q}}} \otimes \mathcal{B}_{q, \overline{\mbb{Q}}}$. To prove this, let $v$ be a non-archimedean place. By the first statement of Prop. \ref{PS1} and the Iwasawa decomposition $\mrm{GL}_2(\mbb{Q}_v)=\mrm{B}_2(\mbb{Q}_v) \mrm{GL}_2(\mbb{Z}_v)$, the function $$g \longmapsto Z_v^{\Phi_{j, v}}(g, \nu_{j, v}, p+q+3/2)$$ is determined by its restriction to $\mrm{GL}_2(\mbb{Z}_v)$. Assume that $v$ does not belong to $S_1 \cup S_2$, in other words, for $j=1, 2$, the character $\nu_j$ is unramified at $v$ and $\Phi_{j, v}$ is the indicator function of $\mbb{Z}_v^2$. Let us show that $g \mapsto Z_v^{\Phi_{j, v}}(g, \nu_{j, v}, p+q+3/2)$ is constant on  $\mrm{GL}_2(\mbb{Z}_v)$ and let us compute its value. Let $g=\begin{pmatrix}
a & b\\
c  & d\\
\end{pmatrix} \in \mrm{GL}_2(\mbb{Z}_v)$. As $c$ is coprime to $d$, for any non-negative integer $n$ and any $t \in \mbb{Q}_v^\times$, we have
$$
(0, t) \in \mbb{Z}_v^2 \iff (t c_v, t d_v) \in \mbb{Z}_v^2
$$
and this implies $
\Phi_{j, v}((0, t)g)=\Phi_{j, v}(t c, t d)=\Phi_{j, v}(0, t)$. Hence, by a standard computation, for any $j=1, 2$, any non-archimedean $v \notin S_1 \cup S_2$ and any $g \in \mrm{GL}_2(\mbb{Z}_v)$, we have
$$
Z_v^{\Phi_{j, v}}(g, \nu_{j, v}, p+q+3/2)=L_v(p+q+2, \nu_j)
$$
where $L_v(s, \nu_j)$ is the local $L$-factor (see \cite{tate} 2.5 p. 320). Moreover, it follows from the computations of \cite{tate} 2.5 p. 321 that, for any non-archimedean $v \in S_1 \cup S_2$ and any $g \in \mrm{GL}_2(\mbb{Z}_v)$, we have
$$
Z_v^{\Phi_{j, v}}(g, \nu_{j, v}, p+q+3/2) \in \overline{\mbb{Q}}.
$$
Hence, by our choice of the normalization factor $1/(L_{S_1}(p+q+2, \nu_1) L_{S_2}(p+q+2, \nu_2))$,
the non-archimedean part $f^{\Phi_f}$  of $f^\Phi$ is $\overline{\mbb{Q}}$-valued. Furthermore, it is obviously invariant by right translation by the subgroup $\begin{pmatrix}
\widehat{\mbb{Z}}^\times & \\
  & 1\\
\end{pmatrix}$ of $\mrm{GL}_2(\widehat{\mbb{Z}})$.
As a consequence, it follows from Lem. \ref{eiseinstein-ext-scalaires} and from the second statement of Lem. \ref{comp-char} that $f^{\Phi_f}$ belongs to  $\mathcal{B}_{p, \overline{\mbb{Q}}} \otimes \mathcal{B}_{q, \overline{\mbb{Q}}}$. The conclusion follows.
\end{proof}

\subsection{Bessel models and local computations} The previous result shows that the integrals of Cor. \ref{resume}, which compute the regulator, coincide with special values of integrals of the shape
$$
\int_{G'(\mathbb{Q})Z'(\mbb{A})\backslash G'(\mathbb{A})} \Psi(g) E^\Phi(g, \mu, \nu_1, \nu_2, s) dg
$$
for some specific choices of $\Psi$ and $\Phi$. The properties of these integrals rely on the Fourier expansion of the cusp form $\Psi$ along the Siegel parabolic subgroup
$$
P=\left\{ \begin{pmatrix}
\alpha A & AS\\
  & ^t\!\!A^{-1}\\
\end{pmatrix}, \alpha \in \mbb{G}_m, A \in \mrm{GL}_2, ^t\!\!S=S
\right\}
$$
of $G$. More precisely, they rely on the existence of Bessel models for cuspidal automorphic representations of $G$. Let us remark that some authors use the terminology "generalized Whittaker model" rather than  "Bessel model". As a motivation for the study of such objects, the reader might find the first section of \cite{moriyama} very interesting. To introduce Bessel models, let $\eta: \mbb{Q} \backslash \mbb{A} \longrightarrow \mbb{C}^\times$ be a fixed additive character, let $U$ denote the unipotent radical of $P$ and let $\Lambda: U(\mbb{Q}) \backslash U(\mbb{A}) \lra \mbb{C}^\times$ denote the character defined by
$$
\Lambda \left(\begin{pmatrix}
1 &  & r & t\\
 & 1 &  t & s\\
 &  &  1 & \\
 &  &  & 1\\
\end{pmatrix} \right)=\eta(t).
$$
Introduce the following subgroups of $G$:
\begin{eqnarray*}
D &=& \left\{ d=\begin{pmatrix}
d_1 &  &  & \\
 & d_2 &  & \\
 &  &  d_2 & \\
 &  &  & d_1\\
\end{pmatrix}, d_1, d_2 \in \mbb{G}_m \right\},\\
N &=& \left\{ n=\begin{pmatrix}
1 &  & r  & \\
 & 1 &  & s \\
 &  &  1 & \\
 &  &  & 1\\
\end{pmatrix}, u,  w \in \mbb{G}_a \right\},\\
R &=& DU.
\end{eqnarray*}

\begin{defn} \label{bessel} Let $\nu_1, \nu_2: \mbb{Q}^\times \backslash \mbb{A}^\times \lra \mbb{C}^\times$ be continuous characters. Let $\alpha_\nu$ be the character of $R(\mbb{A})$ defined by
$$
\alpha_\nu(du) = \nu_1(d_1) \nu_2(d_2) \Lambda(u).
$$
A cuspidal automorphic representation $\pi$ of $G$ has a split Bessel model associated to $(\nu_1, \nu_2)$ if its central character $\omega_\pi$ coincides with $\nu_1 \nu_2$ and if there exists $\Psi \in \pi$ such that the function on $G(\mbb{A})$ defined by
$$
g \longmapsto W_\Psi(g)=\int_{(Z(\mbb{A})R(\mbb{Q}))\backslash R(\mbb{A})} \Psi(rg) \alpha_\nu(r)^{-1} dr
$$
is not identically zero.
\end{defn}

\begin{rems} Obviously, $\pi$ has a split Bessel model associated to $(\nu_1, \nu_2)$ if and only if the function $g \longmapsto W_\Psi(g)$ is non-zero for any non-zero $\Psi \in \pi$.
\end{rems}

The connection between the integral we are interested in and split Bessel models is given by the following result.

\begin{lem} Let $\pi$ be a cuspidal automorphic representation of $G$ and let $\Psi \in \pi$. Let
$$
Z(\Psi, \Phi, \mu, \nu_1, \nu_2, s)=\int_{(Z(\mbb{A}) G'(\mbb{Q})) \backslash G'(\mbb{A})} \Psi(g)E^\Phi(g, \mu, \nu_1, \nu_2, s) dg.
$$
Then
$$
Z(\Psi, \Phi, \mu, \nu_1, \nu_2, s)=\int_{(D(\mbb{A}) N(\mbb{A}) \backslash G'(\mbb{A}))} W_\Psi(g)  f^\Phi(g, \mu, \nu_1, \nu_2, s)dg
$$
for $Re\,s$ big enough.
\end{lem}

\begin{proof}  Note that we are in the setting of \cite{piatetski-shapiro} 2, for the choice
$$
\beta=\begin{pmatrix}
 & 1/2\\
1/2  & \\
\end{pmatrix} \in U(\mbb{Q}).
$$
Hence, the statement follows from the proof of \cite{piatetski-shapiro} Thm. 5.2.
\end{proof}

\begin{rems} The lemma shows that if $\pi$ does not have a split Bessel model associated to $(\nu_1, \nu_2)$ then, the integrals $Z(\Psi, \Phi, \mu, \nu_1, \nu_2, s)$ are identically zero.
\end{rems}

The global definition above has a local analog. Roughly speaking, given an irreducible representation $(\pi_v, V_{\pi_v})$ of $G(\mbb{Q}_v)$, a local Bessel model of $\pi_v$ is a $G(\mbb{Q}_v)$ equivariant map from $V_{\pi_v}$ to a space of functions $W: G(\mbb{Q}_v) \lra \mbb{C}$ such that $W(rg)=\alpha(r) W(g)$, for some character $\alpha$ of $R(\mbb{Q}_v)$. For a precise definition, in particular at the archimedean place, we refer the reader to \cite{moriyama} 1.5. Assume that $\pi$ has  split Bessel model associated to $(\nu_1, \nu_2)$. If $\Psi=\bigotimes'_v \Psi_v$ is factorizable, it follows from the unicity of local Bessel models (\cite{piatetski-shapiro} Th. 3.1) that the function $W_\Psi$ factors into a restricted product $W_\Psi=\prod'_v W_{\Psi_v}$ of local Bessel functions.

\begin{cor} \label{completed-euler-product} Assume that $\Phi=\prod'_v \Phi_v$ and $\Psi= \bigotimes'_v \Psi_v$ are factorizable. Then, for $Re\,s$ big enough, we have the Euler product expansion
$$
Z(\Psi, \Phi, \mu, \nu_1, \nu_2, s)=\prod_v Z_v(W_{\Psi_v}, \Phi_v, \mu_v, \nu_{1, v}, \nu_{2, v}, s)
$$
where
$$
Z_v(W_{\Psi_v}, \Phi_v, \mu_v, \nu_{1, v}, \nu_{2, v}, s)=\int_{(D(\mbb{Q}_v) N(\mbb{Q}_v) \backslash G'(\mbb{Q}_v))} W_{\Psi_v}(g_v)  f^{\Phi_v}(g_v, \mu_v, \nu_{1, v}, \nu_{2, v}, s)dg_v
$$
for all place $v$.
\end{cor}

The end of section \ref{calcul-de-l-integrale} is devoted to the computation of some of the local integrals above, under the assumption that $\pi$ has a split Bessel model associated to $(\nu_1, \nu_2)$.

\begin{pro} \label{integrale-nonram} Let $p$ be a non-archimedean place where $\pi$, $\nu_1$ and $\nu_2$ are unramified. Let $\Psi_p$ be the standard unramified vector of $\pi_p$ and let $\Phi_p$ be the indicator function of $\mbb{Z}_p^4$. Normalize $W_{\Psi_p}$ in such a way that $W_{\Psi_p}(1)=1$. Then,
$$
Z_p(W_{\Psi_p}, \Phi_p, 1, \nu_{1, p}, \nu_{2, p}, s)= L(s+1/2, \nu_{1, p}) L(s+1/2, \nu_{2, p}) L(s+2, \pi_p, r)
$$
for $Re\,s$ big enough.
\end{pro}

\begin{proof} The representation $\pi_p$ is an unramified principal series representation $\mrm{ind}_{B(\mbb{Q}_p)}^{G(\mbb{Q}_p)} \chi_p$ (\cite{casselman} Prop. 2.6), where $\chi_p$ is an unramified character of $T(\mbb{Q}_p)$ and where $\mrm{ind}$ denotes normalized induction. As explained in \cite{asgari-schmidt} 3.2, the Satake parameters are
\begin{eqnarray*}
b_0 &=& \chi_p \left( \begin{pmatrix}
1 &  &  & \\
 & 1 &  & \\
 &  &  p & \\
 &  &  & p\\
\end{pmatrix} \right),\\
b_1 &=& \chi_p \left( \begin{pmatrix}
p &  &  & \\
 & 1 &  & \\
 &  &  p^{-1} & \\
 &  &  & 1\\
\end{pmatrix} \right),\\
b_2 &=& \chi_p \left( \begin{pmatrix}
1 &  &  & \\
 & p &  & \\
 &  &  1 & \\
 &  &  & p^{-1}\\
\end{pmatrix} \right)\\
\end{eqnarray*}
and the Langlands Euler factor is
\begin{eqnarray*}
L(s, \pi_p, r) &=& \frac{1}{(1-\alpha_1 p^{-s})(1-\alpha_2 p^{-s})(1-\alpha_3 p^{-s})(1-\alpha_4 p^{-s})}
\end{eqnarray*}
where we introduced the convenient notation $\alpha_1=b_0 b_1 b_2$, $\alpha_2=b_0 b_1$, $\alpha_3=b_0$, $\alpha_4=b_0 b_2$. It follows from the first statement of Prop. \ref{PS1} that the function $g \longmapsto f^{\Phi_p}(g, 1, \nu_{1, p}, \nu_{2, p}, s)$ belongs to $\mrm{ind}_{B'(\mbb{Q}_p)}^{G'(\mbb{Q}_p)} \chi_{\nu_{1, p}, \nu_{2, p}, s}$. Furthermore, it is easy to see that with our choice of $\Phi_p$, $f^{\Phi_p}$ is constant on $G'(\mbb{Z}_p)$. As a consequence, according to the Iwasawa decomposition $G'(\mbb{Q}_p)=B'(\mbb{Q}_p) G'(\mbb{Z}_p)$, the local integral  $Z_p(W_{\Psi_p}, \Phi_p, 1, \nu_{1, p}, \nu_{2, p}, s)$ equals
\begin{eqnarray*}
&& f^{\Phi_p}(1, 1, \nu_{1, p}, \nu_{2, p}, s) \int_{\mbb{Q}_p^\times} |x|^{s+1/2} W_{\Psi_p} \left( \begin{pmatrix}
x &  &  & \\
 & x &  & \\
 &  &  1 & \\
 &  &  & 1\\
\end{pmatrix} \right) d^\times x\\
&=& L(s+1/2, \nu_{1, p}) L(s+1/2, \nu_{2, p}) \int_{\mbb{Q}_p^\times} |x|^{s+1/2} W_{\Psi_p} \left( \begin{pmatrix}
x &  &  & \\
 & x &  & \\
 &  &  1 & \\
 &  &  & 1\\
\end{pmatrix} \right) d^\times x.
\end{eqnarray*}
As the Bessel function $W_{\Psi_p}$ satisfies $W_{\Psi_p}(ugk)=\Lambda_p(u) W_{\psi_p}(g)$ for any $u \in U(\mbb{Q}_p)$, any $g \in G(\mbb{Q}_p)$ and $k \in G(\mbb{Z}_p)$, we have
\begin{eqnarray*}
W_{\Psi_p} \left( \begin{pmatrix}
p^{-m} &  &  & \\
 & p^{-m} &  & \\
 &  &  1 & \\
 &  &  & 1\\
\end{pmatrix} \right) &=& W_{\Psi_p} \left( \begin{pmatrix}
p^{-m} &  &  & \\
 & p^{-m} &  & \\
 &  &  1 & \\
 &  &  & 1\\
\end{pmatrix} \begin{pmatrix}
1 &  &  & 1 \\
 & 1 & 1  & \\
 &  &  1 & \\
 &  &  & 1\\
\end{pmatrix}\right)\\
&=& \eta(p^{-m}) W_{\Psi_p} \left( \begin{pmatrix}
p^{-m} &  &  & \\
 & p^{-m} &  & \\
 &  &  1 & \\
 &  &  & 1\\
\end{pmatrix} \right).
\end{eqnarray*}
For any integer $m>0,$ we have $\eta(p^{-m}) \neq 1$. Hence,
$$
W_{\Psi_p} \left( \begin{pmatrix}
p^{-m} &  &  & \\
 & p^{-m} &  & \\
 &  &  1 & \\
 &  &  & 1\\
\end{pmatrix} \right)=0.
$$
As a consequence, $$Z_p(\Psi_p, \Phi_p, 1, \nu_{1, p}, \nu_{2, p}, s)$$
$$
=L(s+1/2, \nu_{1, p}) L(s+1/2, \nu_{2, p}) \sum_{m=0}^{+\infty} p^{-m(s+1/2)} W_{\Psi_p} \left( \begin{pmatrix}
p^{m} &  &  & \\
 & p^{m} &  & \\
 &  &  1 & \\
 &  &  & 1\\
\end{pmatrix} \right).
$$
Let $\beta_1 = \nu_{1, p}(p)$ and $\beta_2 = \nu_{2, p}(p)$. Because we assume that $\pi$ has a split Bessel model associated to $(\nu_1, \nu_2)$, $\nu_1 \nu_2$ coincides with the central character of $\pi$, hence we have $\alpha_1 \alpha_3=\alpha_2 \alpha_4=\beta_1 \beta_2$. The Weyl group $W$ acts on the $\alpha_i$ through all permutations which preserve the relation $\alpha_1 \alpha_3= \alpha_2 \alpha_4$. More precisely, with the notations of section 2.3, we have $s_1 \alpha_1=\alpha_2$, $s_1 \alpha_2=\alpha_1$, $s_1 \alpha_3=\alpha_4$, $s_1 \alpha_4=\alpha_3$ and $s_2 \alpha_1=\alpha_1$, $s_2 \alpha_2=\alpha_4$, $s_2 \alpha_3=\alpha_3$, $s_2 \alpha_4=\alpha_2$. Let $\mathcal{A}$ denote $\sum_{w \in W} (-1)^{l(w)} w$, seen as an element of the group algebra $\mbb{C}[W]$, where $l(w)$ denotes the length of $w$. According to \cite{bump-friedberg-furusawa} Thm. 1.6 and Cor. 1.9  (2), the following explicit formula
$$
W_{\Psi_p} \left( \begin{pmatrix}
p^{m} &  &  & \\
 & p^{m} &  & \\
 &  &  1 & \\
 &  &  & 1\\
\end{pmatrix} \right)W_{\Psi_p}(1)^{-1} =p^{-3m/2} \frac{\mathcal{A}(\alpha_3^{m+2}\alpha_4^{-1})}{\mathcal{A}(\alpha_3^2 \alpha_4^{-1})}
$$
is true for any integer $m \geq 0$. As a consequence, using $W_{\Psi_p}(1)=1$, we have
$$
Z_p(W_{\Psi_p}, \Phi_p, 1, \nu_{1, p}, \nu_{2, p}, s)=L(s+1/2, \nu_{1, p}) L(s+1/2, \nu_{2, p})
$$
$$
\times \mathcal{A}(\alpha_3^2 \alpha_4)^{-1}
\mathcal{A}\left(\alpha_3^2 \alpha_4^{-1}\left(1-\alpha_3 p^{-(s+2)}\right)^{-1}\right)
$$
which can be rewritten as $$L(s+1/2, \nu_{1, p}) L(s+1/2, \nu_{2, p}) L(s+2, \pi_p, r)$$
$$
\times \mathcal{A}(\alpha_3^2 \alpha_4)^{-1}\mathcal{A}\left(\alpha_3^2 \alpha_4^{-1}
(1-\alpha_1 p^{-(s+2)})(1-\alpha_2 p^{-(s+2)})(1-\alpha_4 p^{-(s+2)})\right)
$$
because the local $L$-factor $L(s+2, \pi_p, r)$ is invariant under $W$. We claim that
$$
\mathcal{A}\left(\alpha_3^2 \alpha_4^{-1}
(1-\alpha_1 p^{-(s+2)})(1-\alpha_2 p^{-(s+2)})(1-\alpha_4 p^{-(s+2)})\right)=\mathcal{A}\left(\alpha_3^2 \alpha_4^{-1} \right).
$$
Indeed, using the relation, $\alpha_1 \alpha_3=\alpha_2 \alpha_4$, we find
\begin{eqnarray*}
&& \alpha_3^2 \alpha_4^{-1}(1-\alpha_1 p^{-(s+2)})(1-\alpha_2 p^{-(s+2)})(1-\alpha_4 p^{-(s+2)})\\
&=& \alpha_3^2 \alpha_4^{-1}-(\alpha_2 \alpha_3+\alpha_2 \alpha_3^2 \alpha_4^{-1}+\alpha_3^2)p^{-(s+2)}+(\alpha_1\alpha_3^2+\alpha_2 \alpha_3^2+\alpha_2^2 \alpha_3)p^{-2(s+2)}-\alpha_1\alpha_2\alpha_3^2 p^{-3(s+2)}.
\end{eqnarray*}
But, as $\alpha_2 \alpha_3$ is fixed by $s_2 s_1 s_2 \in W$, which has odd length, we have $\mathcal{A}(\alpha_2 \alpha_3 p^{-(s+2)})=0.$ Similarly, note that $\alpha_2 \alpha_3^2 \alpha_4^{-1}$ is fixed by $s_2 s_1 s_2$ because of the relation $\alpha_1 \alpha_3=\alpha_2 \alpha_4$. So $\mathcal{A}(\alpha_2 \alpha_3^2 \alpha_4^{-1}p^{-(s+2)})=0$. Similarly we find that 
$$
\mathcal{A}(\alpha_3^2 p^{-(s+2)})=\mathcal{A}(\alpha_3^2 \alpha_1 p^{-2(s+2)})=\mathcal{A}(\alpha_1 \alpha_2 \alpha_3^2 p^{-3(s+2)})=0.
$$ As $(s_2 s_1 s_2)(\alpha_3^2 \alpha_2)=\alpha_2^2 \alpha_3$, we have $\mathcal{A}((\alpha_2 \alpha_3^2+\alpha_2^2 \alpha_3)p^{-2(s+2)})=0$. This proves our claim. As a consequence, we have the equality
$$
Z_p(W_{\Psi_p}, \Phi_p, 1, \nu_{1, p}, \nu_{2, p}, s)= L(s+1/2, \nu_{1, p}) L(s+1/2, \nu_{2, p}) L(s+2, \pi_p, r).
$$
\end{proof}

\begin{rems} A similar result is stated without proof in \cite{piatetski-shapiro} Thm. 4.4.
\end{rems}
Let us consider the ramified non-archimedean integrals.

\begin{pro} \label{integrales-ram} \cite{occult} Lem. 3.5.4. Let $p$ be a non-archimedean place. Then, if the Schwartz function $\Phi_p$ and the Bessel function $W_{\Psi_p}$ are $\overline{\mathbb{Q}}$-valued, we have $$Z_p(W_{\Psi_p}, \Phi_p, 1, \nu_{1, p}, \nu_{2, p}, p+q+2) \in \overline{\mbb{Q}}.$$ Furthermore, if $\Psi_p$ has a split Bessel model associated to $(\nu_{1, p}, \nu_{2, p})$, then there exists a $\overline{\mathbb{Q}}$-valued function $\Phi_p$ such that $$Z_p(W_{\Psi_p}, \Phi_p, 1, \nu_{1, p}, \nu_{2, p}, p+q+2) \in \overline{\mbb{Q}}^\times.$$
\end{pro}

The archimedean computation below is a direct application of \cite{moriyama} Thm. 7.1 and of the Mellin inversion formula.

\begin{pro} \label{integrale-archi} Let $\pi_\infty$ be a discrete series representation of $G(\mbb{R})^+$ with minimal $K$-type $\tau_{(\lambda_1, \lambda_2)}$. Let $\nu_{1, \infty}, \nu_{2, \infty}: \mbb{R}^\times \longrightarrow \mbb{C}^\times$ be the characters defined by
\begin{eqnarray*}
\nu_{1, \infty}(x)=|x|^q sgn(x)^p,\\
\nu_{2, \infty}(x)=|x|^p sgn(x)^q.
\end{eqnarray*}
Let $(v_t)_{0 \leq t \leq \lambda_1-\lambda_2}$ be a standard basis of $\tau_{(\lambda_1, \lambda_2)}$ and for any integer $0 \leq t \leq \lambda_1-\lambda_2$, let $W_{\infty, t}: G(\mbb{R})^+ \longrightarrow \mbb{C}$ denote a Bessel function corresponding to $v_t$ in the sense of \cite{moriyama} 3.1. Let $r \equiv p \pmod 2$ and $s \equiv q \pmod 2$ be two integers. Assume that $\Phi_\infty$ is of the shape
$$
(x_1, y_1, x_2, y_2) \longmapsto \Phi_\infty(x_1, y_1, x_2, y_2)= \Phi_{\infty, 1}(x_1, y_1) \Phi_{\infty, 2}(x_2, y_2)
$$
where
\begin{eqnarray*}
\Phi_{\infty, 1}(x_1, y_1)=(ix_1+y_1)^\frac{p-r}{2}(ix_1-y_1)^\frac{p+r}{2} e^{-\pi(x_1^2+y_1^2)},\\
\Phi_{\infty, 2}(x_2, y_2)=(ix_2+y_2)^\frac{q-s}{2}(ix_2-y_2)^\frac{q+s}{2} e^{-\pi(x_2^2+y_2^2)}.
\end{eqnarray*}
Then, the following statements are satisfied.
\begin{itemize}
\item If $t+\lambda_2+r \neq 0$ or $-t+\lambda_1+s \neq 0$, then,
$$
Z_\infty(W_{\infty, t}, \Phi_\infty, 1, \nu_{1, \infty}, \nu_{2, \infty}, p+q+3/2)=0.
$$
\item Else, let
\begin{eqnarray*}
a_1 &=& \frac{t-\lambda_2-(q-p)/2+2}{2},\\
a_2 &=& \frac{2\lambda_1+\lambda_2-t+(q-p)/2+2}{2},\\
c_1 &=& \frac{\lambda_1+\lambda_2+4}{4},\\
c_2 &=& \frac{\lambda_1-\lambda_2+4}{4},\\
c_3 &=& \frac{\lambda_1+\lambda_2+2}{4},\\
c_4 &=& \frac{\lambda_1-\lambda_2+2}{4}.
\end{eqnarray*}
Then,
$$
Z_\infty(W_{\infty, t}, \Phi_\infty, 1, \nu_{1, \infty}, \nu_{2, \infty}, p+q+3/2)=
$$
$$
W_{\infty, t}(1) \left( 2 \pi^{\frac{3(p+q)+6}{2}} \int_L \frac{\Gamma \left( c_1-s \right) \Gamma(c_2-s)\Gamma(c_3-s)\Gamma(c_4-s)}{\Gamma(a_1-s)\Gamma(a_2-s)} \pi^{2s} \frac{ds}{2 \pi i} \right)^{-1}
$$
$$
\times \frac{ \Gamma \left( c_1+\frac{3(p+q)+6}{4} \right)\Gamma \left( c_2+\frac{3(p+q)+6}{4} \right) \Gamma \left( c_3+\frac{3(p+q)+6}{4} \right) \Gamma \left( c_4+\frac{3(p+q)+6}{4} \right)}{ \Gamma \left( a_1+\frac{3(p+q)+6}{4} \right) \Gamma \left( a_2+\frac{3(p+q)+6}{4} \right)}
$$
where $\Gamma$ denotes the gamma function and where the path $L$ is a loop starting and ending at $+\infty$ and encircling all the poles of $\Gamma(c_j-s)$ for $1 \leq j \leq 4$ once, in the negative direction.
\end{itemize}
\end{pro}

\begin{proof}  The vector $v_t$ has weight $\lambda'(t+\lambda_2, -t+\lambda_1)$ and the function $\Phi_\infty$ has weight $\lambda'(r, s)$ for $T'=\mrm{U}(1)^2$. Hence $W_{\infty, t}$ has weight $\lambda'(t+\lambda_2, -t+\lambda_1)$ and $f^{\Phi_\infty}$ has weight $\lambda'(r, s)$. The Iwasawa decomposition $G'(\mbb{R})=B'(\mbb{R}) \mrm{U}(1)^2$ implies that $Z_\infty(W_{\infty, t}, \Phi_\infty, 1, \nu_{1, \infty}, \nu_{2, \infty}, s)$ equals
\begin{eqnarray*}
\int_{\mbb{R}^\times_+} \int_{\mrm{U}(1)^2} W_{\infty, t} \left( \begin{pmatrix}
x &  &  & \\
 & x &  & \\
 &  &  1 & \\
 &  &  & 1\\
\end{pmatrix} k \right)  f^{\Phi_\infty}\left(\begin{pmatrix}
x &  &  & \\
 & x &  & \\
 &  &  1 & \\
 &  &  & 1\\
\end{pmatrix} k, 1, \nu_{1, \infty}, \nu_{2, \infty}, s \right)d^\times x dk.
\end{eqnarray*}
So, if $t+\lambda_2+r \neq 0$ or $-t+\lambda_1+s \neq 0$, we have
$$
\int_{\mrm{U}(1)^2} W_{\infty, t} \left( \begin{pmatrix}
x &  &  & \\
 & x &  & \\
 &  &  1 & \\
 &  &  & 1\\
\end{pmatrix} k \right)  f^{\Phi_\infty}\left(\begin{pmatrix}
x &  &  & \\
 & x &  & \\
 &  &  1 & \\
 &  &  & 1\\
\end{pmatrix} k, 1, \nu_{1, \infty}, \nu_{2, \infty}, s \right)dk=0
$$
for all $x$ because the integral of a non-trivial character on a group is zero. In the case where $t+\lambda_2+r=-t+\lambda_1+s =0$, the integral equals
\begin{eqnarray*}
\int_{\mbb{R}^\times_+} W_{\infty, t} \left( \begin{pmatrix}
x &  &  & \\
 & x &  & \\
 &  &  1 & \\
 &  &  & 1\\
\end{pmatrix} \right)  f^{\Phi_\infty}\left(\begin{pmatrix}
x &  &  & \\
 & x &  & \\
 &  &  1 & \\
 &  &  & 1\\
\end{pmatrix}, 1, \nu_{1, \infty}, \nu_{2, \infty}, s \right) d^\times x.
\end{eqnarray*}
The Meijer $G$-function $G^{4, 0}_{2, 4}(z, a_1, a_2, c_1, c_2, c_3, c_4)$ is defined for any non-zero complex number $z$ by
$$
G^{4, 0}_{2, 4}(z, a_1, a_2, c_1, c_2, c_3, c_4)=\int_L \frac{\Gamma(c_1-s)\Gamma(c_2-s)\Gamma(c_3-s)\Gamma(c_4-s)}{\Gamma(a_1-s)\Gamma(a_2-s)} z^{s} \frac{ds}{2 \pi i}.
$$
It follows from \cite{moriyama} Thm. 7.1. (ii) that for any positive real number $x$, we have
$$
W_{\infty, t} \left( \begin{pmatrix}
x &  &  & \\
 & x &  & \\
 &  &  1 & \\
 &  &  & 1\\
\end{pmatrix} \right)=
W_{\infty, t}(1) x^\frac{p+q}{2} \frac{G^{4, 0}_{2, 4}((\pi x)^2, a_1, a_2, c_1, c_2, c_3, c_4)}{G_{2, 4}^{4, 0}(\pi^2, a_1, a_2, c_1, c_2, c_3, c_4)}.
$$
It is well known and explained in \cite{beals-s} p. 870, for example, that the contour $L$ can be replaced by a path $L'$ from $-i\infty$ to $+i\infty$ such that for $1 \leq j \leq 4$, the poles of $\Gamma(c_j-s)$ are on the right of $L'$. Furthermore, we have
$$
 f^{\Phi_\infty}\left(\begin{pmatrix}
x &  &  & \\
 & x &  & \\
 &  &  1 & \\
 &  &  & 1\\
\end{pmatrix}, 1, \nu_{1, \infty}, \nu_{2, \infty}, s \right)=(-1)^\frac{p+q+r+s}{2} \pi^{-2(p+q)} ((p+q-1)!)^2 x^{p+q+2}
$$
for any $x \in \mbb{R}^\times_+$. Hence, we need to compute
$$
\int_{\mbb{R}_+^\times} x^{\frac{3(p+q)}{2}+2} \int_{L'} \frac{\Gamma(c_1-s)\Gamma(c_2-s)\Gamma(c_3-s)\Gamma(c_4-s)}{\Gamma(a_1-s)\Gamma(a_2-s)} (\pi x)^{2s} \frac{ds}{2 \pi i} \frac{dx}{x}
$$
which equals
$$
\frac{1}{2}\pi^{-\frac{3(p+q)+6}{2}} \frac{ \Gamma \left( \frac{\lambda_1+\lambda_2+3(p+q)+10}{4} \right)\Gamma \left( \frac{\lambda_1-\lambda_2+3(p+q)+10}{4} \right) \Gamma \left( \frac{\lambda_1+\lambda_2+3(p+q)+8}{4} \right) \Gamma \left( \frac{\lambda_1-\lambda_2+3(p+q)+8}{4} \right)}{ \Gamma \left( \frac{2(t-\lambda_2)+4p+2q+10}{4} \right) \Gamma \left( \frac{2(2\lambda_1+\lambda_2-t)+4q+2p+10}{4} \right)}
$$
by the Mellin inversion formula. This completes the proof.
\end{proof}

\section{Periods} \label{section-periodes}

In this section, we compute the pairing $\langle \omega, \widetilde{v}_\mathcal{D} \rangle_B$ (see Lem. \ref{dualite}) and we introduce Harris' occult period invariant, which plays a crucial role in the present work.\\

The $E(\pi_f)$-module $M_B(\pi_f, W)^-(-1)$ has rank two and the $E(\pi_f)$-module $F^0 M_{dR}(\pi_f, W)$ has rank one (Lem. \ref{dimensions} and Hyp. \ref{hypothesis}). Let us fix a basis $(v_1, v_2)$ of $M_B(\pi_f, W)^-(-1)$ and let $\theta$ be a non-zero vector of $F^0 M_{dR}(\pi_f, W)$. This vector can be regarded as a vector of $M_B(\pi_f, W)_\mbb{R}^-(-1)$ via the second arrow of the exact sequence of Lem. \ref{suite_ex_ext1}.  Let $\lambda_1, \lambda_2 \in E(\pi_f) \otimes \mbb{R}$ be its coordinates in the basis $(v_1, v_2)$.

\begin{lem} \label{charac-deligne} Let $\mu_1, \mu_2 \in E(\pi_f) \otimes \mbb{R}$. Then $\mu_1 v_1 +\mu_2 v_2$ is mapped to a generator of $\mathcal{B}(\pi_f, W)$ by the surjection
$$
\begin{CD}
M_B(\pi_f, W)_\mbb{R}^-(-1) @>>> \mathrm{Ext}^1_{\mathrm{MHS}_\mathbb{R}^+}(\mathbb{R}(0),  M_{B}(\pi_f, W)_\mathbb{R})
\end{CD}
$$
of Cor. \ref{suite_ex_ext1} if and only if $\lambda_1 \mu_2 - \lambda_2 \mu_1 \in E(\pi_f)^\times$.
\end{lem}

\begin{proof} The map
$$
\theta \wedge : M_B(\pi_f, W)_\mbb{R}^-(-1) \lra {det}_{E(\pi_f) \otimes \mbb{R}} M_B(\pi_f, W)_\mbb{R}^-(-1)
$$
is part of the following commutative diagram with exact lines
$$
\begin{CD}
0 @>>> F^0 M_{dR, \mbb{R}} @>>> M_{B, \mbb{R}}^-(-1) @>\theta \wedge>> {det}\,M_{B, \mbb{R}}^-(-1) @>>> 0\\
@.              @|                                       @|                                                            @VVV         \\
0 @>>> F^0 M_{dR, \mbb{R}} @>>> M_{B, \mbb{R}}^-(-1) @>>> \mrm{Ext}^1_{\mrm{MHS}_\mbb{R}^+}(\mbb{R}(0), M) @>>> 0.
\end{CD}
$$
Moreover, via the isomorphism $\mathcal{B}(\pi_f, W) \simeq {det}_{E(\pi_f)} M_B(\pi_f, W)^-(-1)$ induced by the choice of $\theta$, the vector $\mu_1 v_1 +\mu_2 v_2$ is mapped to a generator of $\mathcal{B}(\pi_f, W)$ if and only if $\theta \wedge (\mu_1 v_1 +\mu_2 v_2)= \rho v_1 \wedge v_2$ for some $\rho \in E(\pi_f)^\times$. As $\theta \wedge (\mu_1 v_1 +\mu_2 v_2)=(\lambda_1 \mu_2 - \lambda_2 \mu_1) v_1 \wedge v_2$, the statement is proven.
\end{proof}

We need to recall the definition of the Deligne periods $c^\pm(\pi_f, W)$ of the "motive" $M(\pi_f,  W)$  from \cite{valeurs-deligne} 1.7. Let $t$ denote the integer $t=\frac{p+q+6-k-k'}{2}$. We have the Hodge decomposition
$$
M_B(\pi_f, W)_\mbb{C}=M_B^{3-t, -k-k'-t} \oplus M_B^{2-k'-t, 1-k-t} \oplus M_B^{1-k-t, 2-k'-t} \oplus M_B^{-k-k'-t, 3-t}
$$
where each $M_B^{r, s}$ is an $E(\pi_f) \otimes \mbb{C}$-module of rank one (Prop. \ref{dimension}, Prop. \ref{dec-hodge} and Hyp. \ref{stable}). Furthermore, the involution $F_\infty$ exchanges $M_B^{r, s}$ and $M_B^{s, r}$. This implies that the $E(\pi_f)$-subspaces $M_B(\pi_f, W)^\pm$ of $M_B(\pi_f, W)$ where $F_\infty$ acts by multiplication by $\pm1$ both have dimension two. Let $I_\infty: M_B(\pi_f, W)_\mbb{C} \longrightarrow M_{dR}(\pi_f, W)_\mbb{C}$ denote the comparison isomorphism.  The subspaces $F^\pm(\pi_f, W)$ of the de Rham filtration of $M_{dR}(\pi_f, W)$ defined in \cite{valeurs-deligne} 1.7 are equal and characterized by the fact that their complexification is mapped isomorphically to $M_B^{3-t, -k-k'-t} \oplus M_B^{2-k'-t, 1-k-t}$ by $I_\infty^{-1}$. The determinant of the isomorphism
$$
I_\infty^\pm: M_B(\pi_f, W)^\pm_\mbb{C} \lra \left( M_{dR}(\pi_f, W)/F^\pm(\pi_f, W) \right)_\mbb{C},
$$
computed in basis defined over $E(\pi_f)$, is by definition the Deligne period $c^\pm(\pi_f, W)$. Its equivalence class modulo the relation $\sim$ of Def \ref{equivalence} does not depend on the chosen basis. The Deligne periods $c^\pm(\check{\pi}_f|\nu|^{-3}, W(-p-q-3))$ are defined similarly.

\begin{pro} \label{calcul-periode} Let $\omega \in M_B(\check{\pi}_f|\nu|^{-3}, W(-p-q-3))_\mbb{C}$
satisfying the condition of Lem. \ref{dualite}. Assume that $\omega$ belongs to $M_B(\check{\pi}_f|\nu|^{-3}, W(-p-q-3))_\mbb{C}^+$ and that the image of $\omega$ by the comparison isomorphism
$$
I_\infty^+: M_B(\check{\pi}_f|\nu|^{-3}, W(-p-q-3))_\mbb{C}^+ \lra (M_{dR}(\check{\pi}_f|\nu|^{-3}, W(-p-q-3))/F^+)_\mbb{C}
$$
belongs to the $E(\pi_f)$-structure $M_{dR}(\check{\pi}_f|\nu|^{-3}, W(-p-q-3))/F^+$. Let $\widetilde{v}_\mathcal{D}$ be a lift of a generator of $\mathcal{D}(\pi_f, W)$ by the surjection
$$
M_{B}(\pi_f, W)_\mathbb{R}^-(-1) \longrightarrow \mathrm{Ext}^1_{\mathrm{MHS}_\mathbb{R}^+}(\mathbb{R}(0),  M_{B}(\pi_f, W)_\mathbb{R})
$$
of Cor. \ref{suite_ex_ext1}. Then,
$$
\langle \omega, \widetilde{v}_\mathcal{D} \rangle_B \sim (2\pi i)^2 c^-(\pi_f, W)^{-1}.
$$
\end{pro}

\begin{proof} Recall that the Deligne and Beilinson $E(\pi_f)$-structures are related by the identity
$$
\mathcal{D}(\pi_f, W)=(2\pi i)^{2} \delta(\pi_f, W)^{-1} \mathcal{B}(\pi_f, W).
$$
Let $\mu_1, \mu_2 \in E(\pi_f) \otimes \mbb{R}$ such that $\lambda_1 \mu_2 -\lambda_2 \mu_1=1$. Then, according to Lem. \ref{charac-deligne}, the vector $v=\mu_1 v_1 + \mu_2 v_2$ is a lift of a generator of $\mathcal{B}(\pi_f, W)$ by the surjection
$$
M_B(\pi_f, W)_\mbb{R}^-(-1) \longrightarrow \mathrm{Ext}^1_{\mathrm{MHS}_\mathbb{R}^+}(\mathbb{R}(0),  M_{B}(\pi_f, W)_\mathbb{R}).
$$
For any other lift $w$ of any other generator of $\mathcal{B}(\pi_f, W)$, we have $\langle \omega, v \rangle_B \sim \langle \omega, w \rangle_B$. Hence, we need to compute $\langle \omega, \widetilde{v}_\mathcal{D} \rangle_B=(2\pi i)^{2} \delta(\pi_f, W)^{-1}  \langle \omega, v \rangle_B$. Let $\langle \,\,,\,\, \rangle_{dR}$ denote the Poincar\'e duality pairing in de Rham cohomology. Its complexification is part of the following commutative diagram
$$
\begin{CD}
M_B(\pi_f, W)_\mbb{C} \otimes M_B(\check{\pi}_f|\nu|^{-3}, W(-p-q-3))_\mbb{C} @>\langle \,\,,\,\, \rangle_B>> E(\pi_f)(0)_{B, \mbb{C}}\\
@VV{I_\infty}V                                                                                                                                                                                      @VV{J_\infty}V\\
M_{dR}(\pi_f, W)_\mbb{C} \otimes M_{dR}(\check{\pi}_f|\nu|^{-3}, W(-p-q-3))_\mbb{C} @>\langle \,\,,\,\, \rangle_{dR}>> E(\pi_f)(0)_{dR, \mbb{C}}\\
\\
\end{CD}
$$
where the vertical maps are the comparison isomorphisms. Let $1_B$ and $1_{dR}$ be generators of $E(\pi_f)(0)_{B}$ and $E(\pi_f)(0)_{dR}$ respectively. We can assume that $J_\infty(1_B)=1_{dR}$. Let $\theta'$ be an element of $M_{dR}(\check{\pi}_f|\nu|^{-3}, W(-p-q-3))$ such that $\langle \theta, \theta' \rangle=1_{dR}$, where $\theta \in F^0 M_{dR}(\pi_f, W)$ is as above. Let $\omega'$ denote $I_\infty^+(\omega)$. It follows easily from the consideration of Hodge types that $(\omega', \theta')$ is a basis of $M_{dR}(\check{\pi}_f|\nu|^{-3}, W(-p-q-3))/F^\pm$. By definiton $(v_1, v_2)$ is a basis of the $E(\pi_f)$-vector space $M_B(\pi_f, W)^-(-1)$. As a consequence, $(2\pi i v_1, 2 \pi i v_2)$ is a basis of $M_B(\pi_f, W)^+$. Let $w_1, w_2$  be vectors of $M_B(\check{\pi}_f|\nu|^{-3}, W(-p-q-3))^+$ such that $\langle 2 \pi i v_1, w_1 \rangle_B=\langle 2 \pi i v_2, w_2 \rangle_B=1_B$ and $\langle 2 \pi i v_1, w_2 \rangle_B=\langle 2 \pi i v_2, w_1 \rangle_B=0$. Hence $(w_1, w_2)$ is a basis of $M_B(\check{\pi}_f|\nu|^{-3}, W(-p-q-3))^+$. Let $\alpha_1, \alpha_2, \beta_1, \beta_2 \in E(\pi_f) \otimes \mbb{C}$ be such that $I_\infty^+(w_1)=\alpha_1 \omega'+ \beta_1 \theta'$ and $I_\infty^+(w_2)=\alpha_2 \omega'+ \beta_2 \theta'$. By definition, we have $$c^+(\check{\pi}_f|\nu|^{-3}, W(-p-q-3))=\alpha_1 \beta_2 - \alpha_2 \beta_1$$ and this implies the identity
$$
\omega=c^+(\check{\pi}_f|\nu|^{-3}, W(-p-q-3))^{-1} \left(\beta_2 w_1 - \beta_1 w_2 \right).
$$
Hence,
\begin{eqnarray*}
\langle \omega, v \rangle_B &=&  c^+(\check{\pi}_f|\nu|^{-3}, W(-p-q-3))^{-1} \langle \beta_2 w_1 -\beta_1 w_2, \mu_1 v_1 + \mu_2 v_2 \rangle\\
&=& c^+(\check{\pi}_f|\nu|^{-3}, W(-p-q-3))^{-1} (\mu_1 \beta_2 - \mu_2 \beta_1).
\end{eqnarray*}
To finish the proof, note that the pairing $\langle \theta, \,\,\rangle_{dR}$ vanishes on $F^\pm$ because the Hodge types do not match. So $\langle \theta, \omega' \rangle_{dR}$ is meaningful, and in fact equal to zero, again because of the Hodge types. As a consequence, we have
$$
\beta_1 1_{dR}=\langle \theta, I_\infty^+(w_1) \rangle_{dR}=J_\infty \left( \langle \lambda_1 v_1 + \lambda_2 v_2, w_1 \rangle_B \right)=\lambda_1 1_{dR}.
$$
Hence $\beta_1=\lambda_1$. Similarly $\beta_2=\lambda_2$. Hence $\langle \omega, v \rangle_B = -c^+(\check{\pi}_f|\nu|^{-3}, W(-p-q-3))^{-1}$. It follows from \cite{valeurs-deligne} (5.1.7) and (5.1.1) that
$
c^+(\check{\pi}_f|\nu|^{-3}, W(-p-q-3)) \sim \delta(\pi_f, W)^{-1} c^-(\pi_f, W).
$
As a consequence,
$
\langle \omega, v \rangle_B \sim \delta(\pi_f, W) c^-(\pi_f, W)^{-1}.
$
This completes the proof.
\end{proof}

In the paper \cite{occult}, Harris defines the occult period invariant and relates it to critical values of the spinor $L$-function. Roughly speaking, the occult period invariant measures the difference between the rational structure on $\pi_f$ defined in terms of de Rham cohomology with the one defined in terms of the Bessel model. To give a precise definition, let us fix $\Psi=\bigotimes_v \Psi_v \in \pi'=\check{\pi} |\nu|^{-3}$ a factorizable vector  whose archimedean component $\Psi_\infty$ is a lowest weight vector of the minimal $K$-type of the discrete series $\pi_\infty^W \in P(W(-p-q-3))$ (see Def. \ref{p(e)}). This defines a cuspidal differential form $\Omega(\Psi_\infty)$ belonging to $\mathrm{Hom}_{K_G}\left( \bigwedge^2 \mathfrak{p}^+ \otimes_\mbb{C} \mathfrak{p}^-,  W(-p-q-3) \otimes_\mbb{C} \pi_\infty^W \right)$ (Lem. \ref{def.omegaw}) and hence an element $(\Omega(\Psi_\infty))_{\sigma: E(\pi_f) \rightarrow \mbb{C}} \in M_B^{2-k'-t', 1-k-t'}$ (Prop. \ref{dimension} and Prop. \ref{dec-hodge}). As $F_\infty$ exchanges $M_B^{2-k'-t', 1-k-t'}$ and $M_B^{1-k-t', 2-k'-t'}$, the class $\omega(\Psi_\infty)=\frac{1}{2} \left((\Omega(\Psi_\infty)_{\sigma}+F_\infty((\Omega(\Psi_\infty))_{\sigma}) \right)$ satisfies the conditions of Lem. \ref{dualite} and belongs to $M_B(\check{\pi}_f|\nu|^{-3}, W(-p-q-3))_\mbb{C}^+$.

\begin{defn} \label{arithmetic} A vector $\Psi_f \in \pi_f'$ is arithmetic if it is the non-archimedean component of a factorizable cusp form $\Psi=\Psi_\infty \otimes \Psi_f \in \pi'$ such that the image of $\omega(\Psi_\infty)$ by the comparison isomorphism
$$
I_\infty^+: M_B(\check{\pi}_f|\nu|^{-3}, W(-p-q-3))_\mbb{C}^+ \lra (M_{dR}(\check{\pi}_f|\nu|^{-3}, W(-p-q-3))/F^+)_\mbb{C}
$$
belongs to the $E(\pi_f)$-structure $M_{dR}(\check{\pi}_f|\nu|^{-3}, W(-p-q-3))/F^+$.
\end{defn}

The following proposition is a reformulation of \cite{occult} Prop. 3.5.2:

\begin{pro} \label{invariant-occulte} Assume that $\pi$ has a split Bessel model associated to $(\nu_1, \nu_2)$. Then, there exists $a(\pi, \nu_1, \nu_2) \in \mbb{C}^\times$ such that the functional $\Psi_f \longmapsto a(\pi, \nu_1, \nu_2) W_{\Psi_f}$ sends the $\Psi_f$ which are arithmetic to functions $W_{\psi_f}$ which are $\overline{\mbb{Q}}$-valued.
\end{pro}

\section{The main result} \label{derniere-section}

\begin{thm} \label{principal} Let $k \geq k' \geq 0$ be two integers. Let $W$ be an irreducible algebraic representation of $G$ of highest weight $\lambda(k, k', k+k'+4)$. Let $\pi=\pi_\infty \otimes \pi_f$ be a cuspidal automorphic representation of $G$ whose central character has infinity type $-k-k'-4$ and whose archimedean component $\pi_\infty$ is a discrete series of Harish-Chandra parameter $(k+2, k'+1)$. Let $\nu_1^0$ be a finite order Hecke character of sign $(-1)^{k-1}$ and let $\nu_1$ denote the Hecke character $|\,\,|^{1-k'}\nu_1^0$. Let $\nu_2^0$ be a finite order Hecke character of sign $(-1)^{k'-1}$ and let $\nu_2$ denote the Hecke character $|\,\,|^{1-k} \nu_2^0$. Let $V$ be the finite set of places where $\pi, \nu_1$ or $\nu_2$ is ramified, together with the archimedean place. Assume that:
\begin{itemize}
\item we have $k>k'>0$,
\item we have $k+1 \equiv k' \equiv 0 \pmod 2$,
\item the automorphic representation $\pi$ is stable at infinity,
\item the automorphic representation $\pi'=\check{\pi}|\nu|^{-3}$ has a split Bessel model associated to $(\nu_1, \nu_2)$.
\end{itemize}
Then,
$$
\mathcal{K}(\pi_f, W)= \pi^{-2}a(\pi, \nu_1, \nu_2) c^-(\pi_f, W)L_V(k+k'-1/2, \check{\pi})\mathcal{D}(\pi_f, W).
$$
\end{thm}

\begin{proof} The assumptions on $k$ and $k'$ imply that if we take $p=k-1$ and $q=k'-1$, then $p, q, k, k'$ satisfy the assumptions of Thm. \ref{lemma1}. Let $\Psi=\bigotimes_v \Psi_v \in \pi'=\check{\pi}|\nu|^{-3}$ be a factorizable cusp form whose archimedean component $\Psi_\infty$ is a lowest weight vector of the minimal $K$-type $\tau_{(k+3, -k'-1)}$ of the element $\pi_\infty^W$ of $P(W(-p-q-3))$. Via a vector $v$ given by Cor. \ref{cor-technique}, and as explained above the statement of Def. \ref{arithmetic}, we associate to $\Psi$ the Betti cohomology class $\omega \in M_B(\check{\pi}_f|\nu|^{-3}, W(-p-q-3))_\mbb{C}^+$
satisfying the Hodge types conditions of Lem. \ref{dualite}. Note that multiplying $\omega$ by a scalar, we can assume that the image of $\omega$ by the comparison isomorphism
$$
I_\infty^+: M_B(\check{\pi}_f|\nu|^{-3}, W(-p-q-3))_\mbb{C}^+ \lra (M_{dR}(\check{\pi}_f|\nu|^{-3}, W(-p-q-3))/F^+)_\mbb{C}
$$
belongs to $M_{dR}(\check{\pi}_f|\nu|^{-3}, W(-p-q-3))/F^+$. Let $Eis_\mathcal{H}^{k-1, k'-1, W}(\phi_f \otimes \phi'_f) \in \mathcal{K}(\pi_f, W)$. Let $\rho$ be a current given by Lem. \ref{relevement}. According to the first statement of Lem. \ref{relevement}, the Betti cohomology class $[\rho]$ of $\rho$ is a lift of $Eis_\mathcal{H}^{k-1, k'-1, W}(\phi_f \otimes \phi'_f)$ by the natural surjection
$$
M_{B}(\pi_f, W)_\mathbb{R}^-(-1) \longrightarrow \mathrm{Ext}^1_{\mathrm{MHS}_\mathbb{R}^+}(\mathbb{R}(0),  M_{B}(\pi_f, W)_\mathbb{R})
$$
of Lem. \ref{suite_ex_ext1}. Hence, Lem. \ref{dualite} implies that
\begin{eqnarray*}
\mathcal{K}(\pi_f, W) &=& \frac{\langle \omega, [\rho] \rangle_B}{\langle \omega, \widetilde{v}_\mathcal{D} \rangle_B} \mathcal{D}(\pi_f, W).
\end{eqnarray*}
According to Prop. \ref{calcul-periode}, we have
$$
\langle \omega, \widetilde{v}_\mathcal{D} \rangle_B\sim (2\pi i)^2 c^-(\pi_f, W)^{-1}.
$$
With the notations of Cor. \ref{resume}, the pairing $\langle \omega, [\rho] \rangle_B$ is equal to
\begin{eqnarray*}
&& C_1 \frac{3}{160(p+1)} \sum_{j=0}^3(-1)^{k'+q+j}\binom{3}{j} A_{k, k', k'+q+j, j}\int \Xi_{k-q-2j+4, -k+k'+q, -q-2}(\phi_f, \phi'_f)\\
&-& C_2 \frac{(-1)^k}{8(q+1)}(B_{k, k', k-p}-C_{k, k', k-p+1}) \int \Xi_{k'+p+3, -p-2, -k+k'+p}(\phi_f, \phi'_f)\\
&+& C_3 \frac{3}{160 (q+1)} \sum_{j=0}^3 (-1)^{k'+j} \binom{3}{j} A_{k, k', k'+p+j, j} \int \overline{\Xi}_{k-p-2j+4, p+2, -k+k'+p}(\phi_f, \phi'_f)\\
&-& C_4 \frac{ (-1)^{k'+p+1}}{8(q+1)} (B_{k, k', k'+p}-C_{k, k', k'-p+1}) \int \overline{\Xi}_{k-p+3, -k+k'-q, q+2}(\phi_f, \phi'_f).
\end{eqnarray*}
According to Prop. \ref{comparaison}, if we choose $\phi_f$ and $\phi'_f$ properly, we have
\begin{eqnarray*}
\int \Xi_{n, r, s}(\phi_f, \phi'_f) &=& \int X_{(1, -1)}^{n} \Psi E^{\Phi}
\end{eqnarray*}
where the archimedean component of the Schwartz-Bruhat function $\Phi$ is defined by
$$
\Phi_\infty(x_1, y_1, x_2, y_2)= (ix_1+y_1)^\frac{p-r}{2} (ix_1-y_1)^\frac{p+r}{2}  (ix_2+y_2)^\frac{q-s}{2} (ix_2-y_2)^\frac{q+s}{2} e^{-\pi(x_1^2+y_1^2+x_2^2+y_2^2)}.
$$
This has weight $\lambda'(r, s)$ for the action of $\mrm{U}(1)^2$. Moreover, the archimedean component of $\Psi$ is chosen to be a lowest weight vector of the minimal $K$-type $\tau_{(k+3, -k'-1)}$ of $\pi_\infty^W$. Hence $X_{(1, -1)}^{n} \Psi$ is the $n$-th vector of a standard basis of $\tau_{(k+3, -k'-1)}$. The integrals above expand into Euler products of $v$-adic integrals (Cor. \ref{completed-euler-product}) and we have
$$
\int \Xi_{k-q-2j+4, -k+k'+q, -q-2}(\phi_f, \phi'_f) =0
$$
because the identities
$$
k-q-2j+4-k'-1+(-k+k'+q)=-2j+3 \neq 0,
$$
imply the vanishing of its archimedean factor (see the first statement of Prop. \ref{integrale-archi}). Similarly,
$$
\int \overline{\Xi}_{k-p-2j+4, p+2, -k+k'+p}(\phi_f, \phi'_f)=\int \overline{\Xi}_{k-p+3, -k+k'-q, q+2}(\phi_f, \phi'_f)=0.
$$
As a consequence,
\begin{eqnarray*}
\langle \omega, [\rho] \rangle_B &=& -C_2 \frac{(-1)^k}{8(q+1)}(B_{k, k', k-p}-C_{k, k', k-p+1}) \int \Xi_{k'+p+3, -p-2, -k+k'+p}(\phi_f, \phi'_f)\\
&\sim& C_2 \int (X_{(1, -1)}^{k'+p+3} \Psi) E^{\Phi}.
\end{eqnarray*}
For this integral, with the notations of Prop. \ref{integrale-archi}, we have $\lambda_1=k+3$, $\lambda_2=-k'-1$ and $t=n=k'+p+3$ hence  $t+\lambda_2+r=-t+\lambda_1+s = 0$, which means that we can apply the second result of Prop. \ref{integrale-archi}. Normalize the Bessel functional $W_\Psi=\prod'_v W_{\Psi_v}$ in such a way that for any place $p \notin V$, we have $W_{\Psi_p}(1)=1$, and that
$$
W_{\Psi_\infty}(1)= 2 \pi^{\frac{3(p+q)+6}{2}} \int_L \frac{\Gamma \left( c_1-s \right) \Gamma(c_2-s)\Gamma(c_3-s)\Gamma(c_4-s)}{\Gamma(a_1-s)\Gamma(a_2-s)} \pi^{2s} \frac{ds}{2 \pi i}
$$
where the notations are the ones of Prop. \ref{integrale-archi}. Combining the statements Cor. \ref{cor-technique}, Prop. \ref{integrale-nonram}, Prop. \ref{integrales-ram}, Prop. \ref{integrale-archi} and Prop. \ref{invariant-occulte}, we get
$$
\langle \omega, [\rho] \rangle_B
$$
$$
\sim \frac{ \Gamma \left( c_1+\frac{3(p+q)+6}{4} \right)\Gamma \left( c_2+\frac{3(p+q)+6}{4} \right) \Gamma \left( c_3+\frac{3(p+q)+6}{4} \right) \Gamma \left( c_4+\frac{3(p+q)+6}{4} \right)}{ \Gamma \left( a_1+\frac{3(p+q)+6}{4} \right) \Gamma \left( a_2+\frac{3(p+q)+6}{4} \right)}
$$
$$
\times a(\pi, \nu_1, \nu_2)L_V(p+q+7/2, \check{\pi}, r).
$$
We have
\begin{eqnarray*}
c_1+\frac{3(p+q)+6}{4} &=& k+\frac{k'}{2}+1,\\
c_2+\frac{3(p+q)+6}{4} &=& k+k'+\frac{4}{2},\\
c_3+\frac{3(p+q)+6}{4} &=& k-\frac{k'}{2}+1, \\
c_4+\frac{3(p+q)+6}{4} &=& k+k'+\frac{3}{2}.\\
\end{eqnarray*}
By our assumption $k' \equiv 0 \pmod 2$, the numbers $c_1, c_2$ and $c_3$ are integers and $c_4$ is a half-integer. By well known formulas for the positive integral and half-integral values of the gamma function, for $i=1, 2, 3$, we have
\begin{eqnarray*}
\Gamma \left( c_i+\frac{3(p+q)+6}{4} \right) &\sim& 1,\\
\Gamma \left( c_4+\frac{3(p+q)+6}{4} \right) &\sim& \pi^{1/2}.
\end{eqnarray*}
Similarly, using the assumption $k+1 \equiv k' \equiv 0 \pmod 2$, we show
\begin{eqnarray*}
\Gamma \left( a_1+\frac{3(p+q)+6}{4} \right) &\sim& 1\\
\Gamma \left( a_2+\frac{3(p+q)+6}{4} \right) &\sim& \pi^{1/2}.
\end{eqnarray*}
Then, we obtain
$$
 \frac{\langle \omega, [\rho] \rangle_B}{\langle \omega, \widetilde{v}_\mathcal{D} \rangle_B} \sim  \pi^{-2} a(\pi, \nu_1, \nu_2)c^-(\pi_f, W)L_S(k+k'-1/2, \check{\pi}).
$$
\end{proof}

In the case we are interested in, the existence of a split Bessel model follows from results of Moriyama and Takloo-Bighash.

\begin{pro} \label{valeur-centrale} Let $k, k', \pi, \nu_1, \nu_2$ be as above. If $k \neq 3, k' \neq 2$ and $\pi'$ is generic, then $\pi'$ has a split Bessel model associated to $(\nu_1, \nu_2)$.
\end{pro}

\begin{proof} With the notations of Def. \ref{bessel}, we see that
$$
W_\Psi(1)=\int_{\mbb{Q}^\times \backslash \mbb{A}^\times} \int_{(\mbb{Q} \backslash \mbb{A})^3} \Psi \left( \begin{pmatrix}
1 &  &  r & t \\
 & 1 &  t & s \\
 &  &  1 & \\
 &  &  & 1\\
\end{pmatrix} \begin{pmatrix}
y &  &  & \\
 & 1 &  & \\
 &  &  1 & \\
 &  &  & y\\
\end{pmatrix}\right) \nu_1(y)^{-1}\eta(t)^{-1} drdsdt d^\times y.
$$
This is the central value $\mathcal{Z}\left(1/2, \Psi \otimes \nu_1^{-1} \right)$ of the integral representation $\mathcal{Z}\left(s, \Psi \otimes \nu_1^{-1} \right)$ of the spinor $L$-function of $\Psi$ twisted by $\nu_1^{-1}$, as defined by Novodvorsky (see \cite{takloo-bighash}). Let us mention that Novodvorsky's expression for this integral is obtained from ours after an easy change of variables. To show that this central value is non-zero, it is enough to show that $\mathcal{Z}\left(1/2, \Psi \otimes \nu_1^{-1} \right)$ expands into an Euler product where each factor is non-zero. Note that the central character $\Psi \otimes \nu_1^{-1}$ has infinity type $\omega_\infty=k+2k'-3$ and that, for integers $k, k'$ satisfying our assumptions, we have $k+2k' \geq 7$. Hence, excluding the case $k=3$ and $k'=2$, we have the inequality $1/2>(5-\omega_\infty)/2$, and so it follows from \cite{moriyama-ajm} Prop. 4, that the above integral expands into an Euler product
$$
\mathcal{Z}\left( 1/2, \Psi \otimes \nu_1^{-1} \right)=\prod_v \mathcal{Z}_v\left( 1/2, \Psi \otimes \nu_1^{-1} \right)
$$
indexed by all places of $\mathbb{Q}$. The archimedean factor is a gamma factor, hence doesn't vanish. As a consequence, the computations of the local non-archimedean factors $\mathcal{Z}_v\left( 1/2, \Psi \otimes \nu_1^{-1} \right)$ performed in \cite{takloo-bighash} show that
$$
\mathcal{Z}\left( 1/2, \Psi \otimes \nu_1^{-1} \right) \neq 0.
$$
This shows that $\pi'$ has a split Bessel model associated to $(\nu_1, \nu_2)$.
\end{proof}

As we work with a stable $L$-packet, which always contains a generic member, Thm. \ref{principal} and Prop. \ref{valeur-centrale} directly imply the following result:

\begin{thm}  \label{final} Let $k \geq k' \geq 0$ be two integers. Let $W$ be an irreducible algebraic representation of $G$ of highest weight $\lambda(k, k', k+k'+4)$. Let $\pi=\pi_\infty \otimes \pi_f$ be a cuspidal automorphic representation of $G$ whose central character has infinity type $-k-k'-4$ and whose archimedean component $\pi_\infty$ is a discrete series of Harish-Chandra parameter $(k+2, k'+1)$. Let $\nu_1^0$ be a finite order Hecke character of sign $(-1)^{k-1}$ and let $\nu_1$ denote the Hecke character $|\,\,|^{1-k'}\nu_1^0$. Let $\nu_2^0$ be a finite order Hecke character of sign $(-1)^{k'-1}$ and let $\nu_2$ denote the Hecke character $|\,\,|^{1-k} \nu_2^0$. Let $V$ be the finite set of places where $\pi, \nu_1$ or $\nu_2$ is ramified, together with the archimedean place. Assume that:
\begin{itemize}
\item we have $k>k'>0$,
\item we have $k+1 \equiv k' \equiv 0 \pmod 2$,
\item we have $k \neq 3, k' \neq 2$,
\item the automorphic representation $\pi$ is stable at infinity.
\end{itemize}
Then,
$$
\mathcal{K}(\pi_f, W)= \pi^{-2}a(\pi, \nu_1, \nu_2) c^-(\pi_f, W)L_V(k+k'-1/2, \check{\pi})\mathcal{D}(\pi_f, W).
$$
\end{thm}

\begin{cor} \label{non-nullite} Let $n$ be an integer. Let $A \lra S$ be the universal abelian surface of infinite level over the Siegel threefold and let $A^n$ be the $n$-th fold fiber product over $S$. If $n$ is  odd and $n \geq 7$, then motivic cohomology space $H^{n+4}_\mathcal{M}(A^n, \mbb{Q}(n+2))$ is non-zero.
\end{cor}

\begin{proof} Let $k$ and $k'$ be two integers satisfying the assumptions of Thm. \ref{final} and such that $n=k+k'$. Let $t$ denote the integer $(k+k'+p+q+6)/2$. The target of the map $Eis_\mathcal{M}^{p, q, W}$ is the motivic cohomology group $H^4_\mathcal{M}(S, W)$, for $W$ as above, which is a subspace of  $H^{k+k'+4}_\mathcal{M}(A^{k+k'}, \mbb{Q}(t))$ according to Prop. \ref{motifdeW}. Assume that $p=k-1$ and $q=k'-1$. It follows from \cite{moriyama-ajm} Prop. 4 that the value $k+k'-1/2$ is in the absolute convergence region of the spinor $L$-function, hence that the special value $L_V(k+k'-1/2, \check{\pi}, r)$ is non-zero. As a consequence, Thm \ref{final} implies that the vector space $\mathcal{K}(\pi_f, W)$ is non-zero. Via Beilinson's regulator, the image of $Eis_\mathcal{M}^{p, q, W}$ surjects on $\mathcal{K}(\pi_f, W)$ for any $\pi_f$ which is the non-archimedean part of a cuspidal automorphic representation satisfying the conditions of Thm. \ref{final}. Hence, the motivic cohomology group $H^{n+4}_\mathcal{M}(A^n, \mbb{Q}(n+2))$ is non-zero.
\end{proof}

\begin{acknowledgements}
This article has been written while I benefited of a funding of the ANR-13-BS01-0012
FERPLAY. I would like to thank Giuseppe Ancona, Jean-Yves Charbonnel, Michael Harris, David Loeffler, Alexander Molev, Brooks Roberts and  J\"{o}rg Wildeshaus for answering my questions and for useful remarks. Finally, it is a pleasure to thank Pascal Molin for writing a program thanks to which I could perform numerical computations related to this work and the anonymous referee for her/his careful reading which led to many improvements in the presentation of this article.
\end{acknowledgements}

\end{document}